\documentclass{amsart}
\usepackage[top=50pt,bottom=57pt,left=76pt,right=76pt]{geometry}

\usepackage{amsfonts, mathrsfs, enumerate, amsmath, amsthm, amssymb,
thm-restate, pifont, mathrsfs}
\usepackage{textcomp}

\usepackage{tikzit}
\usetikzlibrary{shapes.multipart}

\tikzstyle{new edge style 2}=[-, thick]

\usepackage[colorlinks]{hyperref}
\hypersetup{linkcolor=blue, urlcolor=blue, citecolor=red}

\usepackage[capitalise]{cleveref}

\newcommand{\cmark}{\ding{51}}%

\newtheorem{theorem}{Theorem}[section]
\newtheorem{cor}[theorem]{Corollary}
\newtheorem{prop}[theorem]{Proposition}
\newtheorem{lem}[theorem]{Lemma}
\newtheorem{question}[theorem]{Question}
\newtheorem{example}[theorem]{Example}
\newenvironment{ex}{\begin{example} \rm}{\end{example}}

\newcommand{\N}{\mathbb{N}}
\newcommand{\R}{\mathbb{R}}
\newcommand{\T}{\mathcal{T}}
\renewcommand{\S}{\mathcal{S}}
\newcommand{\set}[2]{\{ #1:#2 \} }  
\newcommand{\genset}[1]{\langle #1 \rangle }  

\newcommand{\Inj}{\operatorname{Inj}}  
\newcommand{\Surj}{\operatorname{Surj}}  
\newcommand{\Homeo}{\operatorname{Homeo}}  
\newcommand{\Sym}{\operatorname{Sym}}  

\newcommand{\im}{\operatorname{im}}
\newcommand{\id}{\operatorname{id}}
\renewcommand{\P}{\mathcal{P}}
\newcommand{\Q}{\mathbb{Q}}
\newcommand{\Z}{\mathbb{Z}}
\newcommand{\B}{\mathscr{B}}

\renewcommand{\to}{\longrightarrow}
\newcommand{\pentagram}{\bullet}

\newcommand{\dom}{\operatorname{dom}}
\newcommand{\na}{\mathbin{\blacklozenge}}
\newcommand{\ar}{\operatorname{arity}}

\renewcommand{\P}{\mathcal{P}}

\newcommand{\mft}[2]{V_{#1, #2}}

\newcommand{\Aut}{\operatorname{Aut}}
\newcommand{\End}{\operatorname{End}}

\newcommand{\Poly}{\operatorname{Pol}}
\newcommand{\Hom}{\operatorname{Hom}}

\newcommand{\x}{\mathbf{x}}
\newcommand{\y}{\mathbf{y}}
\newcommand{\z}{\mathbf{z}}

\title[Automatic continuity, unique Polish topologies, and Zariski
topologies]{Automatic continuity, unique Polish topologies, and Zariski
topologies on monoids and clones}
\author{L. Elliott, J. Jonu\v{s}as, Z. Mesyan, J. D. Mitchell, M. Morayne,
and Y. P\'eresse}
\thanks{L. Elliott would like to acknowledge the support of Mathematics and
Statistics at the University of St Andrews for supporting their Ph.D. studies.
J. Jonu\v{s}as received funding from the Austrian Science Fund (FWF) through
Lise Meitner grant No M 2555.}

\begin{document}
\maketitle
\begin{abstract}
  In this paper we explore  the extent to which the algebraic structure of a
  monoid $M$ determines the topologies on $M$ that are compatible with its
  multiplication. Specifically we study the notions of automatic continuity;
  minimal Hausdorff or $T_1$ topologies; Polish semigroup topologies; and we
  formulate a notion of the Zariski topology for monoids and inverse monoids.

  If $M$ is a topological monoid such that every homomorphism from $M$ to a
  second countable topological monoid $N$ is continuous, then we say that $M$
  has \emph{automatic continuity}.  We show that many well-known, and
  extensively studied, monoids have automatic continuity with respect to a
natural semigroup topology,  namely: the full transformation monoid
  $\mathbb{N} ^ \mathbb{N}$; the full binary relation monoid $B_{\mathbb{N}}$;
  the partial transformation monoid $P_{\mathbb{N}}$; the symmetric inverse
  monoid $I_{\mathbb{N}}$; the monoid $\operatorname{Inj}(\mathbb{N})$
  consisting of the injective  transformations of $\mathbb{N}$; and the monoid
  $C(2^{\mathbb{N}})$ of continuous functions on the Cantor set
  $2^{\mathbb{N}}$.

  The monoid $\mathbb{N} ^ \mathbb{N}$ can be equipped with the product
  topology, where the natural numbers $\mathbb{N}$ have the discrete topology;
  this topology is referred to as the \textit{pointwise topology}.  We show
  that
  the pointwise topology on $\mathbb{N} ^ \mathbb{N}$, and its analogue on
  $P_{\mathbb{N}}$, are the unique Polish semigroup topologies on these
  monoids.
  The compact-open topology is the unique Polish semigroup topology on $C(2 ^
    \mathbb{N})$, and on the monoid $C([0, 1] ^ \mathbb{N})$ of continuous
  functions on the Hilbert cube $[0, 1] ^ \mathbb{N}$.	The symmetric inverse
  monoid $I_{\mathbb{N}}$ has at least 3 Polish semigroup topologies, but a
  unique Polish inverse semigroup topology.  The full binary relation monoid
  $B_{\mathbb{N}}$ has no Polish semigroup topologies, nor do the partition
  monoids. At the other extreme, $\operatorname{Inj}(\mathbb{N})$ and the
  monoid
  $\operatorname{Surj}(\mathbb{N})$ of all surjective transformations of
  $\mathbb{N}$ each have infinitely many distinct Polish semigroup topologies.

  We prove that the Zariski topologies on $\mathbb{N} ^ \mathbb{N}$,
  $P_{\mathbb{N}}$, and $\operatorname{Inj}(\mathbb{N})$ coincide with the
  pointwise topology; and we characterise the Zariski topology on
  $B_{\mathbb{N}}$.

  Along the way we provide many additional results relating to the Markov topology, the small
  index property for monoids, and topological embeddings of semigroups in
  $\mathbb{N}^{\mathbb{N}}$ and inverse monoids in $I_{\mathbb{N}}$.
        
        Finally, the techniques developed in this paper to prove the results
        about monoids, are applied to function clones. In particular, we show
        that: the full function clone has a unique Polish topology; the Horn
        clone, the polymorphism clones of the Cantor set and the countably
        infinite atomless Boolean algebra all have automatic continuity with
        respect to second countable function clone topologies.  
\end{abstract}

\tableofcontents

\section{Introduction}
\subsection{Summary of main theorems and background}

A \textit{semigroup} is a set together with an associative binary operation,
and a \textit{monoid} is a semigroup with an identity.	Broadly speaking, the
topic of this paper is to explore the extent to which the algebraic structure
of a semigroup, or monoid, $S$ determines the topologies on $S$ that are
compatible with its multiplication.  If $S$ is a semigroup and $\T$ is a
topology on the set $S$, then $\T$ is called a \textit{semigroup topology} if
the multiplication function from $S\times S$, with the product topology, to $S$
is continuous.	A semigroup together with a semigroup topology is referred to
as a \textit{topological semigroup}.  For many semigroups and monoids, the
problem of determining the admissible semigroup topologies seems to be
difficult.  In this article, we consider the following well-known, and
extensively studied, infinite monoids: the full transformation monoid $X ^ X$;
the full binary relation monoid $B_{X}$; the partial transformation monoid
$P_{X}$; the symmetric inverse monoid $I_{X}$; $\Inj(X)$  consisting of the
injective transformations of $X$; $\Surj(X)$ consisting of the surjective
transformations of $X$; and the monoids  $C(2^{\N})$ and $C([0, 1] ^ \N)$ of
continuous functions on the Cantor set $2^{\mathbb{N}}$ and the Hilbert cube
$[0, 1] ^ \N$, respectively.   The general aim being to classify the minimal and
maximal semigroup topologies, under some natural additional assumptions, such
as being $T_1$, second countable, or Polish, for example. 
Along the way we
will encounter a number of natural topologies that can be defined for arbitrary
semigroups (the Fr\'echet-Markov, Hausdorff-Markov, and Zariski topologies),
and discuss there interrelations (\cref{section-markov}); we develop
some general machinery (property \textbf{X}, \cref{section-property-x});
study the related notions of automatic continuity and the small index property
for semigroups (\cref{subsection-auto-cont}); and characterise when a
topological semigroup topologically embeds in $\N ^ \N$
(\cref{subsection-full-transf}) and when an inverse monoid topologically
embeds in $I_{\N}$  (\cref{subsection-sym-inv}). 
Finally, we foray slightly into the land of clones, defined in \cref{section-clones}, and prove analogues of several of the main results in the previous sections. 

This paper arose out of
a previous version \cite{mesyan2018topological} which considered topologies on
the full transformation monoid $X^X$.

The main theorems in this paper are summarised in~\cref{table-the-only} and~\cref{table-the-not-so-only};
the relevant definitions are given in the later sections of the paper.
Tables~\ref{table-the-only} and~\ref{table-the-not-so-only} are not exhaustive, in particular, many of the
results in the paper apply to sets of arbitrary cardinality, rather than just
$\N$. 

\begin{table}
  \begin{tabular}{ll|c|c|c|c|c|c|c|l}
                                          &                                      & AC
                                          & \# PST                               & Min    & Z      & F-M    & H-M    &
    \# PIST                                                                                                            \\ \hline
    full transformations                  & $\N ^ \N$                            & \cmark
                                          & 1                                    & \cmark & \cmark & \cmark &
    \cmark                                & -                                    &
    \textsection\ref{subsection-full-transf}
    \\
    binary relations                      & $B_{\N}$                             & \cmark
                                          & 0                                    & \cmark & \cmark & \cmark & \cmark &
    -                                     &
    \textsection\ref{subsection-binary}

    \\
    partial transformations               & $P_{\N}$                             & \cmark
                                          & 1                                    & \cmark & \cmark & \cmark & \cmark &
    -                                     &
    \textsection\ref{subsection-partial}

    \\
    partial bijections                    & $I_{\N}$                             & \cmark
                                          & $\geq 3$                             & \cmark & ?      & \cmark & \cmark
                                          & 1                                    &
    \textsection\ref{subsection-sym-inv}
    \\
    injective transformations             & $\Inj(\N)$                           & \cmark
                                          & $\geq\aleph_0$                       & ?      & \cmark & ?      & \cmark
                                          & -                                    &
    \textsection\ref{subsection-inj-surj}
    \\
    surjective transformations            & $\Surj(\N)$                          & ?
                                          & $\geq\aleph_0$                       & ?      & ?      & ?      & ?
                                          & -                                    &
    \textsection\ref{subsection-inj-surj}
    \\
    continuous functions on $[0, 1] ^ \N$ & $C([0, 1] ^ \N)$                     & ?
                                          & 1                                    & ?      & ?      & ?      & ?      &
    -                                     & \textsection\ref{subsection-hilbert}                                       \\
    continuous functions on $2 ^ \N$      & $C(2 ^ \N)$                          & \cmark
                                          & 1                                    & ?      & ?      & ?      & ?      &
    -                                     & \textsection\ref{subsection-cantor}                                        \\
  \end{tabular}
  \medskip

  \begin{tabular}{rcl}
    AC      & = & automatic continuity with respect to a second countable
    topology on the                                                       \\
            &   & semigroup and the class of second countable topological
    semigroups                                                            \\
    \# PST  & = & number of Polish semigroup topologies
    \\
    Min     & = & the minimum $T_1$ semitopological topology is known
    \\
    Z       & = & the semigroup Zariski topology is known
    \\
    F-M     & = & Fr\'echet-Markov topology is known
    \\
    H-M     & = & Hausdorff-Markov topology is known
    \\
    \# PIST & = & number of Polish inverse semigroup topologies
  \end{tabular}
  \medskip

  \caption[A summary of the main results about monoids in this paper]{A
    summary of the main results in this paper, and where they can
    be found in the document.}
  \label{table-the-only}
\end{table}

\begin{table}
  \begin{tabular}{ll|c|c|l}
                    &   & AC     & \# PCT          & \\ \hline
    full function clone& $\mathscr{O}_{\N}$ & \cmark & 1               &
    \textsection\ref{subsection-full-clone}                    \\
    Horn clone& $\mathscr{H}_{\N}$ & \cmark & $\geq \aleph_0$ &
    \textsection\ref{subsection-horn-clone}                    \\
    polymorphisms on \(2^\N\) &$\Poly(2^\N)$      & \cmark & $1$             &
    \textsection\ref{subsection-poly-clone}                    \\
    polymorphisms on \(B_\infty\) &$\Poly(B_\infty)$  & \cmark & $ 1$            &
    \textsection\ref{subsection-poly-clone}                    \\
  \end{tabular}
  \begin{tabular}{rcl}
    AC     & = & automatic continuity with respect to the class of second countable topological function clones\\
    \# PCT & = & number of Polish clone topologies                                             \\
  \end{tabular}
  \caption[A summary of the main results for clones in this paper]{A summary of
    the main results for clones in this paper, and where they can be
    found in the document.}
  \label{table-the-not-so-only}
\end{table}

The questions that are the focus of this paper originally arose in the context
of groups; we will briefly, and probably incompletely, discuss the history of
these questions for groups, and then for semigroups.  A topological semigroup
that happens to be a group is called a \textit{paratopological group},~i.e.\
multiplication is continuous but inversion need not be.  If $G$ is a
paratopological group and inversion is also continuous, then the topology is a
\textit{group topology}, and the group $G$ is a \textit{topological group}.

The problem of determining what topologies are compatible with the
multiplication and inversion in a group has an extensive history that can be
traced back to Cauchy and Markov.  Markov~\cite{Markov1950aa} asked whether
there exists an infinite group whose only group topologies are the trivial and
discrete topologies; such a group is called \textit{non-topologizable}.
Shelah~\cite{Shelah1980aa} showed that a non-topologizable group exists
assuming the continuum hypothesis; Hesse~\cite{Hesse1979aa} showed that the
assumption of the continuum hypothesis in Shelah's construction can be avoided.
Olshanskii~\cite{Olshanskii1980aa} showed that an infinite family of the Adian
groups (constructed by Adian~\cite{Adian1975aa} as a counter-example to the
Burnside problem) are non-topologizable. A more recent paper on this topic
is~\cite{Klyachko2013aa}.

Given that there exist groups where the only group topologies are trivial and
discrete, it is natural to ask if there are groups that admit a unique
non-trivial non-discrete group topology. The only results in this direction,
that we are aware of, require additional assumptions on the topology.

A topological space $X$ is \textit{Polish} if it is completely
metrizable and separable. A \textit{Polish semigroup} is just a topological
semigroup where the topology is Polish.
\textit{Polish groups} are defined analogously to Polish semigroups.
The fundamental results of R. M. Solovay~\cite{Solovay1970aa} and S.
Shelah~\cite{Shelah1984aa} show that it is consistent with ZF without choice
that any Polish group has a unique Polish group topology. The same is not true
in ZFC: the additive group of real numbers $\R$ is a Polish group with the
usual
topology on $\R$, as too is the additive group $\R ^ 2$. The two groups $\R$
and $\R ^ 2$ are isomorphic, since they are vector spaces of equal dimension
over the rationals $\Q$, but are not homeomorphic, since $\R ^ 2$ with any
point removed is connected and $\R$ with any point removed is not.

The Baire space $\N ^ \N$ is equipped with the product topology, where the
natural numbers $\N = \{0, 1, \ldots\}$ have the discrete topology.  This
topology on $\N ^ \N$, and the subspace topology that it induces on any subset
of $\N ^ \N$, will be referred to as the \textit{pointwise topology}
throughout this paper. The pointwise topology on $\N ^ \N$ is Polish. In
addition to being a topological space, $\N ^ \N$ is a monoid under composition
of functions; called the \textit{full transformation monoid}. Indeed, $\N ^
  \N$ is a topological semigroup with respect to the pointwise topology; see
\cref{subsection-full-transf}.

The analogue, in the context of groups, of the full transformation monoid $\N ^
  \N$ is the symmetric group $\Sym(\N)$.  Since $\Sym(\N)$ is a $G_{\delta}$
subset of $\N ^ \N$, the pointwise topology on $\Sym(\N)$ is also Polish (see
\cite[Theorem 3.11]{Kechris1995aa}). Furthermore, $\Sym(\N)$ with the pointwise
topology is a topological group.  In problem number 96 of the famous Scottish
Book~\cite{Mauldin2015aa}, Ulam asked if the symmetric group $\Sym(\N)$ on the
natural numbers has a locally compact Polish group topology.  Ulam's problem
was answered in the negative by Gaughan~\cite{Gaughan1967aa}, who also showed
that every $T_1$ group topology on $\Sym(\N)$ contains the pointwise topology.
It can be shown that if $\T_1$ and $\T_2$ are Polish group topologies on the
same group and $\T_1\subseteq \T_2$, then $\T_1 = \T_2$. It therefore follows
by Gaughan's result in~\cite{Gaughan1967aa} that $\Sym(\N)$ has a unique Polish
group topology. This result was strengthened in~\cite[Theorem
  6.26]{Rosendal2007ac}, where it was shown that the pointwise topology is the
only non-trivial separable group topology on $\Sym(\N)$.

Many further examples of groups are known to have unique Polish group
topologies: the groups of isometries of the Urysohn space and of the Urysohn
sphere~\cite{Sabok2019aa}; homeomorphism groups of a wide class of metric
spaces such as any separable metric manifold~\cite{Kallman1986aa} (including
the Hilbert cube $[0,1]^{\N}$) or the Cantor set $2^{\N}$); the automorphism
groups of many countable relational structures, in particular, Fra\"iss\'e
limits, such as the rational numbers $\Q$ under the usual linear order, the
countable random graph $R$~\cite{Hodges1993ab, Hrushovski1992aa}, and the group
of Lipschitz homeomorphisms of the Baire space~\cite{Kechris2007aa}.  Some
further references on the uniqueness of Polish group topologies
are~\cite{Chang2017aa, Gartside2008ab}.  On the other hand, many groups have
been shown to have no non-discrete Polish group topologies, for example: free
groups~\cite{Dudley1961aa};  the homeomorphism groups of the irrational and of
the rational numbers and the group of Borel automorphisms of
$\R$~\cite{Rosendal2005aa}. Additional references include~\cite{Cohen2016aa,
  Kallman1976aa, Kallman1979aa, Kallman1984aa, Kallman1984ab, Kallman2010aa,
  Rosendal2007ac}.

A topological group $G$ is said to have \emph{automatic continuity} with
respect to a class of topological groups, if every homomorphism from $G$ to a
member of that class is continuous.  Several of the topological groups
mentioned in the previous paragraph have automatic continuity with respect to
a natural class of topological groups, such as, for example, the separable
topological groups.  The group $\Aut([0, 1], \lambda)$ was shown to have
automatic continuity by Yaacov, Berenstein, and Melleray~\cite{Yaacov2010aa},
as was the infinite-dimensional unitary or orthogonal group to a separable
group; see Tsankov~\cite{Tsankov2013aa}; further references
are~\cite{Barbina2007,Bodirsky2014,Evans1990,Herwig1998aa,Lascar1991,Mann2016aa,Paolini2019,Paolini2020,  Rubin1994,Truss1989}.  If $G$ is a Polish group that has automatic continuity
with respect the class of Polish groups, then $G$ has a unique Polish group
topology.

Among many other interesting results in~\cite{Kechris2007aa}, it is shown that
a Polish group $G$ has automatic continuity with respect to the class of
separable groups whenever there is a comeagre orbit of the action by
conjugation of $G$ on $G ^ n$ for every $n\in \N$; such a group $G$ is said to
have \textit{ample generics}. Many permutation groups have ample generics: the
symmetric group $\Sym(\N)$; the automorphisms of the random
graph~\cite{Hodges1993ab, Hrushovski1992aa}; the free group on countably many
generators~\cite{Bryant1997aa}; some further references are~\cite{Herwig2000aa,
  Schmerl2005aa, Solecki2005aa}. Some groups do not have ample generics, for
instance: $\Aut(\Q, \leq)$~\cite{Kuske2001aa}, and the homeomorphism group
$\Homeo(\R)$ of the reals $\R$.

Another notion that is implied by automatic continuity for a group (with
respect to the class of separable groups) is the so-called small index
property.  A topological group $G$ has the \textit{small index property} if
every subgroup of index at most $\aleph_0$ is open; some authors assume this
condition for subgroups of index less than $2 ^ {\aleph_0}$. Small index
property does not necessarily imply automatic continuity; see
\cref{example-small-index-reals}.  A topological group $G$ has the small index
property if and only if every homomorphism from $G$ to $\Sym(\N)$ is
continuous.  Many groups were shown to have the small index property before
they were shown to have automatic continuity: $\Sym(\N)$~\cite{Dixon1986aa,
  Rabinovich1977aa, Semmes1981aa}, the automorphism groups of countable vector
spaces over finite fields \cite{Evans1986aa}; $\Aut(\Q, <)$ and the
automorphism group of the countably infinite atomless Boolean
algebra~\cite{Truss1989aa}; the automorphism group of the random graph and all
automorphism groups of $\omega$-categorical $\omega$-stable
structures~\cite{Hodges1993ab}; the automorphism groups of the Henson
graphs~\cite{Herwig1998aa}.

Analogously to the definition for groups, we will say that a topological
semigroup $S$ has \emph{automatic continuity} with respect to a class of
topological semigroups, if every homomorphism from $S$ to a member of that
class is continuous.  A related notion is that of \textit{automatic
  homeomorphicity}: a topological semigroup $S$ has automatic homeomorphicity
with respect to a class $\mathcal{C}$ of topological semigroups if every
isomorphism from $S$ to a member of $\mathcal{C}$ is necessarily a
homeomorphism.	In the literature, such as \cite{Behrisch2017aa,
  Bodirsky2018aa, Bodirsky2017aa, Pech2016aa, Pech2017aa, Pech:2018aa}, the
terms
automatic continuity and automatic homeomorphicity are  exclusively used for the case where $S$ is a
submonoid of the full transformation monoid $\N ^ \N$ with the pointwise
topology, and the class $\mathcal{C}$ consists of submonoids of $\N ^ \N$ with
the pointwise topology.

Examples of monoids with automatic homeomorphicity with respect to the class of
closed submonoids of $\N ^ {\N}$ include the full transformation monoid $\N ^
  \N$ and the monoids of: order-endomorphisms or order-embeddings of the
rational
numbers $\Q$; injective functions $\Inj(\N)$ on $\N$; self-embeddings of the
countable random graph (the Rado graph); endomorphisms of the countable random
graph; self-embeddings of the countable universal homogeneous digraph
(see~\cite{Behrisch2017aa, Bodirsky2017aa}).  Many of the results
in~\cite{Behrisch2017aa, Bodirsky2017aa} are extended to the corresponding
clones of polymorphisms; see also~\cite{Bodirsky2018aa, Pech2016aa, Pech2017aa,
  Pech:2018aa}.

\subsection{Preliminaries}
We aim to give the relevant definitions within the sections where they are used
whenever possible. In this section, we collect some technical definitions that
are required throughout the paper.

In this paper, functions are written to the right of their arguments,
and composed from left to right. If $S$ is a semigroup, then we use $S^1$ to
denote the monoid obtained by adjoining an identity $1$ to $S$. A set $X$ is
\textit{countable} if it is finite or has cardinality $\aleph_0$.

Suppose that $S$ is a semigroup and that $\T$ is a topology on the set $S$.
Then we
will say that $\T$ is \textit{left semitopological} for $S$ if for every $s\in
  S$ the function $\lambda_s: S \to S$ defined by $(x)\lambda_s = sx$ is
continuous. If every function $\rho_s: S\to S$ defined by $(x)\rho_s = xs$ is
continuous, then we say that $\T$ is \textit{right semitopological} for $S$. If
$\T$ is left and right semitopological for $S$, then we say that $\T$ is
\textit{semitopological} for $S$. Every semigroup topology is semitopological.
We refer to a semigroup with a left semitopological topology as a \textit{left
  semitopological semigroup}.  Analogous definitions can be made for
\textit{right semitopological} and \textit{semitopological}.

If $\T_1$ and $\T_2$ are topologies on a set $X$ (or indeed any collections of
subsets of $X$), then the least topology on $X$ containing $\T_1$ and $\T_2$
will be referred to as \textit{the topology generated by} $\T_1$ and $\T_2$. If
$X$ and $Y$ are topological spaces and $\B$ is a subbasis for $Y$, then $f:
  X\to Y$ is continuous if and only if $(B)f ^ {-1}$ is open for all $B\in \B$.
Hence if $\T_1$ and $\T_2$ are topological for $S$, then so too is the topology
generated by $\T_1$ and $\T_2$. The analogous statement holds if
``topological'' is replaced by ``semitopological'', ``right semitopological'',
or ``left semitopological''.

An \textit{inverse semigroup} is a semigroup $S$ such that for every $x\in S$
there exists a unique $y\in S$ such that $xyx = x$ and $yxy = y$; $y$ is
usually denoted by $x ^ {-1}$.	\textit{Inverse semigroup topologies} and
\textit{topological inverse semigroups} are defined analogously to group
topologies and topological groups. In other words, a topological inverse
semigroup is a topological semigroup with continuous inversion.  A
semitopological group or inverse semigroup is a semitopological semigroup that
happens to be a group or inverse semigroup, i.e.\ inversion is not assumed to
be continuous.

\section{Fr\'echet-Markov, Hausdorff-Markov, and Zariski topologies}
\label{section-markov}

In this section, we introduce, and prove some results about, three topologies
that can be defined for any semigroup, that, in some sense, arise from the
algebraic structure of that semigroup.

\subsection{Definitions}
The \textit{Fr\'echet-Markov topology} of a semigroup $S$ is the intersection
of all $T_1$ semigroup topologies on $S$. Similarly, the
\textit{Hausdorff-Markov topology} of a semigroup $S$ is the intersection of
all Hausdorff semigroup topologies for $S$.  The \textit{inverse
  Fr\'echet-Markov topology} and \textit{inverse Hausdorff-Markov topologies}
of
an inverse semigroup $S$ are similarly defined to be the intersections of all
$T_1$ and Hausdorff inverse semigroup topologies on $S$ respectively.  Clearly,
the Fr\'echet-Markov topology on a semigroup is contained in the
Hausdorff-Markov topology.  The intersection of $T_1$ topologies is $T_1$ and
the intersection of semigroup topologies is semitopological. Hence the
Fr\'echet-Markov and Hausdorff-Markov topologies of a semigroup $S$ are $T_1$
and $S$ is semitopological with respect to both.

It is well-known that if $G$ is a topological group, then every $T_0$ group
topology is also $T_{3\frac{1}{2}}$ and so, in particular, there is no
distinction in the theory of topological groups between the notions of inverse
Fr\'echet-Markov and inverse Hausdorff-Markov topologies.  In
\cref{ex-hausdorff} we show that these notions are distinct in the context of
topological semigroups and topological inverse semigroups.

By definition, the least $T_1$ topology that is semitopological for $S$ is
contained in the Fr\'echet-Markov topology of $S$.  These two topologies may
coincide, for example, if $S$ is a semigroup of right zeros, then every
topology on $S$ is a semigroup topology. Hence the least $T_1$ topology on a
semigroup $S$ of right zeros is the cofinite topology, and the Hausdorff-Markov
topology is the intersection of all the Hausdorff topologies on the set $S$,
and it is straightforward to verify that this is also the cofinite topology.
We will show that for  several well-known examples of semigroups the
Fr\'echet-Markov and  Hausdorff-Markov topologies coincide. In particular, if
$S$ is a semigroup where the least $T_1$ topology $\T$ that is semitopological
for $S$ happens to be a Hausdorff semigroup topology, then $\T$ is both the
Fr\'echet-Markov and Hausdorff-Markov topology for $S$.

Recall that the \textit{group Zariski} topology on a group $G$ is defined as
the
topology with subbasis consisting of the sets
\[
  \set{g\in G}{(g)\phi \not= 1_G}
\]
where $1_G$ is the identity of $G$ and $\phi: G\to G$ satisfies $(g)\phi = h_0
  g ^ {i_0} h_1 g ^ {i_1} \cdots h_{k-1} g ^ {i_{k-1}}$ for every $g\in G$ and
for some fixed $h_0, \ldots, h_{k-1}\in G$ and $i_0, \ldots, i_{k-1}\in \{-1,
  1\}$.  The group Zariski topology has been extensively studied in the
literature of topological groups; see, for example, \cite{bryant1977verbal,
  DIKRANJAN20101125, Markov1950aa}.

Analogously to the definition for groups, we define the \textit{semigroup
  Zariski} topology on a semigroup $S$ as the topology with subbasis consisting
of
\[
  \set{s\in S}{(s)\phi_0 \not= (s)\phi_1}
\]
where $\phi_0, \phi_1: S\to S$ are any functions such that $(s)\phi_0 = t_0 s
  t_1 s \cdots t_{k-1} s $, $k\geq 1$ for every $s\in S$ and for some fixed
$t_0,
  \ldots, t_{k-1}\in S ^ 1$, and $\phi_1$ is defined analogously for some fixed
$u_0, \ldots, u_{l-1} \in S ^ 1$. Similar to the notion for groups used in the
literature,  the complement of a set in the basis for the Zariski topology:
\[
  \set{s\in S}{(s)\phi_0 = (s)\phi_1}
\]
will be referred to as an \textit{elementary algebraic} set.

The \textit{inverse Zariski} topology for an inverse semigroup is defined
analogously, where the functions are of the form $(s)\phi = t_0  s ^ {i_0} t_1
  s ^ {i_1}\cdots t_{k-1}  s^ {i_{k-1}}$, where $i_0, \ldots, i_{k-1}\in \{-1,
  1\}$.  The inverse Zariski topology on a group coincides with the usual
notion
of the Zariski topology on a group; the inverse semigroup and semigroup Zariski
topologies on an inverse semigroup do not always coincide, see
\cref{cor-zariski-inverse}; the dual question comparing the group and
semigroup Zariski topology on group is open, see~\cref{question-zariski-2}.

We are primarily concerned with the semigroup Zariski topology in this paper,
which we will often refer to as the Zariski topology, if no ambiguity will
arise (such as when the semigroup under consideration is not inverse).

\subsection{Relations between the topologies}

In this section we prove some results about the Zariski, Fr\'echet-Markov, and
Hausdorff-Markov topologies, and their relationships to each other, for an
arbitrary semigroup.

\begin{prop}
  \label{prop-luke-1}
  The Zariski topology on any  semigroup $S$ is contained in every Hausdorff
  semigroup topology on $S$. Similarly, the inverse Zariski topology on any
  inverse semigroup $S$ is contained in every inverse Hausdorff semigroup
  topology on $S$.
\end{prop}
\begin{proof}
  We prove the result for semigroups without inversion, the inverse semigroup
  proof is dual.  Recall that if $X$ is a Hausdorff topological space and $f,
    g: X\to X$ are continuous functions, then the set
  \[
    \set{x\in X}{(x)f = (x)g}
  \]
  is closed in $X$. Let $\T$ denote a Hausdorff semigroup topology on $S$.  Let
  $\phi_0:S \to S$ be from the definition of a subbasic open set for the
  Zariski topology. In other words, $(s)\phi_0 = t_0 s t_1 s \cdots t_{k-1} s
  $, $k\geq 1$ for every $s\in S$ and for some fixed $t_0, \ldots, t_{k-1}\in S
    ^ 1$.  If $m = 2k$, then the function $\psi_0: S\to S ^ m$ defined by
  \[
    (s)\psi_0 = (t_0, s, t_1, s, t_2, \ldots, t_{k-1},s )
  \]
  is continuous in every coordinate, and hence is continuous with respect to
  $\T$. The function $\phi_0$ is then the composite of $\psi_0$ and the
  multiplication function from $S ^ m$ to $S$, and is hence continuous with
  respect to $\T$. The function $\phi_1$ is continuous by an analogous
  argument.  It follows that every subbasic open set $\set{s\in S}{(s)\phi_0
      \not= (s)\phi_1}$ for the Zariski topology is open in $\T$.
\end{proof}

\begin{prop}\label{prop-luke-2}
  The Zariski topology is semitopological on any semigroup.
\end{prop}
\begin{proof}
  Let $S$ be any semigroup, and let $y\in S$ be arbitrary. We will show that
  $\lambda_y : S\to S$ defined by $(x)\lambda_y = yx$ is continuous.  Suppose
  that $\phi_0, \phi_1:S\to S$ are defined by
  \begin{eqnarray*}
    (s)\phi_0 &= &t_0 s t_1 s \cdots t_{k-1} s	\\
    (s)\phi_1 &= &u_0 s  u_1 s	\cdots u_{l-1} s
  \end{eqnarray*}
  for some  $k, l\geq 1$, for every $s\in S$, and for some fixed $t_0, \ldots,
    t_{k-1}, u_0, \ldots, u_{l-1}\in S ^ 1$. Then
  \begin{eqnarray*}
    \set{s\in S}{(s)\phi_0 \not=(s)\phi_1}\lambda_y ^ {-1}
    &  = & \set{x\in S}{(yx)\phi_0 \not=(yx)\phi_1} \\
    & = & \set{x \in S}{t_0 (yx) \cdots t_{k-1} (yx)
      \not= u_0 (yx)  \cdots u_{l-1} (yx)
    }.
  \end{eqnarray*}
  If
  \begin{eqnarray*}
    (x)\phi_0' &= &t_0 (yx) t_1 (yx) \cdots t_{k-1} (yx) \\
    (x)\phi_1' &= & u_0 (yx) u_1 (yx) \cdots u_{l-1} (yx),
  \end{eqnarray*}
  then it is clear that $\set{x\in S}{(x)\phi_0' \not=(x)\phi_1'}$ is open and
  so $S$ is left semitopological with respect to the Zariski topology.

  The proof that $S$ is right semitopological with respect to the Zariski
  topology is dual.
\end{proof}

See \cref{figure-the-only} for the Hass\'e diagram of the containment of
the Hausdorff-Markov, Fr\'echet-Markov, Zariski, and minimal $T_1$ topology
that is semitopological for a given semigroup.

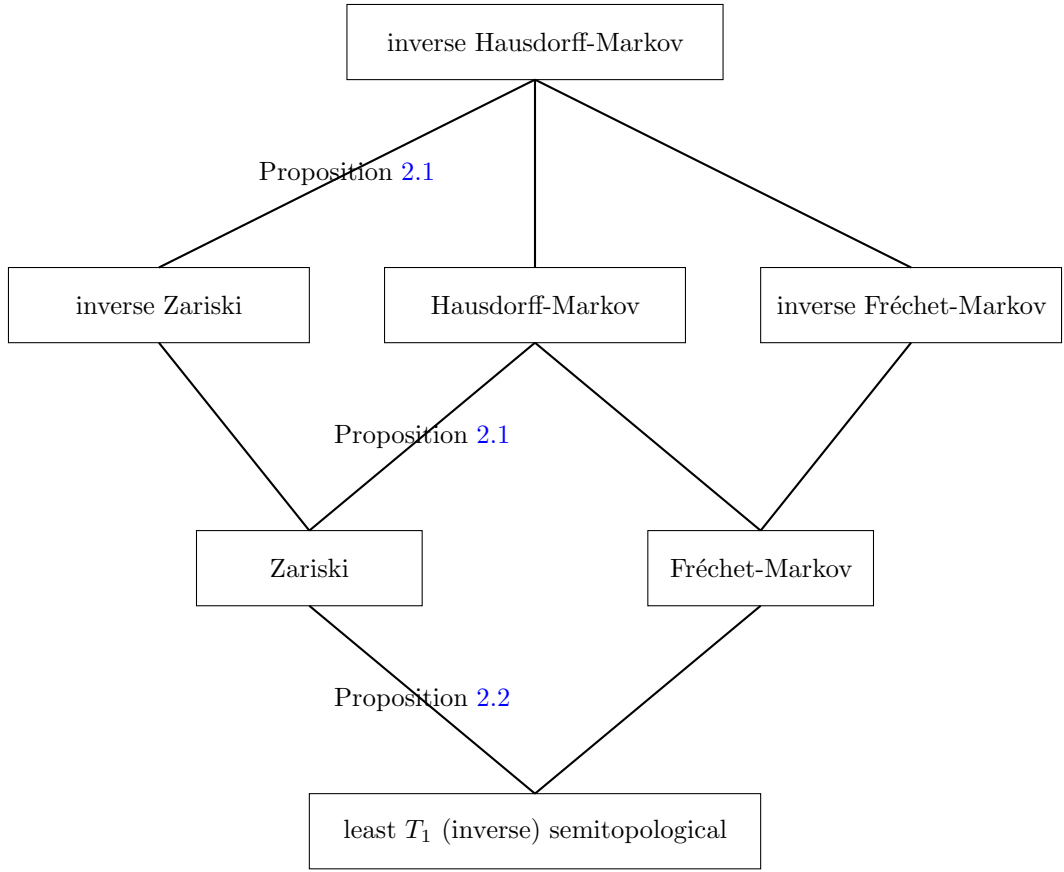
\begin{figure}
  \begin{tikzpicture}
    \begin{pgfonlayer}{nodelayer}
      \draw (-2.5,6.5) rectangle (2.5,7.5) node[pos=.5] {inverse
          Hausdorff-Markov};
      \draw (-7,3) rectangle (-3,4) node[pos=.5] {inverse Zariski};
      \draw (7,3) rectangle (3,4) node[pos=.5] {inverse Fr\'echet-Markov};

      \draw (-2,3) rectangle (2,4) node[pos=.5] {Hausdorff-Markov};
      \draw (-4.5,-0.5) rectangle (-1.5,0.5) node[pos=.5] {Zariski};
      \draw (4.5,-0.5) rectangle (1.5,0.5) node[pos=.5] {Fr\'echet-Markov};
      \draw (-3,-3) rectangle (3,-4) node[pos=.5] {least $T_1$ (inverse)
          semitopological};

      \node [style=none] (1) at (0, 3) {};

      \node [style=none] (2) at (-3, 0.5) {};
      \node [style=none] (5) at (-3, -0.5) {};

      \node [style=none] (3) at (3, 0.5) {};
      \node [style=none] (6) at (3, -0.5) {};

      \node [style=none] (4) at (0, -3) {};

      \node [style=none] (7) at (-5, 3) {};
      \node [style=none] (8) at (5, 3) {};
      \node [style=none] (9) at (-5, 4) {};
      \node [style=none] (10) at (0, 4) {};
      \node [style=none] (11) at (5, 4) {};
      \node [style=none] (12) at (0, 6.5) {};
    \end{pgfonlayer}
    \begin{pgfonlayer}{edgelayer}
      \draw [style=new edge style 2] (2.center) to node [midway]
      {\cref{prop-luke-1}} (1.center);
      \draw [style=new edge style 2] (4.center) to node [midway]
      {\cref{prop-luke-2}} (5.center);
      \draw [style=new edge style 2] (3.center) to (1.center);
      \draw [style=new edge style 2] (4.center) to (6.center);

      \draw [style=new edge style 2] (2.center) to  (7.center);
      \draw [style=new edge style 2] (3.center) to  (8.center);
      \draw [style=new edge style 2] (9.center) to node [midway]
      {\cref{prop-luke-1}} (12.center);
      \draw [style=new edge style 2] (10.center) to  (12.center);
      \draw [style=new edge style 2] (11.center) to (12.center);
    \end{pgfonlayer}
  \end{tikzpicture}
  \caption{A diagram of containments of the Hausdorff-Markov, Fr\'echet-Markov,
    and Zariski topologies for a given (inverse) semigroup; a line from one
    topology to another indicates that for every (inverse) semigroup the
    topology below is contained in the one above. Note that the absence of a
    line does not necessarily indicate that one topology is not contained in
    the other. A label on an edge shows where the indicated containment is
    shown in the paper, trivial containments are not labelled. It is shown that
    the least $T_1$ topology that is semitopological for an inverse semigroup
    $S$ coincides with the least $T_1$ topology that is inverse semitopological
    for $S$ in \cref{cor-inverse}.
  }
  \label{figure-the-only}
\end{figure}

An \textit{anti-homomorphism} from a semigroup $S$ to a semigroup $T$ is a
function $\phi: S \to T$ such that $(st)\phi=(t)\phi\ (s)\phi$ for all $s,t\in
  S$.  An \textit{anti-automorphism} of a semigroup $S$ is a bijective
anti-homomorphism from $S$ to $S$.

\begin{prop}\label{prop-anti-automorph}
  Let $S$ be a semigroup, let $\phi:S\to S$ be an automorphism or
  anti-automorphism, and let $\T$ be a topology on $S$. Then the following
  hold:
  \begin{enumerate}[\rm (i)]
    \item\label{prop-anti-automorph-i}
          if $\T$ is semitopological for $S$, then  $(\T)\phi =
            \set{(U)\phi}{U\in \T}$ is semitopological for $S$;

    \item\label{prop-anti-automorph-ii}
          if $\T$ is topological for $S$, then $(\T)\phi$
          is topological for $S$;

    \item\label{prop-anti-automorph-iii}
          if $\mathscr{B}$ is a subbasis for $\T$, then $\set{(B)\phi}{B \in
              \mathscr{B}}$ is a subbasis for $(\T)\phi$.
  \end{enumerate}
\end{prop}
\begin{proof}
  We will prove the proposition in the case when $\phi$ is an
  anti-automorphism.  The other case is even more straightforward.\medskip

  \noindent \textbf{(\ref{prop-anti-automorph-i}).} 
  We will show that $(\T)\phi$ is right
  semitopological. The proof that $(\T)\phi$ is left semitopological follows by symmetry.  Let $U\in \T$ be
  arbitrary. If $(s)\phi, (t)\phi \in S$ are such that $(s)\phi\ (t)\phi \in
    (U)\phi$, then, since $\phi$ is an anti-automorphism, $(s)\phi\
    (t)\phi=(ts)\phi \in (U)\phi$ and so $ts \in U$. Since $\T$ is left
  semitopological for $S$, there exists an open neighbourhood $V_s \in \T$ of
  $s$, such that $tV_s \subseteq U$. Thus $(V_s)\phi$ is an open neighbourhood
  of $(s)\phi$	under $(\T)\phi$ and $(V_s)\phi\ (t)\phi=(tV_s)\phi\subseteq
    (U)\phi$.\medskip

  \noindent \textbf{(\ref{prop-anti-automorph-ii}).} 
  Let $U\in \T$ be arbitrary. If $(s)\phi, (t)\phi \in S$ are such that
  $(s)\phi\ (t)\phi \in (U)\phi$, then, since $\phi$ is an anti-automorphism,
  $(s)\phi\ (t)\phi=(ts)\phi \in (U)\phi$ and so $ts \in U$. Since $\T$ is a
  semigroup topology, there exist open neighbourhoods $V_s, V_t \in \T$ of $s$
  and $t$, respectively, such that $V_tV_s \subseteq U$. Thus $(V_s)\phi$ and
  $(V_t)\phi$ are open neighbourhoods of $(s)\phi$  and $(t)\phi$ under
  $(\T)\phi$ and $(V_s)\phi\ (V_t)\phi=(V_tV_s)\phi\subseteq (U)\phi$.\medskip

  \noindent \textbf{(\ref{prop-anti-automorph-iii}).} 
  By definition, $\set{(B)\phi}{B \in \mathscr{B}} \subseteq (\T)\phi$.
  Furthermore, $\phi$ is a bijection, so the unions and (finite) intersections
  of images are images of unions and (finite) intersections.
\end{proof}

\begin{prop}\label{prop-all-topo}
  Every automorphism, or anti-automorphism, of a semigroup $S$ is continuous
  with respect to each of the  Zariski, Fr\'echet-Markov, and Hausdorff-Markov
  topologies for $S$, as well as the least $T_1$ topology that is
  semitopological for $S$.
\end{prop}
\begin{proof}
  We first consider the Zariski topology. It suffices to show that the image of
  an elementary algebraic set under an automorphism or anti-automorphism is
  also elementary algebraic.  Let
  \[F=\{s\in S:t_0 s t_1 s  \cdots t_{k-1} s =u_0 s
    u_1
    s  \cdots u_{l-1} s\}\]
  be an elementary algebraic set. Let $\phi:S\to S$ be an anti-automorphism.
  Then
  \begin{align*}
    (F)\phi & =\{(s)\phi\in S:t_0 s t_1 s  \cdots t_{k-1} s =u_0 s  u_1 s
    \cdots u_{l-1} s \}                                                     \\
            & =\{s\in S:t_0 (s\phi^{-1})  t_1 (s\phi^{-1}) \cdots t_{k-1}
    (s\phi^{-1}) =u_0 (s\phi^{-1}) u_1 (s\phi^{-1})
    \cdots u_{l-1} (s\phi^{-1}) \}
    \\
            & =\{s\in S:(t_0 (s\phi^{-1})  t_1 (s\phi^{-1})  \cdots t_{k-1}
    (s\phi^{-1}))\phi=(u_0 (s\phi^{-1})  u_1 (s\phi^{-1})  \cdots u_{l-1}
    (s\phi^{-1}) )\phi\}                                                    \\
            & =\{s\in S: s t_{k-1}\phi s \cdots t_1\phi s t_0\phi= s
    u_{l-1}\phi s  \cdots u_1\phi s u_0\phi\},
  \end{align*}
  which is elementary algebraic. Similarly if $\phi$ is an automorphism of $S$,
  then $(F)\phi$ is also elementary algebraic.

  Let $\phi$ be an automorphism or anti-automorphism of $S$.  If $\T$ is any
  $T_1$ semigroup topology on $S$, then, by \cref{prop-anti-automorph},
  $(\mathcal{T})\phi$ is also a $T_1$ semigroup topology for $S$.  If
  $\mathfrak{T}$ is the collection of all $T_1$ semigroup topologies on $S$,
  then
  \[\left(\bigcap_{\T\in \mathfrak{T}} \T\right) \phi = \bigcap_{\T\in
      \mathfrak{T}} (\T \phi) = \bigcap_{\T\in \mathfrak{T}} \T\]
  and thus $\phi$ is continuous with respect to the Fr\'echet-Markov topology,
  which equals $\bigcap_{\T\in \mathfrak{T}} \T$.  The proofs for the
  Hausdorff-Markov topology and the least $T_1$ topology that is
  semitopological for $S$ are similar (the latter using
  \cref{prop-anti-automorph}\eqref{prop-anti-automorph-i}.
\end{proof}

The analogue of \cref{prop-all-topo} holds for the inverse Zariski, inverse
Fr\'echet-Markov, and inverse Hausdorff-Markov, and the proof is similar.

\begin{cor}\label{cor-inverse}
  If $S$ is an inverse semigroup, then inversion is continuous in each of the
  Zariski, Fr\'echet-Markov, and Hausdorff-Markov topologies on $S$, as well as
  the least $T_1$ topology that is semitopological for $S$.
\end{cor}

As mentioned above, if a topological group is $T_0$, then it is
$T_{3\frac{1}{2}}$, and hence being $T_0$ and $T_{3\frac{1}{2}}$ is equivalent
for such topologies.  On the other hand, every topology is a semigroup
topology, and so no such implication holds for topological semigroups, in
general. It is natural to ask if there is any implication among separation
axioms for certain classes of semigroups, such as the inverse semigroups. We
show, in \cref{ex-hausdorff}, that there is a topological inverse semigroup
which is $T_1$ but not $T_2$. In fact, we find a inverse semigroup such that
its Fr\'echet-Markov and its Hausdorff-Markov topologies are not equal.

Let $\leq$ be a total order on a set $X$. Then the \textit{order topology} on
$X$ is the topology with subbasis consisting of the sets $\set{y\in X}{x < y}$
and $\set{y\in X}{x > y}$.  Moreover, $X$ together with the operation $\max$
forms a commutative inverse semigroup.

\begin{ex}
  \label{ex-hausdorff}
  Let $X$ be an infinite set, let $\leq$ be a total order on $X$, and endow $X$
  with the structure of a semigroup by taking multiplication to be $\max$.
  Then:
  \begin{enumerate}[\rm (i)]
    \item\label{ex-hausforff-i}
          the Hausdorff-Markov topology and the Zariski topology on $X$
          coincide
          with the order topology on $X$;

    \item\label{ex-hausforff-ii}
          the sets of the form
          \[
            B_{x,U}:=\set{y\in U}{y>x}
          \]
          where $U$ is a cofinite subset of $X$ and $x\in X ^ 1$ is arbitrary,
          form a basis for the Fr\'echet-Markov topology on $X$.
  \end{enumerate}
\end{ex}
\begin{proof}
  \textbf{(\ref{ex-hausforff-i}).} The elementary algebraic sets for the Zariski topology are precisely
  $V_{a,b}:=\{x\in X:\max\{x,a\}=\max\{x,b\}\}$ and $W_{a,b}:=\{x\in
    X:\max\{x,a\}=b\}$ where $a, b\in X ^ 1$ are arbitrary.

  If $a, b\in X$ are arbitrary, then the following hold:
  \begin{enumerate}
    \item
          if $a=b$, then $V_{a,b}=X$ and $W_{a, b} = \set{x\in X}{x \leq a}$;

    \item
          if $a<b$, then $V_{a,b}=\set{x\in X}{x\geq b}$ and $W_{a, b} =
            \{b\}$;

    \item
          if $a>b$, then $V_{a, b} = \set{x\in X}{x\geq a}$ and
          $W_{a,b}=\varnothing$.
  \end{enumerate}
  Hence the complements of the sets $V_{a, b}$ and $W_{a, b}$ are subbasic open
  sets for the order topology on $X$ and are a basis for the Zariski topology,
  and so these two topologies coincide.

  The order topology is also a Hausdorff semigroup topology for $X$. It is
  routine to verify that $X$ is Hausdorff. We now show that the order topology
  is indeed a semigroup topology.  Suppose that $\max\{a, b\} = b \in U$ for
  some subbasic open $U$.  If $a, b \in U$, then $UU = U$. If $a\not\in U$,
  then $U = \set{y\in X}{x < y}$ for some $x < b$, and so if $V = \set{y\in
      X}{b > y}$, then $VU\subseteq U$.  Finally, the order topology, being
  equal
  to the Zariski topology, is contained in the Hausdorff-Markov topology, which
  is contained in the order topology by above.\medskip

  \noindent\textbf{(\ref{ex-hausforff-ii}).}
  Let $\mathcal{T}$ be the topology on $X$ with the sets $B_{x, U}$ as a basis.
  Note that $B_{1,U}=U$ for all $U$ where $1\in X ^ 1$ is the adjoined
  identity. Hence the cofinite sets are basic open in $\mathcal{T}$, and so
  $\T$ contains the cofinite topology, and it is $T_1$.  Let $\S$ be a $T_1$
  semigroup topology for $X$. We show that $\mathcal{T}\subseteq \S$. Since
  $\S$ is $T_1$, the singletons are closed in $\S$ and hence
  \[
    (X\backslash\{x\})\lambda_x^{-1}=\{y\in X: \max\{x,y\}\neq x\}= \{y\in
    X:y>x\}\in \S
  \]
  for all $x\in X$. If $x\in X$ is arbitrary and $U$ is a cofinite subset of
  $X$, then
  \[B_{x,U} =\{y\in X:y>x\} \cap U, \]
  and so $\mathcal{T} \subseteq \S$. It follows that $\T$ is contained in the
  Fr\'echet-Markov topology for $X$, and hence it suffices to show that $\T$
  contains the Fr\'echet-Markov topology. We will show that $\mathcal{T}$ is a
  semigroup topology for $X$.

  Let $B_{s,U}\in \mathcal{T}$ be an arbitrary basic open set and let $a,b\in
    X$ be such that $\max\{a,b\}\in B_{s,U}$. We show that there are open
  neighbourhoods $U_a,U_b$ of $a$ and $b$ respectively, such that
  $U_aU_b\subseteq B_{s,U}$. If $a=b$ then we choose $U_b=U_a=B_{s,U}$.
  Otherwise assume without loss of generality that $a<b$. Let $U_b:=
    B_{\max\{a,s\},U}$ and $U_a:= U \cup \{a\}$. If $a' \in U_a$ and $b'\in
    U_b$
  then either $\max\{a',b'\}=b'\in U_{b}\subseteq B_{s,U}$ or
  $a'>b'>\max\{a,s\}$, in which case $\max\{a',b'\}=a'\in
    B_{\max\{a,s\},U}\subseteq B_{s,U}$ as required.
\end{proof}

The topology defined in \cref{ex-hausdorff}\eqref{ex-hausforff-i} can be distinct from the
topology defined in \cref{ex-hausdorff}\eqref{ex-hausforff-ii}. For example, if $X=\Z$, then
every open set in the Fr\'echet-Markov topology is unbounded above, whereas
some of the basic open sets in the Hausdorff-Markov topology are bounded above
by definition. It follows that Fr\'echet-Markov topology for a semigroup can be
strictly contained in its Hausdorff-Markov topology, even if we only consider
commutative inverse semigroups.

\subsection{Open questions}

We end this section with some open problems.
\begin{question}
  Is the Fr\'echet-Markov topology always contained in the Zariski topology?
\end{question}

\begin{question}\label{question-zariski-2}
  Is the semigroup version of the Zariski topology applied
  to a group equal to the group version of the Zariski topology
  applied to the same group?
\end{question}

We will determine the Zariski topologies on several well-known classes of
semigroups. Notably absent from this list is the symmetric inverse monoid, and
so we ask the following question also.

\begin{question}
  What is the Zariski topology of the symmetric inverse monoid?
\end{question}

\section{Property \textbf{X}}\label{section-property-x}

In this section, we introduce a property that will be central to this paper.
If $S$ is a semitopological semigroup and $A$ is a subset of $S$, then we say
that $S$ satisfies \textit{property \textbf{X} with respect to $A$} if the
following holds:
\begin{quote}
  for every $s\in S$ there exists $f_s, g_s\in S$ and $t_s\in A$ such that
  $s = f_s t_s g_s$ and for every neighbourhood $B$ of $t_s$ the set $f_s
    (B\cap
    A)g_s$ is a neighbourhood of $s$.
\end{quote}
Although somewhat technical, property \textbf{X} is the crucial ingredient in
many of the proofs in this paper where we determine unique or maximal Polish
topologies on a semigroup $S$, or show automatic continuity for $S$ with
respect to the class of second countable topological semigroups.

\begin{theorem}
  \label{lem-luke-0}
  Let $S$ be a semigroup, let $\T$ be a topology with respect to which $S$
  is semitopological, and let $A\subseteq S$. If $S$ has property \textbf{X}
  with respect to $A$, then the following hold:
  \begin{enumerate}[\rm (i)]
    \item
          \label{lem-luke-0-i}
          if $T$ is a semitopological semigroup and $\phi: S \to T$ is a
          homomorphism such that $\phi|_{A}$ is continuous, then $\phi$ is
          continuous;

    \item
          \label{lem-luke-0-ii}
          if $\T'$ is a topology with respect to which $S$ is semitopological
          and
          $\T'$ induces the same topology on $A$ as $\T$, then $\T'$ is
          contained
          in $\T$;

    \item
          \label{lem-luke-0-iii}
          if $\T$ is Polish and $A$ is a Polish subgroup of $S$, then $\T$ is
          maximal among the Polish topologies with respect to which $S$ is
          semitopological;

    \item
          \label{lem-luke-0-iv}
          if $A$ is a semigroup which has automatic continuity with respect to
          a class $\mathcal{C}$ of topological semigroups, then the semigroup
          $S$ has automatic continuity with respect to $\mathcal{C}$ also.
  \end{enumerate}
\end{theorem}

To prove \cref{lem-luke-0}, we require the notion of a Borel measurable
function between topological spaces. Recall that a \textit{$\sigma$-algebra} on
a set $X$ is a collection of subsets of $X$ containing $\varnothing$ and which
is closed under complements and countable unions (and hence also closed under
countable intersections).  If $X$ is a topological space, then a set $B$ is
\textit{Borel} if it belongs to the least $\sigma$-algebra containing the open
sets in $X$.  If $X$ and $Y$ are topological spaces, then $f : X\to Y$ is
\textit{Borel measurable} if the pre-image of every Borel set is Borel.
We require the following propositions which follows from Theorem 9.10, Proposition 11.5,  and Corollary 15.2 in~\cite{Kechris1995aa}.

\begin{prop}
  \label{prop-borel-imples-cts}
  If $G$ and $H$ are Polish semitopological groups and $f: G \to H$ is a Borel
  measurable homomorphism, then $f$ is continuous.
\end{prop}

\begin{prop}
  \label{prop-borel-measurable}
  If $X$ and $Y$ are Polish spaces and $f: X \to Y$ is a Borel measurable
  bijection, then $f ^ {-1}$ is Borel measurable also.
\end{prop}

\begin{proof}[Proof of \cref{lem-luke-0}]
  \textbf{(\ref{lem-luke-0-i}).}
  We denote the topology on $T$ by $\T'$. We will show that $\phi$ is
  continuous
  at an arbitrary $s\in S$. Suppose that $U\in \T'$ is an open neighbourhood
  of $(s)\phi$. By property \textbf{X}, there are $f_s, g_s\in S$ and $t_s\in
    A$ such that $s = f_st_sg_s$.
  Since $\phi$ is a homomorphism,
  $(s)\phi=(f_st_sg_s)\phi=(f_s)\phi\ (t_s)\phi\ (g_s)\phi$ and so  $(t_s)\phi
    \in V$
  where
  \[V = (U)(\lambda_{(f_s)\phi} \circ \rho_{(g_s)\phi})^{-1}.\]
  In particular, since $T$ is semitopological, $V$ is an open
  neighbourhood of $(t_s)\phi$ in $\T'$ and
  \[((f_s)\phi) V((g_s)\phi) \subseteq U.\]
  Since $\phi|_A$ is continuous, $(V)\phi^{-1}\cap A$ is open in the subspace
  topology on $A$ induced by $\T$. Hence $(V)\phi^{-1}\cap A=W\cap A$ for some
  $W \in \T$.
  By property \textbf{X}, there exists an open neighbourhood $B$ of
  $s$ such that $B \subseteq f_s(W\cap A)g_s$. Then
  \begin{align*}
    (B)\phi & \subseteq \left(f_s(W\cap A)g_s)\right) \phi
    =\left(f_s((V)\phi^{-1}\cap A)g_s\right)\phi
    \subseteq \left(f_s ((V)\phi^{-1}) g_s\right)\phi\subseteq ((f_s)\phi) V
    ((g_s)\phi)\subseteq U,
  \end{align*}
  and so $\phi$ is continuous at $s$. \medskip

  \noindent\textbf{(\ref{lem-luke-0-ii}).}
  Let $\T'$ be a semitopological semigroup topology for $S$ that induces the
  same subspace topology as $\T$ on $A$. Then the restriction of the identity
  homomorphism $\id:(S,\T) \rightarrow (S,\T')$ to $A$ is
  continuous.
  Thus $\id:(S,\T) \rightarrow
    (S,\T')$ is continuous by part \eqref{lem-luke-0-i} and so $\T'\subseteq \T$, as required.
  \medskip

  \noindent\textbf{(\ref{lem-luke-0-iii}).}
  Suppose that $\T'$ is a Polish semitopological semigroup topology on $S$ and
  that $\T\subseteq \T'$. We will show that $\T'\subseteq \T$, and so $\T$ is
  maximal. As in the previous part, it suffices to show that the restriction
  $\id|_A$ of the identity function $\id: (S,\T)  \to (S,\T')$ is continuous.

  Since $A$ is a Polish subspace of $S$ with respect to $\T$, it follows that
  $A$ is $G_{\delta}$ in $\T$, and so $A$ is $G_{\delta}$ in $\T'$ also.  Hence
  $A$ is a Polish subspace of $S$ with respect to $\T'$.  Since $\id|_A^{-1}$
  is a Borel measurable bijection between Polish spaces, it follows from
  \cref{prop-borel-measurable} that $\id|_{A}$ is a Borel measurable
  function between $(S, \T)$ and $(S, \T')$.  Therefore, by
  \cref{prop-borel-imples-cts}, $\id|_A$ is continuous.
  \medskip

  \noindent\textbf{(\ref{lem-luke-0-iv}).}
  Suppose that $T$ is a topological semigroup belonging to the class
  $\mathcal{C}$ and that $\phi: S \to T$ is a homomorphism. It follows that
  $\phi|_{A}$ is a homomorphism from $A$ to $T$ which is therefore continuous
  by automatic continuity. It follows from part \eqref{lem-luke-0-i} that $\phi$ is continuous.
\end{proof}

We end this section by observing that ``having property \textbf{X} with respect
to a subset'' is a transitive relation.

\begin{lem}\label{lem-prop-X-transitive}
  Let $S$ be a topological semigroup and let \(A\leq T\leq S\). If \(S\) has
  property \textbf{X} with respect to \(T\) and \(T\) has property \textbf{X}
  with respect to \(A\), then \(S\) has property \textbf{X} with respect to
  \(A\).
\end{lem}
\begin{proof}
  Let \(s\in S\). Since $S$ has property \textbf{X} with respect to $T$, there
  exist \(f_s, g_s\in S\) and \(t_s\in T\) such that $f_s t_s g_s = s$ and
  \(f_s(B'\cap T)g_s\) is a neighbourhood of \(s\) for every neighbourhood
  \(B'\) of \(t_s\).  Similarly, there exist \(t_{t_s}\in A\) and \(f_{t_s},
  g_{t_s}\in T\) such that $t_s = f_{t_s} t_{t_s} g_{t_s}$ and \(f_{t_s}(B\cap
  A)g_{t_s}\) is a neighbourhood of \(t_s\) in \(T\) for every neighbourhood
  \(B\) of \(t_{t_s}\).

  If \(B\) is any neighbourhood of \(t_{t_s}\) in \(S\), then \(f_{t_s}(B\cap
  A)g_{t_s}\) is a neighbourhood of \(t_s\) in \(T\). Hence there is a
  neighbourhood \(B'\) of \(t_s\) in \(S\) such that \(B'\cap T=f_{t_s}(B\cap
  A)g_{t_s}\). In particular, \(f_sf_{t_s}(B\cap A)g_{t_s}g_s=f_s(B'\cap
  T)g_s\) is a neighbourhood of \(s\) and \(s = f_sf_{t_s} t_{t_s}
  g_{t_s}g_s\), and $S$ has property \textbf{X} with respect to $A$.
\end{proof}

\section{Automatic continuity and the small index property}
\label{subsection-auto-cont}

The notions of automatic continuity for semigroups and groups, as defined in
the
introduction, are superficially different. The former is defined with respect
to a
class of topological semigroups, and the latter with respect to a class of
topological groups.  Of course, a class of topological groups is also a class
of
topological semigroups. We will use the automatic continuity of the symmetric
group, and the homeomorphisms of the Cantor space, with respect to the class of
second countable semigroups in Sections~\ref{section-classical-m}
and~\ref{section-continuous}, and so we require the following proposition.

\begin{prop}\label{lem-group-auto-cont-is-semigroup}
  Let $G$ be a semitopological group. Then $G$ has automatic continuity with
  respect to the class of second countable topological semigroups if
  and only if $G$ has automatic continuity with respect to the class of second
  countable topological groups.
\end{prop}
\begin{proof}
  ($\Leftarrow$)
  Let $S$ be a second countable topological semigroup and let $\phi: G \to S$
  be a homomorphism.  Then $H:=(G)\phi$ is a subgroup of $S$ and the
  induced topology on $H$ is a second countable paratopological group
  topology for $H$; we denote this topology by $\T$. If
  $\mathcal{T}^{-1} := \{U^{-1}: U\in \mathcal{T}\}$, then, by
  \cref{prop-anti-automorph},	$\mathcal{T}^{-1}$
  is a semigroup topology on $H$ and, since $\T ^ {-1}$ is homeomorphic to
  $\T$, it is second countable.
  It then follows that the topology $\mathcal{T}'$ generated by $\mathcal{T}$
  and $\mathcal{T}^{-1}$ is also a second countable paratopological group
  topology for $H$. Since the topology $\mathcal{T}'$ is generated by an
  inverse closed collection of sets, it follows that the inverse operation of
  $H$ is continuous under this topology, and thus $(H, \mathcal{T}')$ is a
  second countable topological group. Since
  $\phi$ is a homomorphism, $\phi$ is continuous with respect to $\T'$ and the
  topology on $G$.
  If $U\in \mathcal{T}$, then $U\in \T '$ and so
  $(U)\phi^{-1}$ is open in $G$ by the continuity of $\phi:G
    \to (H, \mathcal{T}')$.

  ($\Rightarrow$) This is immediate as second countable
  topological groups are second countable topological semigroups.
\end{proof}

\begin{cor}
  The symmetric group $\Sym(\N)$ and the group $H(2 ^ \N)$ of homeomorphisms of
  the Cantor space $2 ^ {\N}$ have automatic continuity with respect to the class
  of second countable topological semigroups.
\end{cor}
\begin{proof}
  It is known that the given groups both have automatic continuity with respect
  to the class of second countable topological groups (see~\cite{Kechris2007aa}
  and~\cite{Rosendal2007ac}), and so the result follows immediately by
  \cref{lem-group-auto-cont-is-semigroup}.
\end{proof}

A topological group $G$ has the \textit{small
  index property} if every subgroup of index at most $\aleph_0$ is open; some
authors assume this condition for subgroups of index less than $2 ^
    {\aleph_0}$.
A topological group $G$ has the small index property if and only
if every homomorphism from $G$ to $\Sym(\N)$ with the pointwise topology is
continuous.
However, having the small index property is strictly weaker than having
automatic continuity with respect to the class of the second countable
topological groups.
In particular, in \cref{example-small-index-reals}, we give an example of
a group satisfying the small index property but that does not have automatic
continuity with respect to the class of second countable topological groups.

We define a notion analogous to small index property for topological
semigroups.
A topological semigroup \(S\) is said to have the \textit{right small index
  property} if
every right congruence on \(S\) with countably many classes is open in \(S
\times S\) with the product topology. Similarly,
\(S\) is said to have the \textit{left small index property} if every left
congruence on
\(S\) with countably many classes is open. If $G$ is a topological group, then
$G$ has the left small index property if and only if $G$ has the right small
index property.

In \cref{left-not-right}, we show that the notions of left and right
small index property are distinct, by exhibiting an example which has one
property but not the other. If a topological semigroup $S$ has automatic
continuity, then $S$ has both left and right small index property
(\cref{AC-implies-SIP}).  In
\cref{prop-right-small-index}, we will show that a topological
monoid $S$ has the right small index property if and only if every homomorphism
from $S$ to $\N ^ \N$ with the topology of pointwise convergence is continuous.

\begin{prop}\label{prop-right-small-index}
  A topological monoid \(M\) has the right small index property if and only if
  every homomorphism from \(M\) to \(\N^\N\) with the pointwise topology is
  continuous.
\end{prop}
\begin{proof}
  First note that an equivalence relation on a topological space is open if and
  only
  if all of its equivalence classes are open.

  We start by assuming that \(M\) has the right small index property. Let
  \(\phi: M \to
  \N^\N\) be a semigroup homomorphism. Recall that the pointwise topology on
  $\N^\N$ has a subbasis consisting of the sets $U_{i,j}=\set{f\in
      \N^\N}{(i)f=j}$ over all $i,j\in \N$.
  Note that $U_{i,j}$ is an equivalence class under the right congruence
  \[\rho := \{(f, g)\in \N^\N: (i)f= (i)g\}.\]
  Since $\phi$ is a homomorphism, the relation \(\rho':=\{(f, g)\in M:
  (f\phi,g\phi)\in \rho\}\)
  on \(M\) is a right congruence on the monoid \(M\) and $\rho'$ has at most as
  many classes
  as \(\rho\). Since $M$ has the right small index property, \(\rho'\) is open.
  The preimage of $U_{i,j}$ under \(\phi\) is an equivalence
  class of \(\rho'\) and hence open. Since $U_{i,j}$ is an arbitrary subbasic
  open set, it follows that \(\phi\) is continuous.

  Now suppose that every semigroup homomorphism \(\phi: M \to \N^\N\) is
  continuous. Let \(\rho\) be a right congruence on \(M\) with countably many
  classes. We define a homomorphism \(\phi:M \to (M/\rho)^{(M/\rho)}\) by
  \[(g/\rho)((f)\phi) = gf/\rho.\]
  The semigroup \((M/\rho)^{(M/\rho)}\), together with the pointwise topology,
  is topologically isomorphic
  to \(\N^\N\) and thus \(\phi\) is continuous with respect to these
  topologies. Let
  \(m/\rho\) be an arbitrary equivalence class of \(\rho\). It suffices to show
  that
  \(m/\rho\) is an open subset of \(M\). The set $\{(f)\phi\in
    (M)\phi:(1_M/\rho)(f)\phi = m/\rho\}$ is open in $(M)\phi$ under the subspace
  topology inherited from ${M/\rho}^{M/\rho}$ and
  \[\{(f)\phi \in (M)\phi:(1_M/\rho)((f)\phi) = m/\rho\}\phi^{-1}=\{(f)\phi \in
    (M)\phi:f/\rho= m/\rho\}\phi^{-1}=m/\rho.\]
  Since \(\phi\) is continuous, it follows that \(m/\rho\) is indeed open as
  required.
\end{proof}
By symmetry, we obtain the following corollary to
\cref{prop-right-small-index}.
\begin{cor}\label{corollary-left-congruence}
  A topological monoid \(M\) has the left small index property if and only if
  every homomorphism from \(M\) to \(\N^\N\) with left actions is continuous.
\end{cor}

The following corollary follows straight from the definition of automatic
continuity, \cref{prop-right-small-index}, and
\cref{corollary-left-congruence}.

\begin{cor}\label{AC-implies-SIP}
  If a topological monoid has automatic continuity with respect to the class of
  second countable topological semigroups, then it has both the left and right
  small index properties.
\end{cor}
The converse of \cref{AC-implies-SIP} is not true. We give an example which
demonstrates that the small index property is strictly weaker than automatic
continuity even in the case of abelian Hausdorff topological groups.

\begin{ex}\label{example-small-index-reals}
  Define the topology $\T$ on the group of real numbers $\R$ under addition by
  choosing the cosets of all countable index subgroups as a subbasis.
  This subbasis is actually a basis as the intersection of finitely many
  subgroups of countable index is another subgroup of countable index. Moreover,
  $\T$ is a group topology since for every countable index subgroup \(G\) of $\R$
  we have \((G+x)+(G+y)\subseteq (G+x+y)\) and \(-(G+x)=(G-x)\) for all \(x,y\in
  \R\). By its definition, $\T$ clearly gives $(\R, +)$ the small index property.

  Since $\T$ is a group topology, to show that $\T$ is Hausdorff, it suffices
  to show that $\T$ is $T_0$. We will show that for all \(x\in \mathbb{R}
  \backslash \{0\}\) there is a countable index subgroup of \(\mathbb{R}\) which
  does not contain \(x\). Let \(x\in \mathbb{R} \backslash \{0\}\). By Zorn's
  Lemma we can extend the set \(\{x\}\) to a basis \(\B\) for \(\mathbb{R}\) as a
  vector space over \(\mathbb{Q}\). Let \(G\) be the subspace spanned by \(\B
  \backslash \{x\}\). As \(\B\) is linearly independent, \(x\not\in G\). But as
  \(G\) is a codimension \(1\) subspace of a vector space over a countable field,
  it follows that \(G\) has countable index as required.

  To see that $(\R, +)$ under $\T$ does not have automatic continuity, we will
  show that the identity map from $\R$ under $\T$ to $\R$ with the standard
  topology is not continuous. The interval $(-1,1)$ is not open in $\T$ since all
  non-empty open sets in $\T$ contain a translation of a non-trivial group, and
  are thus unbounded.
\end{ex}

In \cref{left-not-right} we give an example of a topological
semigroup which has the right small index property but not the left small index
property, so in particular it also does not have automatic continuity.
Moreover, this shows that the right and left small  index properties are not
equivalent in general.

Although \cref{example-small-index-reals} shows that the small index property
is strictly weaker than automatic continuity for topological groups in general,
we ask if these notions are equivalent for Polish groups.

\begin{question}
  If $G$ is a Polish group and G has the small index property, then does $G$
  have automatic continuity with respect to the class of second countable groups?
\end{question}

\section{Classical monoids}
\label{section-classical-m}
In this section, we consider several monoids that have been extensively studied
in the literature which we refer to collectively as \textit{classical monoids}.

Let $X$ be any set.  A \textit{binary relation} on $X$ is just a subset of $X
  \times X$.  We will denote the set of all binary relations on $X$ by $B_X$.
If
$f,g \in B_X$, then their composition $f\circ g$ is defined by $$(x,y) \in
  f\circ g\quad \text{if } (x,z) \in f \text{ and } (z,y) \in g \text{ for some
  }
  z\in X.$$ The set $B_X$ with composition of binary relations is the
\textit{full binary relation monoid} on $X$.
If $f\in B_X$, then we define the \textit{inverse} of $f$ to be
$f^{-1}=\set{(y,x)}{(x,y) \in f}$. The relation $f^{-1}$ is not always
an inverse for $f$, in the sense of inverse semigroups or groups, since
$f \subseteq f f^{-1} f$, and this containment may be strict.

If $\Sigma \subseteq X$ and $f\in B_X$, then the \textit{image of
  $\Sigma$ under $f$} is the set
$$(\Sigma)f=\set{y\in X}{(x,y) \in f \text{ for some } x \in \Sigma}.$$
The \textit{domain} and \textit{image} of $f$ are $\dom(f)=(X)f^{-1}$ and
$\im(f)=(X)f$.

We will also consider the following natural
subsemigroups of $B_{X}$: the \emph{partial transformation monoid}
\[
  P_X = \set{f\in B_{X}}{|(\{x\})f| \leq 1\ \text{for all }x\in X};
\]
the \textit{full transformation monoid}
\[
  X ^ X = \set{f\in P_{X}}{|(\{x\})f| = 1\ \text{for all }x\in
    X};
\]
the \textit{symmetric inverse monoid}
\[
  I_X = \set{f\in P_X}{
    |(\{x\})f ^ {-1}| \leq 1\ \text{for all } x\in X};
\]
the \textit{monoid of injective functions}
\[
  \Inj(X) = \set{f\in X ^ X}{|(\{x\})f ^ {-1}| \leq 1\ \text{for
      all } x\in X};
\]
the \textit{monoid of surjective functions}
\[
  \Surj(X) = \set{f\in X ^ X}{(X)f = X}.
\]
These monoids have been extensively studied in the literature
in both the finite and infinite cases; see, for example, \cite{East2014,
  Higginsaa, Mesyan2007, Mesyan2011, Mesyan2012, Mitchell2011cd} and the
references therein. Of course, (partial) functions are sets of pairs, and so we may write $(x,y) \in f\in P_X$ to denote that $x\in \dom(f)$ and $(x)f=y$. Similarly, 
we may write that $f\subseteq g$ where $f, g\in P_X$ and $g|_{\dom(f)} = f$.

In \cref{subsection-diagram}, we will also briefly consider certain,
so-called, diagram monoids. The definition of these monoids is more involved,
and as it turns out, studying them as topological monoids is less fruitful than
the monoids defined above. For these reasons we delay the definitions of these
monoids to \cref{subsection-diagram}.

\subsection{The full transformation monoid}
\label{subsection-full-transf}
Recall that a subbasis for the pointwise topology on $X ^ X$, where $X$ is any
infinite set,  consists of the
sets
\[
  U_{x, y} = \set{f\in X ^ X}{(x)f = y}
\]
for all $x,y\in X$. This semigroup topology on $X ^ X$ is essentially the only
such
topology considered in the literature. While it is possible to define further
semigroup topologies on $X ^ X$, we will show that every such semigroup
topology on $X ^ X$ either fails to be $T_1$ or contains the pointwise
topology. If $X$ is countable, then we will show that the pointwise topology is
the unique $T_1$ second countable semigroup topology on $X ^ X$; this may
account, at least in part, for the absence of alternative topologies in the
literature. Furthermore, $X ^ X$ has automatic continuity with respect to the
pointwise topology and the class of second countable topological semigroups. We
show that the pointwise topology coincides with the Zariski topology and
characterise those topological semigroups that topologically embed into $X ^ X$
when $X$ is countable.

We begin this subsection with a technical lemma that we will use repeatedly
later on.

\begin{lem}
  \label{thm-elliott-1}
  Let $X$ be an infinite set, and let $S$ be a subsemigroup of $P_X$ such that
  $S$ contains all of the constant  transformations (defined everywhere on
  $X$), and for every $x\in X$
  there exists $f_{x}\in S$ such that $(x) f_{x} ^ {-1} = \{x\}$ and $(X)f_{x}$
  is finite.  If $\T$ is a topology which is semitopological for
  $S$, then the following are equivalent:
  \begin{enumerate}[\rm (i)]
    \item \label{thm-elliott-1-i} 
    $\T$ is Hausdorff;
    \item \label{thm-elliott-1-ii} 
    $\T$ is $T_1$;
    \item \label{thm-elliott-1-iii} 
    $\{f\in S: (y, z) \in f\}$ and $\{f \in S :
            t\not\in\dom(f)\}$ are open with respect to
          $\T$ for all $y, z, t \in X$;
    \item \label{thm-elliott-1-iv} 
    $\{f\in S: (y,z)\in f\}$ and $\{f \in S :
            t\not\in\dom(f)\}$ are closed with respect to
          $\T$ for all $y, z, t \in X$.
  \end{enumerate}
\end{lem}
\begin{proof}
  \textbf{(\ref{thm-elliott-1-i}) $\Rightarrow$ (\ref{thm-elliott-1-ii}).} This follows since every Hausdorff space is
  \(T_1\).
  \medskip

  \noindent\textbf{(\ref{thm-elliott-1-ii}) $\Rightarrow$ (\ref{thm-elliott-1-iii}).}  
  Suppose that $y, z\in
    X$, and $g, h \in S$ are such that $h$ is the constant transformation
  with value $y$, and $g$ is arbitrary. Then $y\notin \dom(g)$ or $(y)g \not =
    z$ if
  and only if $hgf_{z}$ is $\varnothing$ or a constant with value belonging to
  $(X)f_{z}\setminus
    \{z\}$.
  Since $(X)f_{z}$ is finite by assumption, it follows that $\set{f\in
      S}{(y,z)\not \in f}$ is the preimage under
  left multiplication by $h$ and right
  multiplication by $f_{z}$ of the finite set consisting of the empty function
  and those constant
  transformations whose image belongs to $(X)f_{z}\setminus \{z\}$. Hence
  $\set{f\in S}{(y,z)\not \in f}$ is closed.

  Similarly, $y\in\dom(g)$ if and only if $hgf_y$ is a constant with the value
  belonging to $(X)f_y$. In this case, we obtain \(\{f\in S\colon y\in
  \dom(f)\}\) as the preimage under left multiplication by $h$ and right
  multiplication by $f_{y}$ of the finite set consisting of those constant
  transformations whose image belongs to $(X)f_{y}$.

  In the two cases above, we showed that the sets $\{f\in S\colon (y,z) \not \in f \}$ and $\{f \in S \colon t\in\dom(f)\}$ are
  closed, and so their complements are open, as required.
  \medskip

  \noindent\textbf{(\ref{thm-elliott-1-iii}) $\Rightarrow$ (\ref{thm-elliott-1-i}).}
  If \(f, g\in S\) are such that
  \(f\neq g\), then there is \(y\in X\) with \((y)f\neq (y)g\) or \(\dom(f) \neq
  \dom(g)\). In the first case, \(\{s\in S\colon (y)s=(y)f\}\) and \(\{s\in
  S:(y)s=(y)g\}\) are the required disjoint neighbourhoods of \(f\) and \(g\),
  respectively. In the second case, we may assume that there is \(y \in \dom(f)
  \setminus \dom(g)\), and so the sets \(\{s \in S \colon y \in \dom(s)\} =
  \bigcup_{z \in X} \{s \in S \colon (y,z) \in f\}\) and \(\{ s \in S
  \colon y \notin \dom(s)\}\) are disjoint neighbourhoods of \(f\) and \(g\).
  \medskip

  \noindent \textbf{(\ref{thm-elliott-1-iii}) $\Rightarrow$ (\ref{thm-elliott-1-iv})} This follows from the equalities
  \[
    S\backslash \{f\in S\colon (y, z)\in f\} = \bigcup_{t\in
      X\backslash \{z\}} \{f\in S:(y, t)\in f\} \cup \{f\in S: y\notin
    \dom(f)\}
 \]
  and
  \[S\backslash \{f\in S:y\notin \dom(f)\} = \bigcup_{t\in X} \{f\in S:(y, t)\in
    f\}.\]
  \medskip

  \noindent \textbf{(\ref{thm-elliott-1-iv}) $\Rightarrow$ (\ref{thm-elliott-1-ii}).} Let \(g\in S\) be arbitrary. From
  the equality
  \[\{g\}= \bigcap_{y\in \dom(g)}\{f\in S:(y,(y)g) \in f\} \cap
    \bigcap_{y\in X\backslash \dom(g)} \{f\in S: y\notin \dom(f)\}.\]
  it follows that the singleton sets are closed in \(\T\).
\end{proof}

\begin{cor}
  \label{cor-elliott-1.012}
  Let $X$ be an infinite set, and let $S$ be a subsemigroup of $X ^ X$ such
  that
  $S$ contains all of the constant  transformations, and for every
  $x\in X$ there exists $f_{x}\in S$ such that $(x) f_{x} ^ {-1} = \{x\}$ and
  $(X)f_{x}$
  is finite. Then the pointwise topology coincides with the Hausdorff-Markov,
  Fr\'{e}chet-Markov, and Zariski topologies for \(S\).
\end{cor}
\begin{proof}
  Since the pointwise topology is a Hausdorff semigroup topology for $S$, the
  Hausdorff-Markov topology is contained in the pointwise topology.
  By \cref{figure-the-only}, the Zariski and Fr\'echet-Markov topologies
  are contained in the Hausdorff-Markov topology, and hence in the pointwise
  topology.

  On the other hand,
  \cref{thm-elliott-1} implies that the minimum \(T_1\) topology that is
  semitopological for $S$ contains the pointwise topology, and hence, by
  \cref{figure-the-only}, the Hausdorff-Markov, Fr\'echet-Markov, and
  Zariski topologies contain the pointwise topology too.
\end{proof}

\begin{lem}\label{lem-zariski-pointwise}
  Let \(X\) be an infinite set and let \(S\) be a subsemigroup of \(X^X\) such
  that for every \(x \in X\) there exist \( a, b, c_{0}, \ldots, c_{n-1} \in S \)
  for some $n\in \N$ such that the following hold:
  \begin{enumerate}[\rm (i)]
    \item\label{lem-zariski-pointwise-i} 
    \((y)a = (y)b\) if and only if \(y \neq x\);

    \item\label{lem-zariski-pointwise-ii} 
          \(x \in \im(c_{i})\) for all \(i\);

    \item\label{lem-zariski-pointwise-iii} 
          for every \(s \in S\) and every \(y \in X\setminus \{(x)s\}\) there
          is \(i \in \{0, \ldots, n-1\}\) so that \( \im(c_{i}) \cap (y)s^{-1} =
          \varnothing\).
  \end{enumerate}
  Then the Zariski topology of \(S\) is the pointwise topology.
\end{lem}

\begin{proof}
  The pointwise topology is Hausdorff, and so it contains the Zariski topology
  by \cref{prop-luke-1}. Hence it remains to show that the Zariski topology
  contains the pointwise topology.

  We begin by observing that if $x\in X$ is arbitrary, then there exist $a,
    b\in S$ satisfying part \eqref{lem-zariski-pointwise-i} of the hypothesis and so
  \[
    S \setminus \{f \in S \colon x\in \im(f)\} = \{f \in S \colon  x \notin
    \im(f) \} =  \{f \in S \colon f a = f b\}.
  \]
  Hence \(\{f \in S \colon x\in \im(f)\}\) is open in the Zariski topology.

  If $x \in X$ is arbitrary, then  there exist $c_0, \ldots, c_{n-1}\in S$
  satisfying  part \eqref{lem-zariski-pointwise-ii} of the assumption of the lemma.
  Since the Zariski topology is semitopological, by \cref{prop-luke-2},  the
  map \(\lambda_{c_{i}}\) is continuous for every \(i \in \{0, \ldots, n-1\}\).
  If $y\in X$ is arbitrary, then
  the set
  \[
    U_i := \{f \in S \colon \text{ there is } w \in \im(c_{i}) \text{ with }
    (w)f = y\} =  \{f\in S \colon y \in \im(f)\} \lambda_{c_{i}}^{-1}
  \]
  is also open in Zariski topology for every \(i \in \{0, \ldots, n-1\}\). It
  suffices to show that \(\bigcap_{i = 0}^{n-1} U_i\) is equal to the subbasic
  open set \(\{f\in S \colon (x)f=y\}\) in the pointwise topology.

  By definition,
  \[
    \bigcap_{i=0}^{n-1} U_i = \{f \in S \colon \text{ there exist } w_i \in
    \im(c_{ i}) \text{ with } (w_0)f=\cdots =(w_{n-1})f=y\}.
  \]
  Let \(f \in  \bigcap_{i=0}^{n-1} U_i\), and let \(w_0, \ldots, w_{n - 1}\) be
  such that \( (w_0)f = \ldots = (w_{n - 1})f = y\) where \(w_i \in \im(c_{i})\)
  for all \(i\). Then \( \im(c_{i}) \cap (y)f^{-1} \neq \varnothing\) for every
  \(i \in \{0, \ldots, n-1\}\), and so by the assumption of the lemma, it follows
  that \(y = (x)f\). Hence \(\bigcap_{i=0}^{n-1} U_i \subseteq \{f \in S \colon
  (x)f = y\}\). The converse follows immediately from the fact that \(x \in
  \bigcap_{i = 0}^{n -1} \im(c_{i})\).
\end{proof}

In the next theorem, we show that the pointwise topology on $X ^ X$ is the
unique Polish (even $T_1$ second countable) semigroup topology on $X ^ X$; and
$X ^ X$ with pointwise topology
has automatic continuity with respect to the class of second countable
topological semigroups.

\begin{theorem}\label{propx-fulltrans}
  If $X$ is an infinite set, then the following hold:
  \begin{enumerate}[\rm (i)]
    \item \label{propx-fulltrans_i}
          $X^X$ with the pointwise topology has property \textbf{X} with
          respect to $\Sym(X)$;
    \item\label{propx-fulltrans_ii}
          the pointwise topology is the only $T_1$
          semigroup topology on $X^X$ which induces the pointwise topology on
          $\Sym(X)$;
    \item\label{propx-fulltrans_iii}  the pointwise topology coincides with the
          Hausdorff-Markov, Fr\'echet-Markov, and Zariski topologies for $X^X$;
    \item\label{propx-fulltrans_iv} if $X$ is countable, then $X^X$ with the
          pointwise topology
          has automatic continuity with respect to the class of second
          countable topological semigroups;
    \item\label{propx-fulltrans_v} if $X$ is countable, then the pointwise
          topology on $X ^ X$ is the unique
          $T_1$ second countable semigroup topology.
  \end{enumerate}
\end{theorem}
\begin{proof}
  \noindent\textbf{(\ref{propx-fulltrans_i}).}
  Let $\psi: X \to X \times X$ be a bijection, and let $\pi_1: X\times X \to X$
  be defined by $(x, y)\pi_1 = y$.  We define $f, g\in X^X$ by \[(x)f
    = (x, x)\psi^{-1}\quad\text{and}\quad (x)g = (x)\psi\pi_1.\]
  Let $s\in X^X$
  and let $t\in \Sym(X \times X)$ be any permutation such that \[(x, x)t = (x,
    (x)s)\] for all $x\in X$.  We then define $t_s\in X^X$ to be $\psi
    t\psi^{-1}$. From the definitions of $f$, $g$, and $t_s$,
  \[
    (x)ft_sg = (x, x)\psi^{-1}\psi t\psi^{-1} \psi \pi_1
    = (x, x) t\pi_1
    = (x, (x)s) \pi_1
    = (x)s
  \]
  for all $x\in X$.
  Let $B$ be a basic open neighbourhood of $t_s$. Then there exist $x_0, x_1,
    \ldots, x_n \in X$ such that $t_s\in \bigcap_{i = 0} ^ {n - 1} \{f\in
    X^X:(x_i)f=(x_i)t_{s}\}
    = B$.
  If
  \[
    V:= \{k\in X^X: (x)k = (x)s\ \text{for all } x \text{ such that } (x,
    x)\in \{(x_0)\psi, (x_1)\psi, \ldots, (x_{n-1})\psi\}\},\]
  then $V$ is an open neighbourhood of $s$ and so it suffices to show that
  $V\subseteq f(B\cap \Sym(X))g$.
  For every $x\in X$, we choose distinct $z_x\in X$ which is also distinct from
  the first coordinate of every element of $\{(x_0)\psi t, (x_1)\psi t, \ldots,
    (x_{n-1})\psi t\}$. Let $k\in V$. Then there exists $p\in \Sym(X \times X)$
  such that
  \[(x_i)\psi p = (x_i)\psi t\]
  for all $i$,	and
  \[(x, x)p = (z_x, (x)k),\]
  for all $x\in X$ for which \((x, x)\notin \{(x_0)\psi, (x_1)\psi, \ldots,
  (x_{n-1})\psi\}\).
  We then define $t_k$ to be $\psi p\psi^{-1}$. By the choice of
  $p$, it follows that $(x_i)t_k = (x_i)\psi p\psi ^ {-1} = (x_i)\psi t\psi ^
    {-1} = (x_i)t_s$ for every $i$, and so $t_k\in B \cap \Sym(X)$. Thus
  \[
    (x)ft_kg = (x, x)\psi^{-1}\psi p\psi^{-1} \psi \pi_1 = (x, x) p\pi_1 =
    \begin{cases}
      (x, x)t \pi_1     & \text{if } (x, x)\in \{(x_0)\psi, (x_1)\psi, \ldots,
      (x_{n-1})\psi\}                                                          \\
      (z_x, (x)k) \pi_1 & \text{otherwise.}
    \end{cases}
  \]
  In the first case, $(x)ft_kg = (x, (x)s)\pi_1= (x)s = (x)k$ and in the second
  case,
  $(x)ft_kg = (z_x, (x)k) \pi_1   = (x)k$ also.
  Hence $k= ft_kg\in f(B\cap \Sym(X))g$ and so, since $k$ was arbitrary,
  $V \subseteq	f(B\cap \Sym(X))g$, as required.
  \medskip

  \noindent\textbf{(\ref{propx-fulltrans_ii}).}
  By \cref{thm-elliott-1} every $T_1$ topology on $X^X$ contains the
  pointwise topology. By part (\ref{propx-fulltrans_i}) together with
  \cref{lem-luke-0}\eqref{lem-luke-0-ii}, this is also the largest topology which induces
  the pointwise topology on $\Sym(X)$.
  \medskip

  \noindent\textbf{(\ref{propx-fulltrans_iii}).}
  This follows immediately from \cref{cor-elliott-1.012}.
  \medskip

  \noindent\textbf{(\ref{propx-fulltrans_iv}).}
  This follows from part \eqref{propx-fulltrans_i}, \cref{lem-luke-0}\eqref{lem-luke-0-iv},
  and the automatic continuity of $\Sym(X)$ when $X$ is countable.
  \medskip

  \noindent\textbf{(\ref{propx-fulltrans_v}).}
  By part (\ref{propx-fulltrans_iii}), every second countable semigroup
  topology
  for $X^X$ is contained in the pointwise topology. By
  \cref{thm-elliott-1} every $T_1$ topology on $\N^\N$
  contains the pointwise topology.
\end{proof}

In some sense, every group is a group of permutations, since every group can be
embedded into the symmetric group $\Sym(X)$ for some set $X$. Similarly, every
semigroup embeds as a subsemigroup of $X ^ X$ for some set $X$. Given that $X ^
  X$ has a canonical semigroup topology, it is natural to ask which topological
semigroups embed in $X ^ X$ with the pointwise topology.

It is well-known that a $T_1$ topological group $G$ is topologically isomorphic
to a subgroup of $\Sym(\N)$ if and only if $G$ has a countable neighbourhood
basis of the identity $1_G$ consisting of countable index open subgroups. This
latter condition is referred to as $G$ being \textit{non-archimedean} by some
authors; see, for example,~\cite[Theorem 1]{Bodirsky2017ab}.
In the next theorem, we prove an analogue of this result
for $T_0$ left semitopological semigroups. In~\cite[Theorem 2]{Bodirsky2017ab},
another analogous characterisation is given of those topological monoids that
are topologically isomorphic to closed submonoids of $\N ^ \N$.
There are many topological monoids that are topologically isomorphic to
non-closed submonoids of $\N ^ \N$. For example,
if $D$ is a countable dense subset of $\N ^ \N$, then the least submonoid
$\genset{D}$ of $\N ^ \N$ containing $D$ is countable and dense also. It
follows that
$\genset{D}$ is not $G_{\delta}$, and hence not Polish. It follows that
$\genset{D}$ is not topologically isomorphic to any closed submonoid of $\N ^
  \N$, but it is obviously (topologically isomorphic to) a submonoid of $\N ^
  \N$.
Of course, not every Polish semigroup is topologically isomorphic to a
subsemigroup of $\N ^ \N$. For example, the reals $\R$
under addition form a connected Polish group, but $\N ^ \N$ is totally
disconnected, and so $(\R, +)$ cannot be topology embedded in $\N ^ \N$.

Even though the statement of \cref{theorem-luke-1} is superficially
different from that of~\cite[Theorem 2]{Bodirsky2017ab}, the proof of
\cref{theorem-luke-1} is essentially contained in the proof
of~\cite[Theorem 2]{Bodirsky2017ab}. We include \cref{theorem-luke-1}
because we will use it in \cref{subsection-inj-surj} and we include the
proof for the sake of completeness.

\begin{theorem}[cf. Theorem 2 in \cite{Bodirsky2017ab}]
  \label{theorem-luke-1}
  If $S$ is a $T_0$ left semitopological semigroup,
  then the following are equivalent:
  \begin{enumerate}[\rm (i)]
    \item \label{theorem-luke-1-i}
          there is a sequence $\set{\rho_i}{i\in \N}$ of right congruences of
          $S$,
          each having countably many classes, such that $\set{m/\rho_i}{m\in
              S,\
              i\in \N}$ is a subbasis for $S$;
    \item \label{theorem-luke-1-ii}
          there is a sequence $\set{\sigma_i}{i\in \N}$ of right congruences of
          $S$,
          each having countably many classes, such that $\set{m/\sigma_i}{m\in
              S,\
              i\in \N}$ is a basis for $S$;
    \item \label{theorem-luke-1-iii}
          $S$ is topologically isomorphic to a subsemigroup of $\N ^ \N$ (with
          the
          pointwise topology and right actions).
  \end{enumerate}
\end{theorem}
\begin{proof}
  We will show that \eqref{theorem-luke-1-i} 
  $\Rightarrow$ \eqref{theorem-luke-1-iii}
  $\Rightarrow$ \eqref{theorem-luke-1-ii}
  $\Rightarrow$ \eqref{theorem-luke-1-i}.\medskip

  \noindent\textbf{(\ref{theorem-luke-1-ii}) $\Rightarrow$ (\ref{theorem-luke-1-i})}. 
    This implication follows immediately from the definitions of bases and subbases for topological spaces.\medskip

  \noindent\textbf{(\ref{theorem-luke-1-iii}) $\Rightarrow$ (\ref{theorem-luke-1-ii})}. Assume without loss of generality that
  $S$ is a subsemigroup of $\N ^ \N$. If we define $\rho_i = \set{(f, g)}{(i)f
      = (i)g}$ for every $i\in \N$, then clearly every $\rho_i$ is a right
  congruence
  with the properties given in \eqref{theorem-luke-1-i}.
      If $\sigma_n =\bigcap_{i\leq n}\rho_i$, for every $n\in\N$, then $\set{x/\sigma_n}{x\in S, n\in\N}$ is a basis for the topology on $S$ with the required properties.\medskip

  \noindent\textbf{(\ref{theorem-luke-1-i}) $\Rightarrow$ (\ref{theorem-luke-1-iii})}. We give $S^1$ the disjoint union topology of
  $S$ and $\{1\}$. If $\rho$ is any right congruence on $S$, then $\rho \cup
    \{(1, 1)\}$ is a right congruence on $S^1$. It follows that $S^1$ satisfies
  the
  hypothesis of \eqref{theorem-luke-1-i}. Let $\set{\rho_i}{i\in \N}$ be a sequence of right
  congruences of $S^1$, each having countably many classes, such that
  $\set{m/\rho_i}{m\in S^1,\ i\in \N}$ is a subbasis for $S^1$.

  By assumption, $X = \set{m/\rho_i}{m\in S^1,\ i\in \N}$ is countable,
  and we will show that $S^1$ is topologically isomorphic to a subsemigroup of
  $X ^ X$ with the pointwise topology, from which it will follow that
  $S$ is also topologically isomorphic to a subsemigroup of $\N^{\N}$.
  When $i\neq j$, we will consider classes of $\rho_i$ and classes of $\rho_j$
  as different elements of $X$, even if they should happen to be the same subset
  of $S$.
  We define $\phi: S^1 \to X ^ X$ such that $(m)\phi \in X ^
    X$ is defined by
  \[
    (n/\rho_i)(m)\phi = (nm)/\rho_i.
  \]
  It is sufficient to show that $\phi$ is a well-defined, injective,
  continuous, homomorphism such that if $U$ is open in $S^1$, then $(U)\phi$ is
  open in $(S^1)\phi$ with the subspace topology.

  If $(a, b)\in \rho_i$ for some $i\in \N$ and $m\in S^1$, then
  $(a/\rho_i)(m)\phi = (am)/\rho_i = (bm)/\rho_i = (b/\rho_i)(m)\phi$, since
  $\rho_i$ is a right congruence. It follows that $\phi$ is well-defined.

  Suppose that $m, n\in S^1$ are such that $(m)\phi = (n)\phi$. It follows that
  $m/\rho_i = (1 /\rho_i)(m)\phi = (1/\rho_i)(n)\phi = n/\rho_i$
  for all $i\in \N$. In other words, $(m, n)\in \rho_i$ for all $i\in \N$. But
  since $\set{m/\rho_i}{m\in S^1,\ i\in \N}$ is a subbasis for $S^1$, and $S^1$
  is
  $T_0$ it follows that $\bigcap_{i\in \N}\rho_i = \set{(s, s)}{s\in S ^ 1}$.
  Hence $m = n$
  and $\phi$ is injective.

  Let $a, m, n\in S^1$ and $i\in \N$ be arbitrary. Then
  \[
    ((a/\rho_i)(m)\phi)(n)\phi = (amn)/\rho_i = (a/\rho_i )(mn)\phi
  \]
  and so $\phi$ is a homomorphism.

  To show that $\phi$ is continuous it is sufficient to show that the preimage
  under $\phi$ of every subbasic open set in $X ^ X$ is open in $S^1$.
  Let $\mathcal{S}$ be the subbasis for $X ^ X$ consisting of the
  sets
  \[[\alpha, \beta] = \set{f\in X ^ X}{(\alpha)f = \beta}
  \]
  where $\alpha, \beta\in X$. Suppose that $[\alpha, \beta]\in
    \mathcal{S}$. Then $\alpha = a/\rho_i$ and $\beta = b/\rho_j$ for some $a,
    b\in S^1$ and $i, j\in \N$. If $[a/\rho_i, b/\rho_j] \cap (S^1)\phi
    \not=\varnothing$, then $i = j$ by the definition of $\phi$. Hence we may
  suppose that without loss of generality that $\alpha = a/\rho_i$ and $\beta =
    b/\rho_i$.	It follows that
  \begin{eqnarray*}
    [\alpha, \beta] \phi ^ {-1} & = & \set{m\in S^1}{(\alpha)(m)\phi = \beta}\\
    & = & \set{m\in S^1}{(a/\rho_i)(m)\phi =
      b/\rho_i}\\
    & = & \set{m\in S^1}{(am)/\rho_i = b/\rho_i} \\
    & = & \set{m\in S^1}{am \in b/\rho_i} \\
    & = & (b / \rho_i) \lambda_a ^ {-1}.
  \end{eqnarray*}
  Since $S^1$ is left semitopological, $\lambda_a$ is continuous and since $b /
    \rho_i$ is open, so too is $[\alpha, \beta] \phi ^ {-1}$. Hence $\phi$ is
  continuous.

  If $m \in S^1$ and $i\in \N$ are arbitrary, then
  \begin{eqnarray*}
    (m/\rho_i)\phi & = & \set{(n)\phi}{n\in S^1,\ m\rho_i = n\rho_i} \\
    & = & \set{(n)\phi}{n\in S^1,\ (1 / \rho_i)(n)\phi =
      m\rho_i}\\
    & = & [1 / \rho_i, m\rho_i].
  \end{eqnarray*}
  Hence every subbasic open set in $S^1$ is mapped to a subbasic open set in
  $(S^1)\phi$ and so $\phi$ is open.
\end{proof}

We prove an analogue of \cref{theorem-luke-1} for semitopological inverse
monoids
and the symmetric inverse monoid in \cref{thm-inverse-submonoids}.

Uspenski\u{\i}'s Theorem~\cite[Theorem
  9.18]{Kechris1995aa} and~\cite{Uspenskii1986aa} states that every Polish
group
is isomorphic to a (necessarily closed) subgroup of the group $H([0, 1] ^ \N)$
of
homeomorphisms of the Hilbert cube $[0, 1] ^ \N$.
A similar result holds for separable metrizable compactifiable semigroups,
every such semigroup is topologically isomorphic to a subsemigroup of $C([0, 1]
  ^ \N)$;
see~\cite[Theorem 5.2]{Megrelishvili2007aa}.

\begin{question}
  Is every countable Polish semigroup
  topologically isomorphic to a subsemigroup of $\N ^ \N$?
  \footnote{Addendum: S. Bardyla, L. Elliott, J. D. Mitchell, and Y. P\'eresse recently showed that the answer to this question is no, see~\cite{Bardyla2023}.}
\end{question}


\subsection{The full binary relation monoid}\label{subsection-binary}

Unlike $X ^ X$ which is probably better known as a topological space than as a
semigroup,  there is no obvious candidate for a semigroup topology on the full
binary relation monoid $B_X$.
In this section, we prove that a dichotomy exists for semigroup topologies on $B_X$:
either such a topology is so coarse that it is not $T_1$, or it is so fine that
it is not second countable.
In particular, $B_X$ possesses no Polish semigroup topologies.
We exhibit two semigroup topologies on $B_X$ which are, in some sense,
canonical for the poles of the dichotomy.

The natural starting point is to consider a topology which induces the
pointwise topology on $X ^ X$. This is the subject of the next theorem.

\begin{theorem} \label{finest_inducer_relO}
  Let $X$ be an infinite set and let
  \[
    U_{x, y} = \set{f \in B_X}{(x,y) \in f}
  \]
  for all $x, y\in X$. If $\mathcal{B}_1$ is the topology on $B_X$ with
  subbasis
  $\set{U_{x, y}}{x,y\in X}$, then the following hold:
  \begin{enumerate}[\rm (i)]

    \item\label{thm-item-1}
          $\mathcal{B}_1$ is a semigroup topology for $B_X$ and inversion of
          binary relations $f
            \mapsto f^{-1}$ is continuous;

    \item\label{finest_inducer_rel0-ii}
          $\mathcal{B}_1$ is $T_0$ but not $T_1$, and
          the subspace topology induced by $\mathcal{B}_1$ on the symmetric
          inverse
          monoid $I_X$ is not $T_1$;

    \item\label{finest_inducer_rel0-iii}
          the topological semigroup $B_X$ with the topology $\mathcal{B}_1$ has
          property \textbf{X}
          with respect to $\Sym(X)$;

    \item\label{finest_inducer_rel0-iv}
          every topology that is semitopological for $B_X$ and
          that induces the pointwise topology on $\Sym(X)$ is contained in
          $\mathcal{B}_1$;

    \item\label{finest_inducer_rel0-v}
          if $X$ is countable, then $B_X$ with the topology $\mathcal{B}_1$ has
          automatic continuity with respect to
          the class of second countable topological semigroups.
  \end{enumerate}
\end{theorem}
\begin{proof}
  Note that a basis for $\mathcal{B}_1$ is given by the collection
  of sets $\set{h\in B_X}{f \subseteq h}$ where $f\in B_X$ is finite, since
  such sets are precisely the finite intersections of subbasic sets in
  $\mathcal{B}_1$.
  \medskip

  \noindent
  \textbf{(\ref{thm-item-1}).} Let $x,y \in X$ and $f,g \in B_X$ be such that
  $fg\in
    U_{x,y}$, i.e.  $(x,y) \in fg$. Then there exists $z \in X$ such that $(x,
    z) \in f$ and $(z, y) \in g$. Conversely, for every $f',g'\in B_X$ with
  $(x,z) \in f'$ and $(z,y) \in g'$ we have that $(x,y) \in f'g'$. In other
  words, $f \in U_{x,z}$, $g\in U_{z,y}$ and
  $$U_{x,z}U_{z,y} \subseteq U_{x,y}.$$
  Hence multiplication is continuous.

  The map $f \mapsto f^{-1}$ is a homeomorphism since
  $\left(U_{x,y}\right)^{-1}=U_{y,x}$.
  \medskip

  \noindent\textbf{(\ref{finest_inducer_rel0-ii}).} 
  If $f, g\in B_X$ are distinct, then either $f
    \not \subseteq g$ or $g \not \subseteq f$. Without loss of generality
  assume that
  $f \not \subseteq g$. Then there exists $(x,y) \in X \times X$ such that
  $(x,y) \in f \setminus g$ and so $f \in U_{x,y}$ but $g \not \in
    U_{x,y}$. Hence $\mathcal{B}_1$ is $T_0$.

  On the other hand, suppose that $f,g \in I_X$ and $g \subsetneq f$. Then
  every subbasic open set $U_{x,y}$ containing $g$ also contains $f$. Thus
  every open set containing $g$ contains $f$ and so $\mathcal{B}_1$ is not
  $T_1$ (even
  when restricted to $I_X$).\medskip

  \noindent\textbf{(\ref{finest_inducer_rel0-iii}).} 
  Let $Y \subseteq X$ be such that $|Y|=|X \setminus Y|=|X|$.  Enumerate $X$ as
  $X=\set{x_i}{i\in |X|}$ and let $\set{X_i}{i\in |X|}$ be a partition of $X
    \setminus Y$ such that $|X_i|=|X|$ for every $i \in |X|$.  Define $f\in
    B_X$
  by
  $$f=\set{(x_i, y)\in X \times X}{i \in |X| \text{ and }y \in X_i}.$$
  Let $s\in B_X$ be arbitrary. A binary relation $t\in B_X$ satisfies
  $ftf^{-1}=s$ if and only if $t$ has the following property:
  \begin{equation}\label{relO_generator}
    \text{for all }i,j \in |X|: \quad \left(X_i \times X_j\right) \cap t \neq
    \varnothing \quad \text{if and only if}\quad (x_i,x_j)\in s.
  \end{equation}
  Since $|Y|=|X_i|=|X|$, there exists $t_s\in \Sym(X)$ satisfying
  \eqref{relO_generator}.

  Suppose that $V$ is a basic open neighbourhood of $t_s$. Then there exist
  finite $k\in B_X$ such that $V = \set{h \in B_X}{k \subseteq h}$.
  If $U := \set{h \in B_X}{fkf^{-1} \subseteq h}$, then $U$ is open since
  $fkf^{-1}$ is finite. We show that $s\in U \subseteq fVf^{-1}$. As
  $k\subseteq t_s$ and $ft_sf^{-1}=s$ it is clear that $s\in U$. It
  remains to show that
  \[U\subseteq fVf^{-1}.\]
  Let $u\in U$ be arbitrary. As in \eqref{relO_generator} we need only find
  $t\in V \cap \Sym(X)$ with
  \begin{equation*}
    \text{for all }i,j \in |X|: \quad \left(X_i \times X_j\right) \cap t \neq
    \varnothing \quad \text{if and only if}\quad (x_i,x_j)\in u.
  \end{equation*}
  If $i, j\in |X|$ and $(X_i\times X_j)\cap k \neq \varnothing$, then
  $(x_i, x_j)\in fkf^{-1}\subseteq u$. Therefore, as $k$ is finite, we can
  extend $k$ to an element $t$ of $V\cap \Sym(X)$ with the desired
  property. So $u=ftf^{-1}\in f(V\cap \Sym(X))f^{-1}$ and
  \[U\subseteq	f(V\cap \Sym(X))f^{-1},\]
  as required.
  \medskip

  \noindent\textbf{(\ref{finest_inducer_rel0-iv}).} 
  The subbasis for $\mathcal{B}_1$ induces the usual subbasis
  for the
  pointwise topology on $\Sym(X)$, and so the topology on $\Sym(X)$ induced by
  $\mathcal{B}_1$ is the pointwise topology. Since $B_X$ has property
  \textbf{X} with
  respect to $\Sym(X)$, by part \eqref{finest_inducer_rel0-iii}, it follows from
  \cref{lem-luke-0}\eqref{lem-luke-0-ii} that if $\T$ is a topology that is semitopological
  for $B_X$, then $\T$ is contained in $\mathcal{B}_1$.
  \medskip

  \noindent\textbf{(\ref{finest_inducer_rel0-v}).} 
  This follows from part \eqref{finest_inducer_rel0-iii}, \cref{lem-luke-0}\eqref{lem-luke-0-iv}, and
  the automatic continuity of $\Sym(X)$.
\end{proof}

As an immediate consequence of \cref{finest_inducer_relO}\eqref{finest_inducer_rel0-iii}, there is no
$T_1$ topology that is semitopological for $B_X$ and that induces the pointwise
topology on $\Sym(X)$.
If instead of trying to extend the pointwise topology, as we did in
\cref{finest_inducer_relO}, we look for the weakest
$T_1$ topology that is semitopological for $B_X$, then we obtain the following
theorem.

\begin{theorem}\label{coarsest_T1_relO}
  Let $X$ be an infinite set and let $\mathcal{B}_2$ be the topology on $B_X$
  generated by
  the sets
  \[
    U_{x, y} = \set{h\in B_X}{(x,y) \in h} \quad \text{and}\quad
    V_{Y, Z} = \set{h \in B_X}{(Y)h \subseteq Z}
  \]
  for all $x,y\in X$ and $Y, Z\subseteq X$. Then the following hold:
  \begin{enumerate}[\rm (i)]

    \item\label{coarsest_T1_rel0-i}
          $\mathcal{B}_2$ is a Hausdorff semigroup topology for $B_X$ and
          inversion
          of binary relations $f \mapsto f^{-1}$ is continuous;

    \item\label{coarsest_T1_rel0-ii}
          every $T_1$ topology that is semitopological for $B_X$ contains
          $\mathcal{B}_2$;

    \item\label{coarsest_T1_rel0-iii}
          $\mathcal{B}_2$ is not contained in any second countable topology;

    \item\label{coarsest_T1_rel0-iv}
          $\mathcal{B}_2$ strictly contains $\mathcal{B}_1$;

    \item\label{coarsest_T1_rel0-v}
          if $X$ is countable, then $\mathcal{B}_2$ coincides with the
          Fr\'echet-Markov,
          Hausdorff-Markov, and Zariski topologies for $B_X$.
  \end{enumerate}
\end{theorem}
\begin{proof}
  \textbf{(\ref{coarsest_T1_rel0-i}).}
  If $f,g \in B_X$ are distinct, then without loss of generality, there exist
  $x,y \in X$ such that $(x,y)\in f$ but $(x,y)\not \in g$. Then $f\in
    U_{x,y}$ and $g\in \mft{\{x\}}{X \setminus \{y\}}$. Since these two open
  sets
  are disjoint, $\mathcal{B}_2$ is Hausdorff.

  The topology generated by the sets $U_{x,y}$ is just $\mathcal{B}_1$ as
  defined in
  \cref{finest_inducer_relO}. Since $\mathcal{B}_1$ is a semigroup topology,
  it
  suffices to show that the pre-images under multiplication and inversion of
  the subbasic sets $\mft{Y}{Z}$ are open.

  For any $Y \subseteq X$ and $f,g \in B_X$, we clearly have $f \in
    \mft{Y}{(Y)f}$ and $g \in \mft{(Y)f}{(Y)fg}$. Moreover, if $f'\in
    \mft{Y}{(Y)f}$ and $g' \in \mft{(Y)f}{(Y)fg}$, then $(Y)f'g' \subseteq
    (Y)fg'\subseteq (Y)fg$. Thus if $fg \in \mft{Y}{Z}$, then $f'g' \in
    \mft{Y}{Z}$ and so multiplication is continuous.

  Since
  $$\left(\mft{Y}{Z}\right)^{-1}=\mft{X\setminus Z}{X \setminus Y}$$
  it follows that $f \mapsto f^{-1}$ is also continuous under $\mathcal{B}_2$.\medskip

  \noindent\textbf{(\ref{coarsest_T1_rel0-ii}).}
  Suppose that $\T$ is a $T_1$ topology under which $B_X$ is
  semitopological. We will show that the subbasic open sets of $\mathcal{B}_2$
  are open in $\T$ and so, in particular, $\mathcal{B}_2\subseteq
    \T$.

  For any $Y, Z \subseteq X$ and $f\in B_X$ consider the product
  $$(Y \times Y)\circ f\circ ((X \setminus Z) \times (X \setminus Z))=g.$$
  Either $(Y)f \subseteq Z$ and $g = \varnothing$, or
  $(Y)f \not \subseteq Z$ and $g=Y \times (X \setminus Z)$.
  Thus
  $$\left(\{Y \times (X \setminus Z)\}\right)\left(\lambda_{Y\times Y} \circ
    \rho_{(X \setminus Z)\times (X \setminus Z)}\right)^{-1}=\set{f \in
      B_X}{(Y)f
      \not \subseteq  Z} = B_X\setminus V_{Y, Z}$$
  and
  $$\left(\{\varnothing\}\right)\left(\lambda_{Y\times Y} \circ \rho_{(X
      \setminus Z)\times (X \setminus Z)}\right)^{-1}=\set{f \in B_X}{(Y)f
      \subseteq Z}=\mft{Y}{Z}$$
  as the continuous pre-images of finite sets, are closed in $\T$.
  Hence
  $V_{Y, Z}$ and $B_X\setminus V_{Y, Z}$ are open in $\T$.
  If $x,y \in X$, we may let $Y=\{x\}$ and $Z=X
    \setminus \{y\}$, and then
  $$B_X\setminus V_{Y, Z} = \set{f \in B_X}{(Y)f \not \subseteq  Z}=\set{f\in
      B_X}{(x,y)\in f}=U_{x,y}$$
  and so $U_{x,y}$ is clopen in $\T$ also.\medskip

  \noindent\textbf{(\ref{coarsest_T1_rel0-iii}).}
  If $X$ is uncountable, then the collection
  $$\set{\mft{X}{\{x\}}\cap U_{x,x}}{x\in X} \subseteq \mathcal{B}_2$$
  consists of pairwise disjoint open sets and is uncountable. Hence, if $X$ is
  uncountable, then $\mathcal{B}_2$ is not contained in any second countable
  topology.

  Suppose that $X$ is countable and that $\set{X_i}{i\in I}$ is a family of
  subsets of $X$ with cardinality $2 ^ {\aleph_0}$ such that $X_i\not\subseteq
    X_j$ whenever $i \neq j$.  For every $i\in I$, choose $f_i\in B_X$ such
  that
  $(X)f_i=X_i$. If $\mathcal{U}$ is a basis of any topology containing
  $\mathcal{B}_2$, then, for every $i\in I$, there exists $U_i \in \mathcal{U}$
  such that $f_i \in U_i \subseteq \mft{X}{X_i}$. If $f_i \in U_j\subseteq
    \mft{X}{X_j}$, then $X_i=(X)f_i \subseteq X_j$ and so $i=j$.  In other
  words,
  $i\neq j$ implies that $f_i \not \in U_j$. Thus $|\mathcal{U}|\geq
    |I|=2^{\aleph_0}$ and so no topology containing $\mathcal{B}_2$ is second
  countable. \medskip

  \noindent\textbf{(\ref{coarsest_T1_rel0-iv}).}
  The subspace topology on $\Sym(X)$ induced by $\mathcal{B}_1$ is the
  pointwise topology. On the other hand, if $Y$ and $Z$ are two infinite
  sets with infinite complements, then $V_{Y, Z}\cap \Sym(X)$ is not open in
  the pointwise
  topology on $\Sym(X)$. \medskip

  \noindent\textbf{(\ref{coarsest_T1_rel0-v}).}
  Since the minimal $T_1$ topology that
  is semitopological for $B_X$ coincides with the Hausdorff-Markov topology, by
  parts \eqref{coarsest_T1_rel0-i} and \eqref{coarsest_T1_rel0-ii}, it follows that the Fr\'echet-Markov, Hausdorff-Markov
  and Zariski topologies all equal $\mathcal{B}_2$; see
  \cref{figure-the-only}.
\end{proof}

We obtain the following corollary to \cref{coarsest_T1_relO}\eqref{coarsest_T1_rel0-iii}.

\begin{cor}\label{cor-no-polish-binrel}
  Let $X$ be an infinite set. Then no second countable $T_1$ topology is
  semitopological for $B_X$.  In particular, $B_X$ possesses no Polish
  semigroup topologies.
\end{cor}

\subsection{The partial transformation monoid}
\label{subsection-partial}

A natural way of defining a semigroup topology on the partial transformation
monoid $P_X$, where $X$ is an arbitrary set, is to embed $P_X$ into the
full transformation monoid $X ^ X$, and use the subspace topology induced by
the pointwise topology on $X ^ X$.
We will show that, when $X$ is infinite, this topology is simultaneously
the weakest $T_1$ semigroup topology on $P_X$ and the finest extension of the
pointwise topology of $\Sym(X)$ to $P_X$.

Roughly speaking, the natural way of embedding $P_X$ into $X ^ X$
is to add a new element $\na$ to $X$ that will represent ``not
defined". More precisely, if $X$ is a set, $\na \not \in X$, and $Y=X \cup
  \{\na\}$, then the function $\phi : P_X \to Y ^ Y$ defined by
\begin{equation}\label{natural_embedding}
  (x)((f)\phi) =
  \begin{cases}
    (x)f & \text{if } x \in \dom(f)       \\
    \na  & \text {if } x \not \in \dom(f)
  \end{cases}
\end{equation}
is an embedding.  Note that, in particular, if $g = (f)\phi$, then $(\na)g=
  \na$.  We will refer to $\phi$ as the \emph{natural embedding} of $P_X$ into
$Y
  ^ Y$.

\begin{theorem}\label{rdom_theorem}
  Let $X$ be an infinite set and let $\mathcal{P}$ be the topology on $P_X$
  generated by the sets
  \[
    U_{x, y} = \set{h\in P_X}{(x,y)\in h}
    \quad\text{and}\quad
    W_x = \set{h \in P_X}{x \not \in \dom(h)}
  \]
  for all $x,y \in X$. Then the following hold:
  \begin{enumerate}[\rm (i)]

    \item\label{rdom_theorem_i}
          the topology $\mathcal{P}$ is the subspace topology on $P_X$ induced
          by the
          pointwise topology on ${Y}^{Y}$ and the natural embedding $\phi:
            P_X\to Y
            ^ Y$ defined in \eqref{natural_embedding};

    \item\label{rdom_theorem_ii}
          $\mathcal{P}$ is a Hausdorff semigroup topology for $P_X$;

    \item\label{rdom_theorem_iv}
          $\mathcal{P}$ coincides with the Hausdorff-Markov, Fr\'echet-Markov, and Zariski
          topologies for $P_X$;

    \item\label{rdom_theorem_v}
          $P_X$ has property \textbf{X} with respect to $\mathcal{P}$ and
          $\Sym(X)$;

    \item\label{rdom_theorem_vi}
          if $\T$ is a topology that is semitopological for $P_X$ and $\T$
          induces
          the pointwise topology on $\Sym(X)$, then $\T$ is contained in
          $\mathcal{P}$;

    \item\label{rdom_theorem_viii}
          the topology $\mathcal{P}$ is the unique $T_1$ topology that induces
          the
          pointwise topology on $\Sym(X)$ and that is semitopological for
          $P_X$;

    \item\label{rdom_theorem_vii}
          if $X$ is countable, then $P_X$ with the topology $\P$ has automatic
          continuity with
          respect to the class of second countable topological semigroups;

    \item\label{rdom_theorem_ix}
          if $X$ is countable, then $\mathcal{P}$ is the unique Polish topology
          that is semitopological for $P_X$;

    \item\label{rdom_theorem_x}
          if $X$ is countable, then $\mathcal{P}$ is the unique $T_1$ second
          countable semigroup topology for $P_X$.
  \end{enumerate}
\end{theorem}

\begin{proof}
  \textbf{(\ref{rdom_theorem_i}).}
  The image of $P_X$ under the natural embedding $\phi$ defined in
  \eqref{natural_embedding} is the set
  $$\set{h\in Y^{Y}}{(\na)h=\na}.$$
  The pointwise topology on $Y^{Y}$ is generated by the sets $A_{x,y} =\set{f
      \in Y^{Y}}{(x,y)\in f}$ for all $x,y\in Y$.  Hence the topology induced
  on
  $(P_X)\phi$ is generated by the sets $A_{x,y} \cap (P_X)\phi$.  Note that
  $A_{\na, \na}\cap (P_X)\phi=(P_X)\phi$ and $A_{\na, y}\cap
    (P_X)\phi=\varnothing$ for all $y\in X = Y\setminus\{\na\}$. If $x, y \in
    X$,
  then
  \[
    A_{x,y}\cap (P_X)\phi=\set{h \in (P_X)\phi}{(x,y) \in h} = (U_{x, y})\phi
  \]
  and
  \[
    A_{x,\na}\cap (P_X)\phi=\set{h \in (P_X)\phi}{(x,\na) \in h} = (W_x)\phi.
  \]
  Hence $\phi$ is a homeomorphism between $\mathcal{P}$ and the topology
  generated by $A_{x, y} \cap (P_X)\phi$.
  \medskip

  \noindent
  \textbf{(\ref{rdom_theorem_ii}).}
  Since $Y ^ Y$ is a Hausdorff topological semigroup under the pointwise
  topology, it follows from part \eqref{rdom_theorem_i} that $P_X$ is a
  Hausdorff topological
  semigroup under $\mathcal{P}$.
  \medskip

  \noindent
  \textbf{(\ref{rdom_theorem_iv}).}
  The proof of this part is similar to the proof of \cref{cor-elliott-1.012},
  it is included for the sake of completeness.
  By part \eqref{rdom_theorem_ii}, $\P$ is a Hausdorff semigroup topology for
  $P_X$. Hence the Hausdorff-Markov topology is contained in $\P$.
  By \cref{figure-the-only}, the Zariski and Fr\'echet-Markov topologies
  are contained in the Hausdorff-Markov topology, and hence in $\P$.

  On the other hand,
  \cref{thm-elliott-1} implies that every \(T_1\) topology that is
  semitopological for $P_X$ contains $\P$, and such topologies include the
  Hausdorff-Markov, Fr\'echet-Markov, and Zariski topologies.
  \medskip

  \noindent
  \textbf{(\ref{rdom_theorem_v}).}
  We will show that \(P_X\) has property \textbf{X} with respect to \(X^X\),
  from which it will follow, by \cref{lem-prop-X-transitive} and
  \cref{propx-fulltrans}(\ref{propx-fulltrans_i}),
  that \(P_X\) has property \textbf{X} with respect to \(\Sym(X)\).
  Let \(b\in X\) be fixed and let \(\psi: X\to X\backslash\{b\}\) be a
  bijection.
  If \(s\in P_X\) is arbitrary, then we define \(f=\psi\), \(g= \psi^{-1}\),
  and
  define \(t_s\in X^X\) by
  \[(x)t_s=\begin{cases}
      b                 & \text{if }(x)\psi^{-1}\notin \dom(s) \text{ or }x=b
      \\
      (x)\psi^{-1}s\psi & \text{if }(x)\psi^{-1}\in \dom(s).
    \end{cases}\]
  It follows, from the definitions, that \(s=ft_sg\) for every $s\in P_X$.

  Suppose that $s\in P_X$ is arbitrary but fixed for the remainder of the
  proof.
  Let $B$ be an arbitrary basic open neighbourhood of $t_s$. Then there exist
  \(x_0, x_1, \ldots, x_{n-1}\in X\), $n\geq 1$, such that
  \[B=\bigcap_{i=0}^{n-1} \big\{t\in P_X: (x_i)t=(x_i)t_s\big\}.\]
  Let $V$ be the finite intersection of all of the following subbasic open
  sets:
  \begin{enumerate}[(a)]
    \item
          $U_{(x_i)\psi^{-1}, (x_i)t_s\psi^{-1}}$ for every $x_i\in \{x_0, x_1,
            \ldots
            , x_{n-1}\}\backslash\{b\}$ such that $(x_i)t_s\in \im(\psi)$;
          and
    \item
          $W_{(x_i)\psi ^ {-1}}$ for every $x_i \in \{x_0, x_1, \ldots,
            x_{n-1}\}\backslash\{b\}$ such that $(x_i)t_s = b$.
  \end{enumerate}
  Then $V$ is open and we will show that \(s \in V\).
  If \(x_i \in \{x_0, x_1, \ldots, x_{n-1}\}\backslash\{b\}\) and \((x_i)t_s\in
  \im(\psi)\), then \((x_i)\psi^{-1}\in \dom(s)\) and \(((x_i)\psi^{-1})s=
  ((x_i)t_s)\psi^{-1}\). This implies that $s \in U_{(x_i)\psi^{-1},
        (x_i)t_s\psi^{-1}}$.
  Similarly, if \(x_i \in \{x_0, x_1, \ldots, x_{n-1}\}\backslash\{b\}\) and
  \((x_i)t_s= b\), then \((x_i)\psi^{-1} \notin \dom(s)\).
  Hence $s\in W_{(x_i)\psi ^ {-1}}$ for such $x_i$.
  We conclude that $s\in V$.

  It remains to show that \(V \subseteq f_s (B \cap X^X) g_s\).
  Let \(k \in V\) be arbitrary. We will show that \(t_s\) and
  \(t_k\) agree on \(\{x_0, x_1, \ldots, x_{n-1}\}\), or, in other
  words, \(t_k \in B\). By the definitions of $t_s$  and $t_k$ \( (b)t_s = b =
  (b)t_k\). Suppose
  that \(x_i \in \{x_0, x_1, \ldots, x_{n-1}\}\setminus \{b\}\). If \((x_i)t_s
  = b\), then \((x_i)\psi^{-1} \notin
  \dom(s)\) and, since \(k \in W_{(x_i)\psi^{-1}}\), \((x_i)\psi^{-1} \notin
  \dom(k)\). Hence
  \((x_i)t_s = b = (x_i)t_k\). On the other hand, if \((x_i)t_s \neq b\), then
  \( (x_i)t_s \in \im(\psi)\). It follows from
  the fact that \(k \in U_{(x_i)\psi^{-1},
      (x_i)t_s\psi^{-1}}\) that  \((x_i)\psi^{-1}k =
  (x_i)t_s\psi^{-1}\), and so \( (x_i)t_s =
  (x_i)\psi^{-1}k\psi = (x_i)t_k\) and so $t_k\in B$. Since $t_k\in X^ X$ by
  definition, and so $k = f t_k g \in f(B\cap X^X)g$, as required.
  \medskip

  \noindent
  \textbf{(\ref{rdom_theorem_vi}).}
  Suppose that $\T$ is a topology that is
  semitopological for $P_X$ and that $\T$ induces the pointwise topology on
  $\Sym(X)$.
  Since $W_x \cap X ^ X = \set{h\in P_X}{x\not \in \dom(h)}\cap X ^
    X=\varnothing$, the topology induced by $\mathcal{P}$ on $\Sym(X)$ is just
  the pointwise topology. Hence $\T$ and $\mathcal{P}$ induce the same topology
  on $\Sym(X)$. Since $P_X$ has property \textbf{X} with respect to
  $\mathcal{P}$ and $\Sym(X)$, it follows that $\T\subseteq \mathcal{P}$ by
  \cref{lem-luke-0}\eqref{lem-luke-0-ii}.  \medskip

  \noindent
  \textbf{(\ref{rdom_theorem_viii}).}
  This follows from \cref{thm-elliott-1} and part \eqref{rdom_theorem_vi}.
  \medskip

  \noindent
  \textbf{(\ref{rdom_theorem_vii}).}
  This follows immediately from part \eqref{rdom_theorem_v},
  \cref{lem-luke-0}\eqref{lem-luke-0-iv} and the automatic continuity of $\Sym(X)$.\medskip

  \noindent
  \textbf{(\ref{rdom_theorem_ix}).}
  Let $X$ be countable. By part \eqref{rdom_theorem_i}, $P_X$ is
  homeomorphic to its image under $\phi$.  It is easy to see that $(P_X)\phi$
  is a closed subset of $Y^{Y}$ under the pointwise topology. Thus $P_X$ under
  $\mathcal{P}$ is homeomorphic to a closed subspace of a Polish space and is
  hence Polish.

  Suppose that $\T$ is a Polish topology that is semitopological for $P_X$.
  Then $\T$ is $T_1$ and so $\mathcal{P} \subseteq \T$ by
  \cref{thm-elliott-1}. Since
  $\mathcal{P}$ is Polish and $\Sym(X)$ is a Polish subgroup of $\P_X$, it
  follows from \cref{lem-luke-0}\eqref{lem-luke-0-iii} that $\T\subseteq \P$. Hence $\T =
    \mathcal{P}$, as required.
  \medskip

  \noindent
  \textbf{(\ref{rdom_theorem_x}).}
  Suppose that $\T$ is a $T_1$ second countable semigroup topology for $P_X$.
  By \cref{thm-elliott-1}, $\mathcal{P}$ is contained in $\T$.
  Applying \eqref{rdom_theorem_vii} to the identity function from
  $P_X$ with $\mathcal{P}$ to $P_X$  with $\T$ shows that $\T$ is contained in
  $\mathcal{P}$ also.
\end{proof}

Since $P_X \subseteq B_X$, the topologies $\mathcal{B}_1$ and $\mathcal{B}_2$
on $B_X$,
given in \cref{finest_inducer_relO} and
\cref{coarsest_T1_relO}, induce semigroup topologies on $P_X$.
It is natural to ask how the topology $\P$ from \cref{rdom_theorem}
relates to the semigroup topologies on $P_X$ induced by $\mathcal{B}_1$ and
$\mathcal{B}_2$.

\begin{prop}
  If $X$ is an infinite set,
  then $\P$ strictly contains the subspace topology on $P_X$ induced by
  $\mathcal{B}_1$ and $\P$ is strictly contained in the subspace topology on
  $P_X$ induced by $\mathcal{B}_2$.
\end{prop}
\begin{proof}
  Clearly from the definitions of $\mathcal{B}_1$ and $\P$,
  the subspace topology induced by $\mathcal{B}_1$ on $P_X$ is contained in
  $\P$. This containment is strict because $\P$ is $T_1$ but $\mathcal{B}_1$ is
  not $T_1$ on $I_X\subseteq P_X$ by \cref{finest_inducer_relO}\eqref{finest_inducer_rel0-ii}.

  The second containment follows since
  \[
    \set{h\in P_X}{x \not \in \dom(h)}=\set{h\in B_X}{(\{x\})h\subseteq
      \varnothing} \cap P_X\quad\text{and}\quad \set{h\in
      B_X}{(\{x\})h\subseteq
      \varnothing}\in \mathcal{B}_2
  \]
  for all $x\in X$. The topology $\P$ induces the pointwise topology on
  $\Sym(X)$. On the other hand, if $\mathcal{B}_2$ induced the pointwise
  topology on $\Sym(X)$, then $\mathcal{B}_2\subseteq \mathcal{B}_1$, by
  \cref{finest_inducer_relO}\eqref{finest_inducer_rel0-iv}, contradicting \cref{coarsest_T1_relO}\eqref{coarsest_T1_rel0-iv}.
\end{proof}

\subsection{The symmetric inverse monoid}
\label{subsection-sym-inv}

In this section we consider semigroup and inverse topologies on the symmetric
inverse monoid $I_X$ where $X$ is any infinite set. One natural way to obtain a
Hausdorff semigroup topology on $I_X$ is to consider it as a subspace of $P_X$
with the topology defined in \cref{subsection-partial}. In
\cref{minimal_semigroup_topologies}, we will show that this topology is
minimal among the $T_1$ semigroup topologies on $I_X$ but that inversion is not
continuous. If $U$ is open in the subspace topology, then the homeomorphic
topology generated by the sets $U ^ {-1} = \set{f\in I_X}{f ^ {-1}\in U}$ is
also minimal among the $T_1$ semigroup topologies on $I_X$ and again inversion
is not continuous. Moreover, every $T_1$ semigroup topology on $I_X$ contains
one or the other of these two topologies. This contrasts with the results of
Sections~\ref{subsection-full-transf}, \ref{subsection-binary}, and
\ref{subsection-partial} where it was shown that each of
the full transformation monoid, the full binary monoid, and the partial
transformation monoid have unique minimal $T_1$ semigroup topologies.
On the other hand, $I_X$ has a minimum inverse semigroup topology, the topology
generated by the two minimal $T_1$ semigroup topologies (see
\cref{minimal_inverse_semigroup_topology}). Arguably, this topology is the
best candidate, among those we exhibit, for a canonical topology on $I_X$. With
this viewpoint in mind, we conclude this section by proving an analogue of
\cref{theorem-luke-1} which characterises those topological inverse semigroups
that are topologically isomorphic to inverse subsemigroups of $I_{X}$, when $X$
is countable.

We begin by constructing the least $T_1$ topology that is semitopological for
$I_X$.

\begin{theorem}\label{invO-semitop}
  Let $X$ be an infinite set and let $\mathcal{I}_1$ denote the topology
  on $I_X$ generated by the sets
  \[U_{x, y} = \set{h\in I_X}{(x,y)\in h}\quad\text{and}\quad V_{x, y} =
    \set{h\in I_X}{(x,y)\not \in h}\]
  for all $x,y \in X$.	Then the following hold:
  \begin{enumerate}[\rm (i)]

    \item\label{invO-semitop-i}
          the topology $\mathcal{I}_1$ is compact, Hausdorff, and
          semitopological for $I_X$ and
          inversion is continuous;

    \item\label{invO-semitop-ii}
          $\mathcal{I}_1$ is the least $T_1$ topology that is semitopological
          for $I_X$;

    \item\label{invO-semitop-iii}
          if $X$ is countable, then $\mathcal{I}_1$ is Polish.
  \end{enumerate}
\end{theorem}
\begin{proof}
  \noindent\textbf{(\ref{invO-semitop-i}).}
  The sets $U_{x, y}$ generate the subspace topology on $I_X$ induced by the
  topology $\mathcal{B}_1$ on $B_X$ defined in
  \cref{finest_inducer_relO}. Hence, by
  \cref{finest_inducer_relO}\eqref{thm-item-1}, the subspace topology induced
  by
  $\mathcal{B}_1$ on $I_X$ is a semigroup topology where inversion is
  continuous also. Hence to prove that $I_X$ is semitopological and inversion
  is continuous with respect to $\mathcal{I}_1$, it suffices consider the
  subbasic open sets  $V_{x,
        y}$, $x, y\in X$.

  Suppose that $x, y\in X$ are fixed. If $\iota: I_X \to I_X$ is defined by
  $(f)\iota = f ^ {-1}$, then $(V_{x, y})\iota =\set{h\in I_X}{(y,x)\not \in h}
    =
    V_{y, x}$ is open in $\mathcal{I}_1$, and hence $\iota$ is continuous.
  Suppose that $f\in I_X$ is arbitrary.  If $x\not \in \dom(f)$, then $(x,y)
    \not \in fg$ for all $g\in I_X$. Hence $(V_{x, y})\lambda_f ^ {-1} = I_X$
  is
  open. If $x\in \dom(f)$ and $fg\in V_{x, y}$ for some $g\in I_X$.  Then
  $((x)f,y) \not \in g$. Thus $g \in V_{(x)f, y}\subseteq (V_{x,
      y})\lambda_f^{-1}$ and so $(V_{x, y})\lambda_f^{-1}$ is open. Hence
  $\lambda_f$ is continuous for every $f\in I_X$.
  Since $\rho_f = \iota \lambda_{f ^ {-1}} \iota$ is a composition of
  continuous functions, $\rho_f$ is continuous also.

  It remains to show that $I_X$ is Hausdorff and compact with respect to
  $\mathcal{I}_1$.
  We do this by identifying $I_X$ with a subset of $\{0, 1\} ^ {X \times X}$
  with the product topology. It will follow that \(I_X\) is homeomorphic to a
  closed subspace
  of the compact Hausdorff space  \(\{0, 1\}^{X\times X}\) and is thus compact
  and
  Hausdorff.
  If \(f\in I_X\) is arbitrary, then we identify $f$ with the function from \(X
  \times X\) to \(\{0, 1\}\) defined by
  \[(a,b) \mapsto
    \begin{cases}
      1 & \text{if }(a, b)\in f      \\
      0 & \text{if }(a, b) \notin f,
    \end{cases}\]
  for all $(a,b) \in X \times X$. We may, therefore, identify $I_X$ with a
  subset of $\{0, 1\} ^ {X \times X}$.
  Viewed in this way, the topology $\mathcal{I}_1$ is precisely the subspace
  topology inherited from the product space \(\{0, 1\}^{X\times X}\). The
  complement of \(I_X\) in \(\{0, 1\}^{X\times X}\) is the union of the open sets
  \[
    \left\{f\in \{0, 1\}^{X\times X}: (a, b)f = (a, c)f = 1\right\}
  \]
  and
  \[ \left\{f\in \{0, 1\}^{X\times X}: (b, a)f = (c, a)f = 1\right\}\]
  for all $a, b, c \in X$ with $b\neq c$.
  \medskip

  \noindent\textbf{(\ref{invO-semitop-ii}).}
  Let $\T$ be a $T_1$ topology that is semitopological for $I_X$.
  Let $x,y\in X$  and let $f\in I_X$ be arbitrary. Then
  \[
    \{(x,x)\}\circ f\circ \{(y,y)\} = \varnothing
    \quad\text{if and only if}\quad f\in V_{x, y}
  \]
  and
  \[
    \{(x,x)\}\circ f\circ \{(y,y)\} = \{(x,y)\}
    \quad\text{if and only if}\quad f\in U_{x, y}.
  \]
  Since $\T$ is $T_1$, the singletons $\{\varnothing\}$ and $\{\{(x,y)\}\}$ are
  closed in $\T$. Thus their respective pre-images $V_{x, y}$ and $U_{x,y}$
  under $\lambda_{\{(x,x)\}} \circ \rho_{\{(y,y)\}}$ are
  closed in $\S$. Hence $U_{x, y}$ and $V_{x, y}$ are both open in $\T$, as
  they are mutual
  complements. Thus $\mathcal{I}_1 \subseteq \T$, as required.	\medskip

  \noindent\textbf{(\ref{invO-semitop-iii}).}
  In the proof of part \eqref{invO-semitop-i} we showed that \(I_X\) is homeomorphic to a closed
  subspace of the Cantor space $\{0, 1\} ^ {X \times X}$. It follows that \(I_X\)
  is a compact metrisable space and is thus Polish.

  An explicit complete metric on $I_X$ can be obtained by
  choosing any metric compatible with the topology on \(\{0, 1\}^{X\times X}\)
  and translating it to \(I_X\) via the homeomorphism given in part \eqref{invO-semitop-i}. If we
  choose \(X= \N\), then one such example is given by
  $$d(f,g)= \frac{1}{m+1} $$
  where $m=\min\set{n\in \N}{(n\times n) \cap f \neq (n\times n) \cap g}$.
\end{proof}

We will show that there are precisely two minimal $T_1$ semigroup topologies on
$I_X$. The first is just the topology induced by the minimal $T_1$ semigroup
topology on $P_X$ (see \cref{rdom_theorem}) and the second consists of
the inverses $U^{-1}=\set{f^{-1}}{f \in U}$ of the open sets $U$ of the first
one. In \cref{prop-inverse-not-equal},  we will show that the two topologies
introduced in the next theorem do not coincide with $\mathcal{I}_1$ from the
previous theorem. It will follow that the topology $\mathcal{I}_1$ from
\cref{invO-semitop} is not a semigroup topology on $I_X$.

\begin{theorem}\label{minimal_semigroup_topologies}
  Let $X$ be an infinite set, let $\mathcal{I}_2$ be the topology on the
  symmetric
  inverse monoid $I_X$ generated by the collection of sets
  \[
    U_{x, y} = \set{h\in I_X}{(x,y)\in h}\quad\text{and}\quad
    W_{x}    = \set{h\in I_X}{x\not\in\dom(h)}
  \]
  and let $\mathcal{I}_3$ be the topology on $I_X$ generated by the sets
  \[
    U_{x, y}	 = \set{h\in I_X}{(x,y)\in h}\quad\text{and}\quad
    W_{x} ^ {-1} = \set{h\in I_X}{x\not\in\im(h)}
  \]
  for all $x,y \in X$.
  Then the following hold:
  \begin{enumerate}[\rm (i)]
    \item\label{minimal_semigroup_topologies_0}
          \(\mathcal{I}_2\) and \(\mathcal{I}_3\) are distinct and both contain
          the topology $\mathcal{I}_1$ defined in \cref{invO-semitop};

    \item\label{minimal_semigroup_topologies_i}
          $I_X$ with $\mathcal{I}_2$ and $I_X$ with $\mathcal{I}_3$ are
          homeomorphic Hausdorff topological semigroups;

    \item\label{minimal_semigroup_topologies_ii}
          every $T_1$ semigroup topology for $I_X$ contains $\mathcal{I}_2$ or
          $\mathcal{I}_3$;

    \item\label{minimal_semigroup_topologies_iii}
          $\mathcal{I}_2 \cap \mathcal{I}_3$ coincides with the Hausdorff-Markov and
          Fr\'echet-Markov topologies for
          $I_X$;

    \item\label{minimal_semigroup_topologies_iv}
          if $X$ is countable, then $\mathcal{I}_2$ and $\mathcal{I}_3$ are
          Polish.
  \end{enumerate}
\end{theorem}
\begin{proof}
  \noindent\textbf{(\ref{minimal_semigroup_topologies_0}).}
  Let \(x\in X\) be fixed. We will show that the set \(W_x^{-1}\) is not open
  in \(\mathcal{I}_2\). Seeking a contradiction, suppose that \(W_x^{-1}\) is open in
  \(\mathcal{I}_2\). It follows that there are finite \(f\in I_X\) and
  \(Y\subseteq X\) such that
  \[\varnothing\in \left(\bigcap_{(a,b)\in f}U_{a, b}\right)\cap
    \left(\bigcap_{y\in Y} W_y\right)\subseteq W_x^{-1}\]
  and so \(f=\varnothing\). Furthermore, if \(b\in X\), is distinct from every
  element of \(Y\), then
  \[\{(b,x)\}\in \left(\bigcap_{(a,b)\in f}U_{a, b}\right)\cap
    \left(\bigcap_{y\in Y} W_y\right),\]
  but \(\{(b, x)\}\notin W_x^{-1}\), a contradiction.

  The fact that \(\mathcal{I}_1\) is contained in both \(\mathcal{I}_2\) and
  \(\mathcal{I}_3\) follows since \[V_{x,y}=\left(\bigcup_{z\neq x}
    U_{z,y}\right) \cup W_x=\left(\bigcup_{z\neq y} U_{x,z}\right) \cup W_y^{-1}.\]
  \medskip

  \noindent\textbf{(\ref{minimal_semigroup_topologies_i}).}
  The Hausdorff semigroup topology $\P$ on $P_X$ defined in
  \cref{rdom_theorem} induces $\mathcal{I}_2$ on $I_X$.
  Hence $I_X$ is a Hausdorff topological semigroup under $\mathcal{I}_2$.

  The map $f \mapsto f^{-1}$ defines an anti-automorphism of $I_X$. The images
  of the subbasic open sets for $\mathcal{I}_2$ under inversion give the subbasis for
  $\mathcal{I}_3$.  By \cref{prop-anti-automorph}, it follows that
  $(I_X, \mathcal{I}_3)$ is a topological semigroup homeomorphic to $(I_X,
    \mathcal{I}_2)$.
  \medskip

  \noindent\textbf{(\ref{minimal_semigroup_topologies_ii}).}
  Let $\T$ be any $T_1$ semigroup topology for $I_X$. By
  \cref{invO-semitop}\eqref{invO-semitop-ii}, $\T$ contains the topology
  $\mathcal{I}_1$
  with subbasis
  $$U_{x, y} = \set{h\in I_X}{(x,y)\in h} \text{ and}\quad
    V_{x, y} =\set{h\in I_X}{(x,y)\not \in h}$$
  for all $x,y \in X$.
  It remains to show that either $\set{h \in I_X}{x \not \in \dom(x)} \in \T$
  for all $x\in X$ or $\set{h\in I_X}{y \not \in \im(h)}\in \T$ for all
  $y\in X$.

  For every $x\in X$, the set $V_{x, x}=\set{h\in I_X}{(x,x)\not \in h}$ is an
  open neighbourhood of $\varnothing$ in $\T$.	Since $\varnothing \circ
    \varnothing = \varnothing \in V_{x, x}$ and $\T$ is a semigroup topology,
  there exists an open neighbourhood $U$ of $\varnothing$ such that $U \circ U
    \subseteq V_{x, x}$. In other words, $(x,x) \not \in uv$ for any $u,v \in
    U$.
  If $z\in X$ is arbitrary such that $z\not\in \set{y\in X}{(x,y) \not \in u
      \text{
        for all }u \in U}$ and $z\not\in \set{y\in X}{(y,x) \not \in u \text{
        for
        all }u \in U}$, then there exist $u, v\in U$ such that $(x, z)\in u$
  and $(z,
    x) \in v$, and so $(x, x) \in uv$, a contradiction. Hence every $z\in X$
  belongs to one of the sets:
  \[\set{y\in X}{(x,y) \not \in u\ \text{for all }u \in U}\text{ or }
    \set{y\in X}{(y,x) \not \in u\ \text{for all }u \in U}
  \]
  and so one of these two sets has cardinality $|X|$.
  \medskip

  Suppose that $|\set{y\in X}{(x,y) \not \in u\ \text{for all }u \in U}|=|X|$. It
  follows that there exists $Y \subseteq X$
  with $|Y|=|X \setminus Y|$ such that	$(x,y) \not \in u$ for all $y \in Y$
  and $u\in U$. Let $p\in \Sym(X)$ be any involution such that $(Y)p = X
    \setminus Y$, and so $(X\setminus Y)p = Y$. Then for any $u\in U$ and any
  $y
    \in X \setminus Y$ since $(y)p\in Y$ it follows that $(x, (y)p)\not\in u$
  and
  so $(x, y)\not\in up$.  Let $V=U \cap Up$. Then $V$ is an open neighbourhood
  of $\varnothing$ and $x\not \in \dom(f)$ for all $f \in V$.  Let $g\in I_X$
  be arbitrary. If $x \in \dom(g)$, then $x \in \dom(\{(x,x)\}\circ g)$ and so
  $\{(x,x)\} \circ g \not \in V$. On the other hand, if $x \not \in \dom(g)$,
  then $\{(x,x)\}\circ g=\varnothing \in V$. Thus
  $$(V)\lambda_{\{(x,x)\}}^{-1}=\set{g\in I_X}{\{(x,x)\}\circ g \in V}=\set{g
      \in I_X}{x \not \in \dom(g)}$$
  is open in $\T$. Since $x\in X$ was arbitrary, it
  follows that $\mathcal{I}_2 \subseteq \T$.

  If $|\set{y\in X}{(y,x) \not \in u \text{ for all }u
      \in U}|=|X|$, then $\mathcal{I}_3 \subseteq \T$ by an analogous argument.
  \medskip

  \noindent\textbf{(\ref{minimal_semigroup_topologies_iii}).}
  These follow directly from parts \eqref{minimal_semigroup_topologies_i},
  \eqref{minimal_semigroup_topologies_ii}, and \cref{figure-the-only}.
  \medskip

  \noindent\textbf{(\ref{minimal_semigroup_topologies_iv}).}
  It is shown in \cref{rdom_theorem}\eqref{rdom_theorem_ii}
  and~\eqref{rdom_theorem_ix} that the partial transformation
  monoid $P_{X}$ forms a Polish semigroup with the topology $\mathcal{P}$
  defined in that theorem.  The topology induced by $\mathcal{P}$ on $I_X$ is
  $\mathcal{I}_2$.  It is routine to verify that $I_X$ is a closed subset of
  $P_{X}$ under the topology $\mathcal{P}$, and so $\mathcal{I}_2$ is Polish
  also.  Thus $\mathcal{I}_3$ is Polish as it is homeomorphic to
  $\mathcal{I}_2$.
\end{proof}

We remark that complete metrics on $I_{\N}$, that induce
$\mathcal{I}_2$ and $\mathcal{I}_3$ from
\cref{minimal_semigroup_topologies}, can be defined
using the natural embedding $\phi$ defined in \eqref{natural_embedding} of
\cref{subsection-partial}:
\begin{equation}\label{d_dom}
  d_{1}(f,g)=
  \begin{cases}
    \; 0          & \text{ if } f=g       \\
    \frac{1}{m+1} & \text{ if } f\not = g
  \end{cases}
  \qquad \text{ where }m=\text{min}\set{x \in \N}{(x)(f)\phi\neq (x)(g)\phi}
\end{equation}
and
\begin{equation}\label{d_im}
  d_{2}(f,g)=
  \begin{cases}
    \; 0          & \text{ if } f=g       \\
    \frac{1}{m+1} & \text{ if } f\not = g
  \end{cases}
  \qquad \text{ where }m=\text{min}\set{x \in
    \N}{(x)\left(f^{-1}\right)\phi\neq
    (x)\left(g^{-1}\right)\phi}.
\end{equation}

Next, we show that the Hausdorff-Markov and Fr\'echet-Markov topology
$\mathcal{I}_2\cap \mathcal{I}_3$ is not equal to the least $T_1$
semitopological semigroup topology $\mathcal{I}_1$.

\begin{prop}\label{prop-inverse-not-equal}
  If $\mathcal{Z}$ is the Zariski topology on the symmetric inverse monoid
  $I_X$, then
  $\mathcal{Z}$ properly contains $\mathcal{I}_1$ and so $\mathcal{I}_1$ is
  properly contained in $\mathcal{I}_2 \cap \mathcal{I}_3$.
\end{prop}
\begin{proof}
  Since $\mathcal{I}_1$ is the least $T_1$ topology which is semitopological
  for $I_X$ (\cref{invO-semitop}\eqref{invO-semitop-ii}) and $\mathcal{I}_2\cap
    \mathcal{I}_3$ is the Hausdorff-Markov topology
  (\cref{minimal_semigroup_topologies}\eqref{minimal_semigroup_topologies_iii}),
  it follows from \cref{figure-the-only} that
  \[\mathcal{I}_1\subseteq \mathcal{Z} \subseteq \mathcal{I}_2 \cap
    \mathcal{I}_3.\] It therefore suffices to show that $\mathcal{Z} \not \subseteq
    \mathcal{I}_1$.
  Let $x\in X$ be arbitrary and consider the set
  $$U=\set{s\in I_X}{\{(x,x)\}s^2\neq \{(x,x)\}} \in \mathcal{Z}.$$
  We will show that the element $\varnothing \in U$ does not have an open
  neighbourhood in $\mathcal{I}_1$ which is contained in $U$. If $V$ is a basic
  open neighbourhood of $\varnothing$ in $\mathcal{I}_1$, then $V$ is of the form
  $V=\set{s\in I_X}{h\cap s=\varnothing}$ for some finite $h\in I_X$. Since $h$
  is finite, there exists $y\in X$ such that $y\not \in \dom(h) \cup \im(h)$.
  Then $\{(x,y),(y,x)\} \in V \setminus U$ and so $V \not \subseteq U$, as
  required.
\end{proof}

In the next theorem we consider the topology generated by the union of the two
minimal
$T_1$ semigroup topologies $\mathcal{I}_2$ and $\mathcal{I}_3$ on $I_X$.

\begin{theorem}\label{minimal_inverse_semigroup_topology}
  Let $X$ be an infinite set and let $\mathcal{I}_4$ be the topology on the
  symmetric
  inverse monoid $I_X$ generated by the collection of sets
  \[
    U_{x, y} = \set{h\in I_X}{(x,y)\in h},\
    W_{x}    = \set{h\in I_X}{x\not\in\dom(h)},\
    \text{and}\
    W_{x} ^ {-1} = \set{h\in I_X}{x\not\in\im(h)}.
  \]
  Then the following hold:
  \begin{enumerate}[\rm (i)]

    \item
          \label{minimal_inverse_semigroup_topology_i}
          the topology $\mathcal{I}_4$ is a Hausdorff inverse semigroup
          topology
          for $I_X$;

    \item
          \label{minimal_inverse_semigroup_topology_iii}
          $\mathcal{I}_4$ is the Hausdorff-Markov inverse and the
          Fr\'echet-Markov inverse topology for $I_X$;

    \item
          \label{minimal_inverse_semigroup_topology_v}
          the inverse Zariski topology for $I_X$ is $\mathcal{I}_4$;

    \item
          \label{minimal_inverse_semigroup_topology_vi}
          $I_X$ has property \textbf{X} with respect to $\mathcal{I}_4$
          and $\Sym(X)$;

    \item
          \label{minimal_inverse_semigroup_topology_vii}
          if $\T$ is a topology that is semitopological for $I_X$ and $\T$
          induces
          the pointwise topology on $\Sym(X)$, then $\T$ is contained in
          $\mathcal{I}_4$;

    \item
          \label{minimal_inverse_semigroup_topology_viii}
          $\mathcal{I}_4$ is the unique $T_1$ inverse semigroup topology on
          $I_X$
          inducing the pointwise topology on $\Sym(X)$;

    \item
          \label{minimal_inverse_semigroup_topology_ix}
          if $X$ is countable, then $I_X$ with the topology $\mathcal{I}_4$ has
          automatic continuity with respect to the class of second countable topological
          semigroups;

    \item
          \label{minimal_inverse_semigroup_topology_x}
          if $X$ is countable, then $\mathcal{I}_4$
          is the unique $T_1$ second countable inverse semigroup topology on
          $I_X$;

    \item
          \label{minimal_inverse_semigroup_topology_xi}
          if $X$ is countable, then $\mathcal{I}_4$ is the unique Polish
          inverse
          semigroup topology on $I_X$.
  \end{enumerate}
\end{theorem}
\begin{proof}
  The topology $\mathcal{I}_4$ is the topology generated by the union
  $\mathcal{I}_2 \cup \mathcal{I}_3$ defined in
  \cref{minimal_semigroup_topologies}.
  \medskip

  \noindent\textbf{(\ref{minimal_inverse_semigroup_topology_i}).}
  Since $\mathcal{I}_4$ is generated by the Hausdorff semigroup topologies
  $\mathcal{I}_2$ and $\mathcal{I}_3$ (by
  \cref{minimal_semigroup_topologies}\eqref{minimal_semigroup_topologies_i}),
  it follows that
  $\mathcal{I}_4$ is a Hausdorff semigroup topology for $I_X$.	If $U$ is any
  of the subbasic open sets defining $\mathcal{I}_4$, then
  $U^{-1}=\set{f^{-1}}{f \in U}$ is also a subbasic open set, and so
  $\mathcal{I}_4$ is an inverse semigroup topology.
  \medskip

  \noindent\textbf{(\ref{minimal_inverse_semigroup_topology_iii}).}
  The topology $\mathcal{I}_4$ is Hausdorff
  by~\eqref{minimal_inverse_semigroup_topology_i}. Thus both the
  Fr\'{e}chet-Markov inverse semigroup topology, and the inverse Hausdorff-Markov
  topology for $I_X$ are contained in \(\mathcal{I}_4\).

  For the converse,
  suppose that $\T$ is any $T_1$ inverse semigroup topology for $I_X$. By
  \cref{minimal_semigroup_topologies}\eqref{minimal_semigroup_topologies_ii},
  $\T$ contains either
  $\mathcal{I}_2$ or $\mathcal{I}_3$. Since inversion is continuous with
  respect to $\T$ and $\set{U ^ {-1}}{U\in \mathcal{I}_2}=\mathcal{I}_3$, it
  follows that $\T$ contains $\mathcal{I}_4$. Thus the Fr\'{e}chet-Markov
  inverse semigroup topology, and the inverse Hausdorff-Markov topology for $I_X$
  are contained in \(\mathcal{I}_4\).\medskip

  \noindent\textbf{(\ref{minimal_inverse_semigroup_topology_v}).}
  By \cref{figure-the-only} and part
  \eqref{minimal_inverse_semigroup_topology_iii}, the inverse Zariski topology is
  contained in $\mathcal{I}_4$. On the other hand, the subbasic open sets of
  $\mathcal{I}_4$ of the form $U_{x,y}$ are open in the least $T_1$ topology
  $\mathcal{I}_1$ that is semitopological for $I_X$, as shown in
  \cref{invO-semitop}\eqref{invO-semitop-ii}. Since $\mathcal{I}_1$ is contained
  in the inverse Zariski topology by \cref{figure-the-only}, it only
  remains to show that the subbasic open sets of the form $W_x$ and $W_x^{-1}$ of
  $\mathcal{I}_4$ are open in the inverse Zariski topology.

  Let $x\in X$ and $s\in I_X$ be arbitrary. Then $\{(x,x)\}ss^{-1}=\{(x,x)\}$
  if $x\in \dom(s)$ and $\{(x,x)\}ss^{-1}=\varnothing$ otherwise. Hence
  $$W_x=\set{s\in I_X}{x\not \in \dom(s)}=\set{s\in I_X}{\{(x,x)\}ss^{-1}\neq
    \{(x,x)\}}$$ is open in the inverse Zariski topology. Similarly,
  $W_x^{-1}=\set{s\in I_X}{s^{-1}s\{(x,x)\}\neq \{x,x\}}$ and so the inverse
  Zariski topology coincides with $\mathcal{I}_4$, as required.
  \medskip

  \noindent\textbf{(\ref{minimal_inverse_semigroup_topology_vi}).}
  Suppose that $s\in S$ is arbitrary.
  Let $f\in I_X$ be any element satisfying $\dom(f)=X$ and $|X \setminus
    \im(f)|=|X|$. Then $t\in \Sym(X)$
  satisfies $ftf^{-1}=s$ if and only if
  \begin{equation*}
    (x, (x)s)\in ftf^{-1} \text{ for all } x\in \dom(s) \text{ and } x \not\in
    \dom(ftf^{-1}) \text{ for all } x\in X \setminus \dom(s),
  \end{equation*}
  which is equivalent to
  \begin{equation}\label{invO_gen_condition}
    (x)ft=(x)sf \text{ for all } x\in \dom(s) \text{ and } (x)ft \in X
    \setminus \im(f) \text{ for all } x\in X \setminus \dom(s).
  \end{equation}
  If $\phi_0:(X\backslash \dom(s))f \to X\backslash \im(f)$ is any injection
  such that
  $|(X\backslash \im(f))\backslash \im(\phi_0)| = |X|$, and
  $\phi_1: X\backslash \im(f) \to (X\backslash \im(sf))\backslash \im(\phi_0)$
  is a bijection, then $t: X \to X$ defined by
  \[(x)t=
    \begin{cases}
      (x)f^{-1}sf & \text{ if } x\in (\dom(s))f             \\
      (x)\phi_0   & \text{ if } x\in (X\backslash \dom(s))f \\
      (x)\phi_1   & \text{ if } x\not\in \im(f)
    \end{cases}
  \]
  belongs to $\Sym(X)$ and, by \eqref{invO_gen_condition},  $ftf ^ {-1} = s$.

  We will define a basis $\B$ for $\mathcal{I}_4$ and then show that
  if $U \in \B$ contains $t$, then $f(U \cap \Sym(X))f ^ {-1}\in \B$, and so
  property \textbf{X} holds.
  If $k\in I_{X}$ is finite, and $Y$ and $Z$ are finite subsets of $X$, then we
  define
  \[
    R_{k, Y, Z}
    = \set{h\in I_X}{k\subseteq h\ \text{and }Y\cap \dom(h)=Z\cap
      \im(h)=\varnothing}.
  \]
  Since
  \[
    R_{k, Y, Z} = \bigcap_{x\in \dom(k)} U_{x, (x)k} \cap \bigcap_{y\in Y}
    W_{y} \cap
    \bigcap_{z\in Z} W_z ^ {-1}
  \]
  the sets $R_{k, Y, Z}$  are open in $\mathcal{I}_4$. On the other hand, any
  finite intersection of subbasic open sets for $\mathcal{I}_4$ is of the form
  $R_{k, Y, Z}$ for some finite $k$, $Y$, and $Z$.
  Hence the collection $\B$ of all such sets $R_{k, Y, Z}$ forms a basis for
  $\mathcal{I}_4$. Let $U \in \B$ be any such basic open neighbourhood of $t$.
  By definition, $U$ is one of the sets of the form $R_{k, Y, Z}$ but
  since $t\in \Sym(X)$, $\dom(t) = \im(t) = X$, it follows that
  $U = R_{k, \varnothing, \varnothing}$ for some finite $k$.
  Hence
  $$f(U\cap \Sym(X))f^{-1}=\set{fgf^{-1}}{g \in \Sym(X) \text{ and }k \subseteq
      g}.$$
  It suffices to show that
  \[ f(U \cap \Sym(X))f^{-1} = R_{fkf ^ {-1}, Y, Z} \]
  where
  $$Y=\set{(y)f^{-1}}{y \in \dom(k)\ \text{and }(y)k \not\in \dom(f ^ {-1})}$$
  and
  $$Z=\set{(z)f^{-1}}{z \in \im(k)\ \text{and }(z)k^{-1}\not\in
    \dom(f^{-1})}.$$

  If $fgf ^{-1}\in f(U\cap \Sym(X))f^{-1}$, then $fkf^{-1}\subseteq fg f ^{-1}$
  since $g\in U = R_{k, \varnothing, \varnothing}$.
  If  $(y)f^ {-1} \in Y$ is arbitrary, then
  $y\in \dom(k)$ and $(y)k \not\in \dom(f ^ {-1})$. Hence $(y)f ^ {-1}f k f ^
      {-1} = (y)kf ^
      {-1}$  is not defined and so $(y)f ^ {-1}f g f ^ {-1} = (y)g f ^ {-1} =
    (y)k f ^ {-1}$ is not
  defined.  In other words, $(y)f^{-1} \not \in \dom(fgf ^ {-1})$ and so $Y\cap
    \dom (fgf ^
    {-1}) = \varnothing$. Similarly, if $(z)f ^ {-1} \in Z$ is arbitrary, then
  $z\in \im(k)$ and $(z)k^{-1} \not\in \dom (f^{-1})$.	Hence
  $(z)f^{-1}fk^{-1}f^{-1} = (z)k^{-1}f^{-1}$ is not defined, or, in other words,
  \((z)f^{-1}\notin \dom(fg^{-1}f^{-1}) = \im(fgf^{-1})\). It follows that
  \(Z\cap \im(fgf^{-1})=\varnothing\).
  Hence every $fgf ^ {-1} \in f(U\cap \Sym(X))f^{-1}$ belongs to $R_{fkf ^
          {-1}, Y, Z}$, and so
  $f(U\cap \Sym(X))f^{-1} \subseteq R_{fkf ^ {-1}, Y, Z}$.

  Conversely, suppose that $u\in R_{fkf ^ {-1}, Y, Z}$ is arbitrary.
  We define $\phi_2:(X\backslash \dom(u))f \to (X\backslash \im(f))\backslash
    \{(x)k:x\in X\backslash \dom(f)\}$ is an injection which agrees with $k$ when
  they are both defined and $|(X\backslash \im(f))\backslash \im(\phi_2)| = |X|$,
  and $\phi_3: X\backslash \im(f) \to (X\backslash \im(uf))\backslash
    \im(\phi_2)$ is a bijection which agrees with $k$ when they are both defined.
  We define $g : X \to X$ by
  \[(x)g=
    \begin{cases}
      (x)f^{-1}uf & \text{ if } x\in (\dom(u))f             \\
      (x)\phi_2   & \text{ if } x\in (X\backslash \dom(u))f \\
      (x)\phi_3   & \text{ if } x\not\in \im(f).
    \end{cases}
  \]
  Then $g\in \Sym(X)$ and $k\subseteq g$. Hence $g\in U$. Finally,
  $g$ satisfies condition
  \eqref{invO_gen_condition} (where $s$ and $t$ are replaced by $u$ and $g$,
  respectively), and so $u = f g f ^ {-1}\in f (U \cap \Sym(X)) f ^ {-1}$, as
  required.
  \medskip

  \noindent\textbf{(\ref{minimal_inverse_semigroup_topology_vii}).}
  This follows from \eqref{minimal_inverse_semigroup_topology_vi} together with
  \cref{lem-luke-0}\eqref{lem-luke-0-ii}.
  \medskip

  \noindent\textbf{(\ref{minimal_inverse_semigroup_topology_viii}).}
  This follows straight from \eqref{minimal_inverse_semigroup_topology_iii} and
  \eqref{minimal_inverse_semigroup_topology_vii}.
  \medskip

  \noindent\textbf{(\ref{minimal_inverse_semigroup_topology_ix}).}
  This follows from part \eqref{minimal_inverse_semigroup_topology_vi},
  \cref{lem-luke-0}\eqref{lem-luke-0-iv} and the automatic
  continuity of $\Sym(X)$.
  \medskip

  \noindent\textbf{(\ref{minimal_inverse_semigroup_topology_x}).}
  If $X$ is countable, then the subbasis for $\mathcal{I}_4$ is countable also.
  Thus $\mathcal{I}_4$ is second countable.
  On the other hand, if $\T$ is any $T_1$  second
  countable, inverse semigroup topology for $I_X$, then,
  by part (\ref{minimal_inverse_semigroup_topology_iii}), $\mathcal{I}_4
    \subseteq \T$. By part (\ref{minimal_inverse_semigroup_topology_viii}), we also
  have $\T
    \subseteq \mathcal{I}_4$ and so $\T =\mathcal{I}_4$, as required.
  \medskip

  \noindent\textbf{(\ref{minimal_inverse_semigroup_topology_xi}).}
  Suppose that $X = \N$. We only need to show that $\mathcal{I}_4$ is
  completely
  metrizable, since separability and uniqueness then follow from part
  (\ref{minimal_inverse_semigroup_topology_x}).  Let
  $\phi$ be the natural embedding defined in \eqref{natural_embedding} of
  $I_{\N}$
  into $\N\cup \{\na\}^{\N\cup \{\na\}}$. We define the metric $d$ on $I_{\N}$
  by
  $d(f,f)=0$ and
  if $f\not =g$, then $d(f,g)=\frac{1}{m+1}$ where $$m=\min\set{y \in
    \N}{(y)((f)\phi)\neq (y)((g)\phi) \text{ or } (y)((f^{-1})\phi)\neq
    (y)((g^{-1})\phi)}.$$ It is routine to show that $d$ is a metric on
  $I_{\N}$. (In fact, $d$ is the maximum of the metrics $d_1$ and $d_2$ defined
  in
  \eqref{d_dom} and \eqref{d_im}.)

  We will now show that the topology induced by $d$ is $\mathcal{I}_4$.
  As in the proof of part \eqref{minimal_inverse_semigroup_topology_vi}, the
  sets
  \[R_{k, Y, Z} = \set{f\in I_{\N}}{k \subseteq f \text{ and } Y\cap
      \dom(f)=Z\cap
      \im(f)=\varnothing}\]
  where $f\in I_{\N}$ is finite and $Y$ and $Z$ are finite subsets of $X$, form
  a
  basis $\B$ for $\mathcal{I}_4$.

  For any $f\in I_{\N}$ and $m\in \N$, we have that
  \begin{align*}
    B\left(f,\frac{1}{m+1}\right)
     & =\left\{g\in I_{\N}\ :\ (x)((f)\phi)=(x)((g)\phi) \text{ and
    }(x)((f^{-1})\phi)=(x)((g^{-1})\phi)\text{ for all }x\in
    m\right\}                                                       \\
     & =R_{k, Y, Z}
  \end{align*}
  where $k = (f\cap(m\times \N)) \cup (f\cap (\N\times m))$, $Y = m \setminus
    \dom(f)$, and $Z = m \setminus \im(f)$.
  Hence every open ball under $d$ is open in $\mathcal{I}_4$ and so the
  topology
  induced by $d$ is contained in $\mathcal{I}_4$.

  Suppose that $f\in I_{\N}$ is finite, $Y, Z$ are finite subsets of $X$, and
  $$M=\max\left(\dom(f) \cup \im(f) \cup Y \cup Z\right)\in \N.$$
  If $g\in R_{f, Y, Z}$, then $B(g,1/M) \subseteq R_{f, Y, Z}$ and so
  $R_{f,Y,Z}$
  is open in the topology induced by $d$. We have shown that $\mathcal{I}_4$
  coincides with the topology induced by $d$.

  To show that $d$ is complete, suppose that $f_0, f_1, f_2, \ldots\in I_{\N}$
  is a Cauchy
  sequence. For every $m\in \N$, there exists $M\in \N$ such that $i,j\geq M$
  implies that $d(f_i,f_j)<1/m$. So if $x\leq m$, then
  $(x)((f_i)\phi)=(x)((f_j)\phi)$ and
  $(x)((f_i^{-1})\phi)=(x)((f_j^{-1})\phi)$ for all $i,j\geq M$. In
  particular, the sequences $((x)((f_0)\phi), (x)((f_1)\phi), (x)((f_2)\phi),
    \dots)$ and
  $((x)((f_0^{-1})\phi), (x)((f_1^{-1})\phi), (x)((f_2^{-1})\phi), \dots )$ are
  eventually
  constant with values $(x)F$ and $(x)F^{-1}$, respectively. This defines
  $F\in I_{\N}$ and the sequence $f_0, f_1, f_2, \ldots$ converges to $F$.
  Therefore
  the metric $d$ is complete.
\end{proof}

\begin{cor}\label{cor-zariski-inverse}
  The semigroup and inverse semigroup Zariski topologies on the symmetric
  inverse monoid $I_{\N}$
  are distinct.
\end{cor}
\begin{proof}
  By
  \cref{minimal_inverse_semigroup_topology}(\ref{minimal_inverse_semigroup_topology_v}), 
  the inverse semigroup Zariski topology on $I_{\N}$ is $\mathcal{I}_4$.
  On the other hand, the semigroup Zariski topology is contained in
  $\mathcal{I}_2 \cap \mathcal{I}_3$ by \cref{prop-inverse-not-equal}.
  \cref{minimal_semigroup_topologies}(\ref{minimal_semigroup_topologies_0})
  implies that $\mathcal{I}_2 \cap \mathcal{I}_3$ is strictly contained in
  \(\mathcal{I}_4\), and so the semigroup and inverse semigroup Zariski
  topologies must be distinct.
\end{proof}

By \cref{minimal_semigroup_topologies}\eqref{minimal_semigroup_topologies_ii}, every Polish semigroup
topology on $I_{\N}$ contains $\mathcal{I}_2$ or $\mathcal{I}_3$ and, by
\cref{minimal_inverse_semigroup_topology}\eqref{minimal_inverse_semigroup_topology_vi}, is contained in
$\mathcal{I}_4$.

\begin{question}
  Are $\mathcal{I}_2,\mathcal{I}_3$ and $\mathcal{I}_4$ the only Polish
  semigroup topologies on $I_{\N}$?\footnote{Addendum: S. Bardyla, L. Elliott, J. D. Mitchell, and Y. P\'eresse recently showed that the answer to this question is no, and there are infinitely many Polish semigroup topologies on $I_{\N}$.}
\end{question}

We can now give an example to demonstrate that the right small index property
is really distinct from the left small index property.
\begin{prop}\label{left-not-right}
  The topological semigroup \(I_{\N}\) with \(\mathcal{I}_2\) has the right
  small index
  property but not the left small index property.
\end{prop}
\begin{proof}
  We start by showing that \((I_\N,\mathcal{I}_2)\) does not have the left
  small
  index property.
  Let \(x\in \N\) be fixed. We define a left congruence on \(I_{\N}\) by
  \[\sigma:=\{(f, g)\in I_{\N}: (x\notin \im(f)\cup \im(g))\text{ or
    }\left(x\in \im(f)\cap
    \im(g) \text{ and }(x)f^{-1} = (x)g^{-1}\right)\}.\]
  The classes of this left congruence are the sets \(U_{y, x}\) for \(y\in \N\)
  and \(W_x^{-1}\). In the proof of
  \cref{minimal_semigroup_topologies}\eqref{minimal_semigroup_topologies_0}, it
  was shown that the set \(W_x^{-1}\) is not open in \(\mathcal{I}_2\). Hence
  $I_{\N}$ with the topology $\mathcal{I}_2$
  does not have the left small index property.

  We next show that $I_\N$ with $\mathcal{I}_2$ has the right small index
  property.  Let \(\rho\) be a right congruence on \(I_{\N}\) with countably
  many classes.  By
  \cref{minimal_inverse_semigroup_topology}(\ref{minimal_inverse_semigroup_topology_ix}) 
  together with \cref{AC-implies-SIP},
  \(\rho\) is open with respect to \(\mathcal{I}_4\).
  It suffices to show that $f/\rho$ is open in $\mathcal{I}_2$ for an arbitrary
  $f\in I_\N$. In particular, we need only find
  an open neighbourhood $U$ of $f$ in $\mathcal{I}_2$ such that \(U\subseteq
  f/\rho\). Let
  \[V= \{g\in I_\N: h\subseteq g, X\cap \dom(g) = \varnothing, Y\cap \im(g) =
    \varnothing\},\]
  where \(h\subseteq \N\times\N\) and \(X, Y\subseteq \N\) are finite, be a
  basic open neighbourhood
  of \(f\) in \(\mathcal{I}_4\) such that \(V\subseteq f/\rho\). We define
  \[U:=\{g\in I_\N: h\subseteq g, X\cap \dom(g) = \varnothing\}.\]
  The set \(U\) is an open neighbourhood of \(f\) in \(\mathcal{I}_2\). It
  therefore suffices
  to show that \(U\subseteq f/\rho\). Let \(g\in U\) be arbitrary. Let \(k\in
  \mathcal{I}_\N\)
  be such that \(k|_{\im(h)}\) is the identity function, \(\dom(k) = \N\), and
  \(\im(k) \subseteq \N\backslash Y\). Since $\dom(k)=\N$, $kk^{-1}$ is the
  identity of $I_{\N}$. Furthermore,
  \(fk\in V\subseteq f/\rho\) and also \(gk\in V \subseteq
  f/\rho\). It follows that \(gk/\rho = f/\rho = fk/\rho\), and so
  $g/\rho =gkk ^ {-1}/\rho = fk k ^ {-1}/\rho = f/\rho$, as required.
\end{proof}

The symmetric inverse monoid $I_{\N}$ is the natural analogue of the full
transformation monoid $\N ^ \N$, in that every countable inverse monoid can
be embedded in $I_{\N}$. We will now prove the inverse monoid analogue of
\cref{theorem-luke-1}, characterising
those inverse monoids that embed as topological inverse submonoids of $I_{\N}$
with its unique Polish inverse semigroup topology $\mathcal{I}_4$.

If $S$ is an inverse monoid with identity $1_S$ and $\rho$ is a right
congruence, then we will say that $\rho$ is \textit{Vagner-Preston} if for
every $s\in S$ either:	$t\in s/\rho$ implies $tt ^ {-1} / \rho = 1_S / \rho$;
or $st/\rho = s/\rho$ for all $t\in S$.

\begin{lem}\label{lem-the-wee-lemma}
  Let $\rho$ be a Vagner-Preston right congruence on an inverse monoid $S$ with
  identity $1_S$. If $f,g\in S$ satisfy $ff^{-1}/\rho\neq 1_S/\rho$ and
  $gg^{-1}/\rho \neq 1_S/\rho$, then $f/\rho=g/\rho$.
\end{lem}
\begin{proof}
  Since $\rho$ is Vagner-Preston, $ff^{-1}/\rho\neq 1_S/\rho_i \neq
    gg^{-1}/\rho_i$ implies that $fs/\rho_i=f/\rho_i$ and $gs/\rho_i=g/\rho_i$ for
  all $s\in S$. Hence
  $$f/\rho=ff^{-1}gg^{-1}/\rho=gg^{-1}ff^{-1}/\rho=g/\rho_i$$
  since $ff ^ {-1}, gg ^ {-1}\in S$ are idempotents, and idempotents commute in
  inverse semigroups.
\end{proof}

\begin{lem}\label{lem-Vagner-Preston}
  Let $S$ be a semitopological inverse monoid and let $\rho$ be an open
  Vagner-Preston right congruence on $S$. If $a\in S$ satisfies $aa^{-1}/\rho =
    1_S/\rho$, then
  $\set{f\in S}{a/\rho=aff^{-1}/\rho}$ is clopen.
\end{lem}
\begin{proof}
  Throughout this proof, let $a\in S$ satisfy $aa^{-1}/\rho=1_S/\rho$. If $f\in
    S$ is arbitrary, then
  $$a/\rho=aff^{-1}/\rho \Rightarrow
    1_S/\rho=aa^{-1}/\rho=aff^{-1}a^{-1}/\rho=(af)(af)^{-1}/\rho$$
  and
  $$1_S/\rho=(af)(af)^{-1}/\rho\Rightarrow a/\rho =
    aff^{-1}a^{-1}a/\rho=aa^{-1}aff^{-1}/\rho=aff^{-1}/\rho.$$
  Hence $a/\rho=aff^{-1}/\rho$ if and only if $1_S/\rho=(af)(af)^{-1}/\rho$ and
  so
  \begin{eqnarray*}
    \set{f\in S}{a/\rho =aff^{-1}/\rho} & = & \set{f\in S}{(af)(af)^{-1}/\rho
      =  1_S/\rho} \\
    & = & \set{f\in S}{(af)(af)^{-1} \in 1_S/\rho}\\
    &= & \set{f \in S}{ff^{-1}\in 1_S/\rho}\lambda_a^{-1}\\
    & = &\set{f \in S}{ff^{-1}/\rho=1_S/\rho}\lambda_a^{-1}.
  \end{eqnarray*}
  Since $S$ is semitopological, it now suffices to show that $U=\set{f \in
      S}{ff^{-1}/\rho= 1_S/\rho}$ is clopen. We will do so by showing that $U$ is a
  union of classes of the open equivalence relation $\rho$.
  Let $f\in U$, i.e.  $ff^{-1}/\rho=1_S/\rho$. Since $\rho$ is Vagner-Preston,
  either $tt^{-1}/\rho=1_S/\rho$ for all $t\in f/\rho$ or $ft/\rho=f/\rho$ for
  all $t\in S$. In the first case, clearly $f/\rho \subseteq U$. In the second
  case $t\in f/\rho$ implies that
  $$tt^{-1}/\rho=ft^{-1}/\rho=f/\rho_i=ff^{-1}/\rho=1_S/\rho$$
  and so, again, $f/\rho\subseteq U$, as required.
\end{proof}

\begin{theorem}\label{thm-inverse-submonoids}
  If $S$ is a $T_0$ semitopological inverse monoid,
  then the following are equivalent:
  \begin{enumerate}[\rm (i)]
    \item\label{thm-inverse-submonoids-i}
          there is a sequence $\set{\rho_i}{i\in \N}$ of right Vagner-Preston
          congruences on $S$, each with countably many classes, such that
          the set $\set{s/\rho_i, (s/\rho_i) ^ {-1}}{s\in S,\ i\in \N}$ is a
          subbasis\footnote{The term ``subbasis'' cannot be replaced with ``basis'' here, unlike in \cref{theorem-luke-1}.} for the topology on $S$;

    \item\label{thm-inverse-submonoids-ii} $S$ is topologically isomorphic to an inverse subsemigroup of
          $I_{\N}$ with the topology $\mathcal{I}_4$;

    \item\label{thm-inverse-submonoids-iii} $S$ is topologically isomorphic to an inverse submonoid of $I_{\N}$
          with the topology $\mathcal{I}_4$.
  \end{enumerate}
\end{theorem}

\begin{proof}
  Obviously \eqref{thm-inverse-submonoids-iii} implies \eqref{thm-inverse-submonoids-ii} . So it suffices to show that \eqref{thm-inverse-submonoids-i}  and \eqref{thm-inverse-submonoids-ii}  are
  equivalent, and that \eqref{thm-inverse-submonoids-ii} implies \eqref{thm-inverse-submonoids-iii}.
  \vspace{\baselineskip}

  \noindent\textbf{(\ref{thm-inverse-submonoids-ii}) $\Rightarrow$ (\ref{thm-inverse-submonoids-i})}. Assume without loss of generality that
  $S$ is an inverse subsemigroup of $I_{\N}$. We will show that $S$ with the
  subspace topology induced by $\mathcal{I}_4$ satisfies the conditions in \eqref{thm-inverse-submonoids-i}.
  For every $i\in \N$, we define $\rho_i \subseteq S \times S$ so that
  \[
    f / \rho_i =  \set{g}{ i\in \dom(g),\ (i)f = (i)g } = U_{i, (i)f}
  \]
  if $i\in \dom(f)$ and
  \[
    f / \rho_i = \set{g}{ i\not\in \dom(g)} = W_{i}
  \]
  if $i\not\in \dom(f)$, where $U_{i, (i)f}$ and $W_i$ are the subbasic open
  sets for $\mathcal{I}_4$ defined in \cref{minimal_inverse_semigroup_topology}.
  It is routine to verify that $\rho_i$ is a right congruence for every $i\in
    \N$.
  Clearly,
  \[(f / \rho_i) ^ {-1} = \set{g}{ i\in \im(g),\ (i)f ^ {-1} = (i)g ^ {-1} } =
    U_{(i)f ^ {-1}, i}\]
  if $i\in \im(f)$ and
  \[
    (f / \rho_i) ^ {-1} = \set{g}{ i\not\in \im(g)}=W_{i}^{-1}
  \]
  if $i\not\in \im(f)$.

  It remains to prove that the right congruence $\rho_i$, for every $i\in \N$,
is Vagner-Preston.  If $1_S$ is the identity of $S$, then $1_S$ is not
necessarily $1_{\N}$.  However, $1_S$ is an idempotent satisfying $\dom(s),
\im(s) \subseteq \dom(1_S)$ for all $s\in S$.  Suppose that $s\in S$ and $i\in
\N$ are arbitrary.  If $i\in \dom(s)$ and $t\in s/\rho_i$, then $i\in \dom(t)$
and $(i) tt ^ {-1} = i$. Hence $t t ^ {-1}/\rho_i = 1_S / \rho_i$.  If $i \not
\in \dom(s)$ and $t\in S$ is arbitrary, then $i\not\in \dom(st)$ and so
$st/\rho_i = s/\rho_i$. Hence every $\rho_i$ is Vagner-Preston, as required.
  \vspace{\baselineskip}

  \noindent {\bf \textbf{(\ref{thm-inverse-submonoids-i})} $\Rightarrow$ \textbf{(\ref{thm-inverse-submonoids-ii})}.} Let $X=\set{a/\rho_i}{i\in \N, a\in
      S}$. When $i\neq j$, we will consider classes of $\rho_i$ and classes of
  $\rho_j$ as different elements of $X$, even if they should happen to be the
  same subset of $S$.

  Since $X$ is countable, it suffices to find a topological isomorphism $\phi$
  from $S$ to an inverse subsemigroup of $I_X$ with the topology $\mathcal{I}_4$.
  This $\phi: S \to I_X$ is defined by
  $$(a/\rho_i)(f)\phi=
    \begin{cases}
      af/\rho_i          & \text{ if }aa^{-1}/\rho_i=1_S/\rho_i \text{ and }
      a/\rho_i=aff^{-1}/\rho_i;                                              \\
      \text{ undefined } & \text{ otherwise.}
    \end{cases}
  $$
  We need to show that the partial function $(f)\phi \in I_X$ is well-defined
  and injective, that $\phi$ is a continuous homomorphism, and that if $U$ is
  open in $S$, then $(U)\phi$ is open in $(S)\phi$ with subspace topology.

  To see that $(f)\phi$ is a well-defined partial function from $X$ to $X$, let
  $a,b\in S$ such that $a/\rho_i=b/\rho_i$ for some $i\in \N$. Suppose that $a a
      ^ {-1}/\rho_i = 1_S/\rho_i$ and
  $a/\rho_i = aff ^ {-1}/\rho_i$. Since $\rho_i$ is a right congruence and
  $a/\rho_i = b/\rho_i$, $b/\rho_i = bff ^ {-1}/\rho_i$.
  By the assumption of the theorem, either $bb^{-1}/\rho_i=1_S/\rho_i$ or
  $bs/\rho_i=b/\rho_i$ for all $s\in S$.
  In the latter case, $as/\rho_i = bs/\rho_i = b/\rho_i = a/\rho_i$ and
  so taking $s = a ^ {-1}$, it follows that $a/\rho_i =
    aa^{-1}/\rho_i=1_S/\rho_i$. Hence, right multiplying $1_S$ and $a$ by $bb ^
      {-1}$, we obtain
  $bb^{-1}/\rho_i=abb^{-1}/\rho_i$
  and applying $as/\rho_i = a/ \rho_i$ when $s = bb ^ {-1}$ yields
  $$bb^{-1}/\rho_i=abb^{-1}/\rho_i = a/\rho_i = 1_S/\rho_i.$$
  So in either case, $b b^ {-1}/\rho_i = 1_S/\rho_i$.
  We have shown that $a/\rho_i = b/\rho_i$ and $(a/\rho_i)(f)\phi$ being defined
  implies that  $(b/\rho_i)(f)\phi$ is defined also. That $\phi$ is well-defined
  follows from $(a/\rho)(f)\phi = af/\rho_i = bf /\rho_i = (b/\rho)(f)\phi$.

  To see that $(f)\phi$ is injective, let $a/\rho_i, b/\rho_j \in
    \dom((f)\phi)$ such that $(a/\rho_i)(f)\phi=(b/\rho_j)(f)\phi$. By the
  definition of $(f)\phi$, it follows that $i=j$, $aff^{-1}/\rho_i=a/\rho_i$ and
  $bff^{-1}/\rho_i=b/\rho_i$. Hence
  \begin{equation*}
    (a/\rho_i)(f)\phi=(b/\rho_i)(f)\phi \
    \Rightarrow\ af/\rho_i=bf/\rho_i
    \ \Rightarrow\ aff^{-1}/\rho_i=bff^{-1}/\rho_i
    \ \Rightarrow\ a/\rho_i=b/\rho_i.
  \end{equation*}
  We conclude that $(f)\phi$ is injective and hence $(f)\phi \in I_X$.

  To prove that $\phi$ is injective, suppose that $(f)\phi=(g)\phi$ for some
  $f,g\in S$. We will show that $f/\rho_i=g/\rho_i$ for all $i\in \N$. If
  $f/\rho_i\neq 1_S/\rho_i$ and $g/\rho_i\neq 1_S/\rho_i$, then $f/ \rho_i= g /
    \rho_i$, by \cref{lem-the-wee-lemma}. So, without loss of generality,
  assume that $ff^{-1}/\rho_i=1_S/\rho_i$. Then $1_S/\rho_i=ff^{-1}/\rho_i \in
    \dom((f)\phi)=\dom((g)\phi)$ and so
  $f/\rho_i=(1_S/\rho_i)(f)\phi=(1_S/\rho_i)(g)(\phi)=g/\rho_i$.
  We have shown that $f/\rho_i=g/\rho_i$ for all $i\in \N$. It follows that
  $(f/\rho_i) ^ {-1} = (g/\rho_i) ^ {-1}$, and
  since $\set{a/\rho_i, (a/\rho_i) ^ {-1}}{a\in S,\ i\in \N}$ is a subbasis for
  a $T_0$ topology on $S$, we conclude that $f=g$, and so $\phi$ is injective.

  To prove that $\phi$ is a homomorphism, let $f,g \in S$ be arbitrary. The key
  step in showing that $(fg)\phi=(f)\phi\, (g)\phi$ is to show that
  $\dom((fg)\phi)=\dom((f)\phi\,(g)\phi)$. By definition, $a/\rho_i \in
    \dom((fg)\phi)$ if and only if
  \begin{equation}\label{equation_SIM_homomorphism_1}
    aa^{-1}/\rho_i=1_S /\rho_i \quad \text{ and } \quad
    a/\rho_i=a(fg)(fg)^{-1}/\rho_i.
  \end{equation}
  On the other hand, $a/\rho_i \in \dom((f)\phi\, (g)\phi)$ if and only if
  $a/\rho_i \in \dom((f)\phi)$ and $af/\rho_i \in \dom((g)\phi)$ if and only if
  \begin{equation}\label{equation_SIM_homomorphism_2}
    aa^{-1}/\rho_i=1_S /\rho_i=(af)(af)^{-1}/\rho_i, \quad
    aff^{-1}/\rho_i=a/\rho_i, \quad
    \text{ and } \quad af/\rho_i=afgg^{-1}/\rho_i.
  \end{equation}
  To show that \eqref{equation_SIM_homomorphism_1} and
  \eqref{equation_SIM_homomorphism_2} are equivalent, first assume that $a\in S$
  satisfies \eqref{equation_SIM_homomorphism_1}. Then $aa^{-1}/\rho_i=1_S/\rho_i$
  and
  \begin{eqnarray*}
    aff^{-1}/\rho_i & = &
    (a(fg)(fg)^{-1})ff^{-1}/\rho_i=afgg^{-1}f^{-1}ff^{-1}/\rho_i=afgg^{-1}f^{-1}/
    \rho_i\\
    & = & a(fg)(fg)^{-1}/\rho_i=a/\rho_i.
  \end{eqnarray*}
  Hence
  $$1_S/\rho_i=aa^{-1}/\rho_i=(aff^{-1})a^{-1}/\rho_i=(af)(af)^{-1}/\rho_i,$$
  and

  $$af/\rho_i=(a(fg)(fg)^{-1})f/\rho_i=afgg^{-1}f^{-1}f/\rho_i=aff^{-1}fgg^{-1}/\
    rho_i=afgg^{-1}/\rho_i.$$
  Thus \eqref{equation_SIM_homomorphism_1} implies
  \eqref{equation_SIM_homomorphism_2}. Now assume that  $a\in S$ satisfies
  \eqref{equation_SIM_homomorphism_2}. Then $aa^{-1}/\rho_i=1_S/\rho_i$ and
  $$a/\rho_i=aff^{-1}/\rho_i=(afgg^{-1})f^{-1}/\rho_i=a(fg)(fg)^{-1}/\rho_i.$$
  Hence \eqref{equation_SIM_homomorphism_2} implies
  \eqref{equation_SIM_homomorphism_1} and so the domains of $(fg)\phi$ and
  $(f)\phi\, (g)\phi$ coincide. If $ a/\rho_i\in \dom((fg)\phi)=\dom((f)\phi\,
    (g)\phi)$, then
  $$(a/\rho_i)(f)\phi\,
    (g)\phi=(af/\rho_i)(g)\phi=afg/\rho_i=(a/\rho_i)(fg)\phi.$$
  It follows that $(fg)\phi=(f)\phi\, (g)\phi$ and so $\phi$ is a homomorphism.

  To show that $\phi$ is continuous we need to prove that the preimages of the
  subbasic open sets $U_{a/\rho_i, b/\rho_j}$, $W_{a/\rho_i}$ and
  $W_{a/\rho_i}^{-1}$ of $\mathcal{I}_4$ are open in $S$.

  If $U_{a/\rho_i, b/\rho_j}\cap (S)\phi \neq \varnothing$, then $i=j$ and
  $aa^{-1}/\rho_i=1_S/\rho_i$. Thus
  \begin{align*}
    \left(U_{a/\rho_i, b/\rho_i}\right)\phi^{-1}
     & =\set{f\in S}{(a/\rho_i, b/\rho_i) \in (f)\phi}
    \\
     & =\set{f\in S}{a/\rho_i \in \dom((f)\phi) \text{ and }af/\rho_i=b/\rho_i}
    \\
     & =\set{f\in S}{a/\rho_i \in \dom((f)\phi)}\cap \set{f\in
    S}{af/\rho_i=b/\rho_i}                                                      \\
     & =\set{f\in S}{a/\rho_i =aff^{-1}/\rho_i}\cap \set{f\in S}{af \in
      b/\rho_i}.
  \end{align*}
  The set $\set{f\in S}{a/\rho_i =aff^{-1}/\rho_i}$ is clopen by
  \cref{lem-Vagner-Preston} and $\set{f\in S}{af \in
      b/\rho_i}=(b/\rho_i)\lambda_a^{-1}$ is open since $b/\rho_i$ is open and $S$ is
  semitopological. It follows that  $\left(U_{a/\rho_i,
      b/\rho_i}\right)\phi^{-1}$ is open.

  Now consider $(W_{a/\rho_i})\phi^{-1}=\set{f\in S}{a/\rho_i \not \in
      \dom((f)\phi)}$. If $aa^{-1}/\rho_i \neq 1_S/\rho_i$, then
  $(W_{a/\rho_i})\phi^{-1}=S$. If $aa^{-1}/\rho_i=1_S/\rho_i$,  then
  \[
    (W_{a/\rho_i})\phi^{-1}
    =\set{f\in S}{a/\rho_i \neq aff^{-1}/\rho_i}\\
    =S \setminus \set{f\in S}{a/\rho_i=aff^{-1}/\rho_i}.
  \]
  Hence $(W_{a/\rho_i})\phi^{-1}$ is open as $\set{f\in
    S}{a/\rho_i=aff^{-1}/\rho_i}$ is clopen (again by \cref{lem-Vagner-Preston}).
  Since $\phi$ is a homomorphism
  $$({W_{a/\rho_i}}^{-1})\phi^{-1}=\left(
    (W_{a/\rho_i})\phi^{-1}\right)^{-1},$$
  which is open since inversion is continuous in $S$.

  It only remains to prove that if $U$ is open in $S$, then $(U)\phi$ is open
  in $(S)\phi$.
  Since $\phi:S \rightarrow (S)\phi$ is a bijection, it is enough to show that
  the images of the subbasic open sets $a/\rho_i$ and $(a/\rho_i) ^ {-1}$ are
  open. Let $(f)\phi\in (a/\rho_i)\phi = \set{(g)\phi \in I_X}{g\in a/\rho_i}$ be
  arbitrary. We need to find an open neighbourhood $U$ of $(f)\phi$ such that
  $U\cap (S)\phi \subseteq (a/\rho_i)\phi$. Since $\phi$ is injective, $f\in
    a/\rho_i$ and so $(a/\rho_i)\phi=(f/\rho_i)\phi=\set{(g)\phi\in
      I_X}{g/\rho_i=f/\rho_i}$. It therefore suffices to show that there exists an
  open set $U$ in $\mathcal{I}_4$ such that
  \begin{equation}\label{equation_open_function_condition}
    (f)\phi\in U \text{ and } U \cap (S)\phi \subseteq \set{(g)\phi \in
      I_X}{g/\rho_i=f/\rho_i}.
  \end{equation}

  If $ff^{-1}/\rho_i=1_S/\rho_i$, then $1_S/\rho_i \in \dom((f)\phi)$ and we
  let $U=U_{1_S/\rho_i, f/\rho_i}$. Then $(f)\phi\in U$ and if $(g)\phi\in U\cap
    (S)\phi$, then $g/\rho_i=(1_S/\rho_i)(g)\phi=f/\rho_i$. Hence $U$ satisfies
  \eqref{equation_open_function_condition}.

  If $ff^{-1}/\rho_i =1_Sff^{-1}/\rho_i \neq 1_S/\rho_i$, then $1_S/\rho_i \not
    \in \dom((f)\phi)$. In this case, let $U=W_{1_S/\rho_i}$. Then $(f)\phi\in U$
  and if $(g)\phi \in U\cap (S)\phi$, then $1_S/\rho_i \not \in \dom((g)\phi)$.
  Hence $gg^{-1}/\rho_i \neq 1_S/\rho_i$, and so, by \cref{lem-the-wee-lemma},
  $f/\rho_i=g/\rho_i$.
  It follows that $U$ satisfies \eqref{equation_open_function_condition}. We
  have shown that $(a/\rho_i)\phi$ is open in $(S)\phi$ for all $i\in \N$ and
  $a\in S$.

  Since $\phi$ is a homomorphism and inversion is continuous (and hence open)
  in $S$ and in $I_X$, it follows that
  $$((a/\rho_i)^{-1})\phi=\left((a/\rho_i)\phi\right)^{-1}$$
  is open in $(S)\phi$. We have shown that $\phi$ is open, which concludes the
  proof of this part of the theorem.
  \vspace{\baselineskip}

  \noindent {\bf \textbf{(\ref{thm-inverse-submonoids-ii})} $\Rightarrow$ \textbf{(\ref{thm-inverse-submonoids-iii})}.}
  Without loss of generality, we assume that the inverse monoid $S$ is a
  subsemigroup of $I_{\N}$ with the topology $\mathcal{I}_4$, and the subspace
  topology.  We define $X$ to be $\dom(1_S)$ and note that:
  $$X = \dom(1_S)=\bigcup_{s\in S}\dom(s)=\bigcup_{s\in S}\im(s)=\im(1_S).$$
  It is clear that $S$ is a submonoid of $I_X$, the symmetric inverse monoid on
  the set $X$.
  There are two cases to consider: when $X$ is finite, and when $X$ is
  infinite.

  Suppose that $X$ is infinite.
  In this case, since $X$ is countable, $I_X$ is isomorphic to $I_{\N}$. It is
  also the case that $I_{\N}$ and $I_X$ with the subspace topology induced by
  $\mathcal{I}_4$ (on $I_{\N}$) are topologically isomorphic. Indeed, the
  subbasis for $\mathcal{I}_4$ on $I_{\N}$ given in
  \cref{minimal_inverse_semigroup_topology}, when restricted to $I_X$ gives the
  subbasis for $\mathcal{I}_4$ on $I_X$. Alternatively, the unique $T_1$ second
  countable inverse semigroup topology on $I_{\N}$, and hence $I_X$, is
  $\mathcal{I}_4$
  (\cref{minimal_inverse_semigroup_topology}\eqref{minimal_inverse_semigroup_topology_x}). 
  Hence, by either route, $S$ is topologically isomorphic to an
  inverse submonoid of $I_{\N}$.

  If $X$ is finite, then $S$ is finite and hence discrete, since
  $\mathcal{I}_4$ is $T_1$. The map
  $$s \mapsto s \cup \set{(x,x)}{x \in \N \setminus X}$$
  embeds $S$ as an inverse submonoid of $I_{\N}$. The embedding is topological,
  since the domain and range are both finite and hence discrete.
\end{proof}

\subsection{Injective and surjective transformations}
\label{subsection-inj-surj}

In this section we consider the submonoids $\Inj(X)$, of injective, and
$\Surj(X)$, of surjective,
transformations in $X ^ X$. If $X$ is countable, then
$\Inj(X)$ and $\Surj(X)$ are $G_{\delta}$ subsets of $X ^ X$ with the pointwise
topology, and so these submonoids are Polish semigroups with the subspace
topology.
However, in contrast to the monoids considered up to this point in the paper,
we will show that both of $\Inj(X)$ and $\Surj(X)$ have infinitely many
distinct Polish semigroup topologies when $X$ is countably infinite. In the
case of $\Inj(X)$, when $X$ is countably infinite,  we show that the pointwise
topology is the minimum Polish semigroup topology on $\Inj(X)$, and we show
that there is a maximum Polish semigroup topology on $\Inj(X)$ given by an
explicit subbasis.

Since $\Inj(X)$ is a closed submonoid of the symmetric inverse monoid $I_X$
with its unique Polish inverse semigroup topology $\mathcal{I}_4$ (see
\cref{minimal_inverse_semigroup_topology}), one natural Polish topology on
$\Inj(X)$ is the subspace topology induced by $\mathcal{I}_4$. A subbasis for
this topology comprises the subbasic open sets
\begin{equation}\label{eq-jay-1}
  \set{f\in \Inj(X)}{(x)f = y}\quad \text{and}\quad
  \set{f\in \Inj(X)}{y \not \in \im(f)}\text{ for all }x, y \in X.
\end{equation}
Although natural, this topology does not play a central role in this section.
We will show that this topology is one of infinitely many Polish semigroup
topologies on $\Inj(X)$ between the minimum and maximum Polish semigroup
topologies on $\Inj(X)$.
The topology with subbasis given in \eqref{eq-jay-1} coincides with the
subspace topology induced by the topology $\mathcal{I}_3$ on $I_X$ from
\cref{minimal_semigroup_topologies}. The topology on $\Inj(X)$ induced by
$\mathcal{I}_2$ from \cref{minimal_semigroup_topologies}, is the pointwise
topology.

\begin{theorem}
  \label{theorem-injective}
  Let $X$ be an infinite set and let $\mathcal{J}$ be the topology on $\Inj(X)$
  generated by the pointwise
  topology and the sets
  $$\set{f\in \Inj(X)}{x \not \in \im(f)}\quad \text{and}\quad
    \set{f\in \Inj(X)}{|X\setminus \im(f)|=n}$$
  for every $x\in X$ and for every cardinal $n \leq |X|$.
  Then the following hold:
  \begin{enumerate}[\rm (i)]

    \item\label{theorem-injective-i}
          there are infinitely many distinct
          semigroup topologies on $\Inj(X)$ containing the pointwise topology
          and
          contained in $\mathcal{J}$ including $\mathcal{J}$ itself;
          
    \item\label{theorem-injective-ii}
          $\Inj(X)$ has property \textbf{X} with respect to $\mathcal{J}$
          and $\Sym(X)$;

    \item\label{theorem-injective-iii}
          if $\T$ is any semigroup topology on $\Inj(X)$ and $\T$ induces
          the pointwise topology on $\Sym(X)$, then $\T$ is contained in
          $\mathcal{J}$;

    \item\label{theorem-injective-v}
          the Zariski topology for $\Inj(X)$ is the pointwise topology;
          
    \item\label{theorem-injective-iv}
          if $X$ is countable, then $\Inj(X)$ has automatic continuity with
          respect
          to $\mathcal{J}$ and the class of second countable topological
          semigroups;
    
    \item\label{theorem-injective-vi} 
    if $X$ is countable, then there are 
          infinitely many distinct Polish
          semigroup topologies on $\Inj(X)$ containing the pointwise topology
          and
          contained in $\mathcal{J}$ including $\mathcal{J}$ itself;
          
    \item\label{theorem-injective-vii} 
          if X is countable, then the pointwise topology is the minimum, and $\mathcal{J}$ is the maximum, Polish semigroup topology on $\Inj(X)$.
  \end{enumerate}
\end{theorem}
\begin{proof}
  \textbf{(\ref{theorem-injective-i}).}
  If $\kappa$ is any cardinal such that $\kappa\leq |X|$, then we define
  \[
    F_{\kappa} = \set{f\in \Inj(X)}{|X\setminus \im(f)|=\kappa}.
  \]
  Let $\S_0$ be the subspace topology on $\Inj(X)$ induced by the topology
  $\mathcal{I}_4$ on $I_X$.
  A subbasis for $\S_0$ is given in \eqref{eq-jay-1}.
  If $n\in \N$ and $n > 0$, then we define $\S_n$ to be the topology generated
  by $\S_{n -1}$ and $F_{n - 1}$. It follows that $\S_n$ is the least topology
  containing $\S_0$ and every set $F_m$ where $m< n$.
  The topology $\mathcal{J}$ is the topology generated by $\S_0$ and the
  collection of all $F_{\kappa}$ where $\kappa\leq |X|$.

  We begin by showing that $\mathcal{J}$ and every $\S_n$, $n\in \N$, is a
  semigroup topology on $\Inj(X)$.
  It follows from the definition of $\S_0$ that $\S_0$ is a semigroup topology.
  Suppose that $\T$ is	 $\mathcal{J}$ or  $\S_n$ where $n\in \N$. To show that
  $\T$ is a semigroup topology,
  it suffices to show that if $F_{\kappa} \in \T$,
  then	the preimage of $F_{\kappa}$ under
  the multiplication function is open in $\Inj(X)\times \Inj(X)$ with the
  product topology induced by $\T$.
  Note that  $$|X\setminus \im(fg)|=|X \setminus \im(f)|+|X \setminus \im(g)|$$
  for all
  $f,g\in \Inj(X)$.
  Suppose that $F_{\kappa}\in \T$ for some $\kappa$ such that $\kappa\leq |X|$.
  If $f,g \in \Inj(X)$ are such that $fg \in F_{\kappa}$, then	$|X \setminus
    \im(f)|=\alpha$ and $|X \setminus \im(g)|=\beta$ for some cardinals
  $\alpha,\beta
    \leq \kappa$ such
  that $\alpha+\beta=\kappa$.  Thus $f\in F_{\alpha}$,
  $g\in F_{\beta}$, and $F_{\alpha}F_{\beta}\subseteq F_{\kappa}$, and so $\T$
  is a semigroup topology.

  Next, we will show that the topologies $\S_0, \S_1, \ldots$ and $\mathcal{J}$
  are distinct.
  Suppose that $n\in \N$ is arbitrary. If $f\in F_n$ is arbitrary and $B$ is
  any basic open neighbourhood of
  $f$ in $\S_n$, then $f\not\in F_m$ for every $m \not= n$, and so $B$ is a
  basic open set for $\S_0$. In other words,
  \[
    B = \set{g\in \Inj(\N)}{h\subsetneq g\ \text{and}\ \im(g) \cap Y =
      \varnothing}
  \]
  for finite $h\in I_{\N}$ and finite $Y\subseteq \N$.	It is routine to verify
  that there exists $g\in B$ such that $g\not\in F_n$, and hence $F_n$ is not
  open in $\S_n$ for every $n\in \N$. It follows that the topologies $\S_0,
    \S_1, \ldots$ are distinct. Clearly, $\mathcal{J}$ contains the strictly
  ascending chain of topologies $\S_0, \S_1, \ldots$ and hence
  $\mathcal{J}$ is distinct from every $\S_n$, $n\in \N$.\medskip

  \noindent\textbf{(\ref{theorem-injective-ii}).}
  Let $s\in \Inj(X)$ be arbitrary. We show that there exist
  $f_s, g_s\in \Inj(X)$ and $a_s\in \Sym(X)$ such that $s = f_s a_s g_s$ and
  for every basic open neighbourhood $B$ of $a_s$ there exists an open
  neighbourhood $U$ of $s$ such that $U \subseteq f_s (B\cap \Sym(X))g_s$.

  Let $f_s= s$, $a_s= 1_{\Inj(X)}$, and $g_s= 1_{\Inj(X)}$ and let $B$ be a
  basic open neighbourhood of $a_s$. As $a_s= 1_{\Inj(X)}$ there is some
  finite $Y\subseteq X$ such that
  \[B\cap \Sym(X) = \{h\in \Sym(X): (y)h = y \text{ for all }y\in Y\}.\]
  Then $h \in \Inj(X)$ is an element of $f_s (B\cap
    \Sym(X))g_s$ if and only if
  \begin{enumerate}
    \item  $y\notin \im(h)$ for all $y\in Y \backslash \im(f_s)$;

    \item $((y)s^{-1},y)\in h$	for all $y\in Y \cap \im(f_s)$;

    \item $|X \backslash \im(s)|= |X \backslash \im(h)|$.
  \end{enumerate}
  So $f_s (B\cap \Sym(X))g_s$ is an open neighbourhood of $s$ as required.
  \medskip

  \noindent\textbf{(\ref{theorem-injective-iii}).}
  This follows by part \eqref{theorem-injective-ii} together with
  \cref{lem-luke-0}\eqref{lem-luke-0-ii}. \medskip

  \noindent\textbf{(\ref{theorem-injective-v}).}
  This part of the theorem follows from
  \cref{lem-zariski-pointwise}. \medskip
  
  \noindent\textbf{(\ref{theorem-injective-iv}).}
  This follows by part \eqref{theorem-injective-ii} together with
  \cref{lem-luke-0}\eqref{lem-luke-0-iv} and
  the automatic continuity of $\Sym(X)$. \medskip
  
  \noindent\textbf{(\ref{theorem-injective-vi}).}
  We will show that, when $X$ is countable, the topologies $\S_0, \S_1,\ldots$ and
  $\mathcal{J}$ from the proof of part \eqref{theorem-injective-i} are Polish.
  As already discussed at the start of this section, $\S_0$ is Polish. To show
  that $\S_n$ is Polish we proceed by induction. Given that $\S_{n-1}$ is Polish,
  for some $n > 0$,
  it suffices by~\cite[Lemma 13.2]{Kechris1995aa} to show that $F_n$ is closed
  in $\S_n$.
  The complement of $F_n$  in $\Inj(X)$ is given by
  \begin{eqnarray*}
    \Inj(X) \setminus F_{n}
    & = & \set{f\in \Inj(X)}{|X\setminus \im(f)| \neq n} \\
    & = &  \set{f\in \Inj(X)}{|X\setminus \im(f)| > n}\cup
    \set{f\in \Inj(X)}{|X\setminus \im(f)| < n}.
  \end{eqnarray*}
  The set $\set{f\in \Inj(X)}{|X\setminus \im(f)| > n}$ is open in
  $\S_0$ and hence is open in $\S_{n}$. The set
  \[
    \set{f\in \Inj(X)}{|X\setminus \im(f)| < n} = \bigcup_{m = 0} ^ {n - 1}
    F_{m}
  \]
  is also open by assumption.  Hence the complement of $F_n$ is open, and so
  $F_n$ is closed.  Therefore $\S_n$ is Polish also.
  Hence the topology generated by $\bigcup_{n = 0}
    ^ {\infty} \S_n$ is Polish by~\cite[Lemma 13.3]{Kechris1995aa}. Finally, by
  a
  similar argument, $F_{\aleph_0}$ is closed in the topology generated by
  $\bigcup_{n = 0} ^ {\infty} \S_n$, so by applying~\cite[Lemma
    13.2]{Kechris1995aa} again, $\mathcal{J}$ is Polish.

  Although it is not strictly necessary, we
  note that a complete metric on $\Inj(\N)$ that induces $\mathcal{J}$ is:
  $$d(f,g)=
    \begin{cases}
      0             & \text{if }f=g;
      \\
      1             & \text{if }|\N \setminus \im(f)|\neq |\N \setminus \im(g)|
      \\
      \frac{1}{m+1} & \text{otherwise};
    \end{cases}$$
  where $m=\min\set{n\in \N}{(n\times n) \cap f \neq (n\times n)\cap g}$.\medskip
  
  \noindent\textbf{(\ref{theorem-injective-vii}).}
  This part of the theorem follows immediately from parts \eqref{theorem-injective-v} and \eqref{theorem-injective-iv}, respectively.
\end{proof}


\begin{theorem}\label{theorem-surjective}
  Let $X$ be an infinite set,
  let $\S_1$ denote the topology on $\Surj(X)$ generated by the pointwise topology and the
  set $\Sym(X)$, and let $\S_2$ be the topology generated by
  the pointwise topology together with the collection of sets
  \[U_{\kappa, x}:= \{f\in \Surj(X): |(x)f^{-1}| = \kappa\}\] for all $x\in
    X$ and cardinals $\kappa \leq |X|$.
  Then the following hold:
  \begin{enumerate}[\rm (i)]
    \item\label{theorem-surjective-i}
          the pointwise topology and $\S_1$ are distinct semigroup topologies
          on $\Surj(X)$;
          
    \item\label{theorem-surjective-iv}
          if $X$ is countable, then $\S_1$ and $\S_2$ are Polish semigroup topologies;
          
    \item\label{theorem-surjective-ii}
    if $X$ is countable, then 
           there are infinitely many distinct Polish
          semigroup topologies on $\Surj(X)$ containing the pointwise topology and
          contained in $\S_2$;

    \item\label{theorem-surjective-iii}
          if $X$ is countable, then $\Surj(X)$ with the topology $\S_1$ embeds
          into
          $\N^\N$ with the pointwise topology.
  \end{enumerate}
\end{theorem}
\begin{proof}
  \textbf{(\ref{theorem-surjective-i}).}
  Since the pointwise topology is a semigroup topology for $X^X$, it is also a
  semigroup topology for $\Surj(X)$. If $X$ is countable, then $\Surj(X)$ is
  $G_{\delta}$ in $X ^ X$, and so it follows that the pointwise topology is a
  Polish semigroup topology on $\Surj(X)$.

  Since $\Inj(X)$ is closed in $X ^ X$, it follows that $\Sym(X) = \Inj(X) \cap
    \Surj(X)$ is closed in $\Surj(X)$ with respect to the pointwise topology.
  To show that $\S_1$ is compatible with the multiplication
  in $\Surj(X)$, it suffices to show that $\set{(f, g)}{fg\in \Sym(X)}$ is
  open in $\Surj(X)\times \Surj(X)$. It is routine to verify that
  $\Surj(X)\setminus \Sym(X)$ is an ideal in $\Surj(X)$. Hence if $f, g\in
    \Surj(X)$ are such that $fg\in \Sym(X)$, then $f, g\in \Sym(X)$. Hence
  $\Sym(X)\times \Sym(X) = \set{(f, g)}{fg\in \Sym(X)}$ is open, as
  required.\medskip

  \noindent
  \textbf{(\ref{theorem-surjective-iv}).}
  Since $\Sym(X)$ is closed in $\Surj(X)$ with respect to the pointwise topology,
  it follows by \cite[Lemma 13.2]{Kechris1995aa}, that $\S_1$ is Polish for countable $X$.
  
  It remains to show that $\S_2$ is a Polish semigroup topology when $|X|=
    \aleph_{0}$. We may assume without loss of generality that $X=\N$. First
  note that for all $f, g \in \Surj(\N)$ and $n\in \N$ we have that
  \[|(n)(fg)^{-1}|= \sum_{i\in (n)g^{-1}} |(i)f^{-1}|.\]
  If $fg\in U_{m, n}$, then there are three cases to consider: $m \in \N$;
  $m = \aleph_0 = |(n)g^{-1}|$; and $m = \aleph_0 \neq |(n)g^{-1}|$.
  \begin{enumerate}
    \item
          If $m\in \N$, then
          \[fg\in \left(\bigcap_{i\in (n)g^{-1}} U_{|(i)f^{-1}|, i}\right)\cdot
            \left(U_{|(n)g^{-1}|, n} \cap \bigcap_{i\in (n)g^{-1}}\{h\in
            \Surj(\N):
            (i)h=(i)g \}\right)\subseteq U_{m, n}.\]

    \item
          If $m=\aleph_0$ and $|(n)g^{-1}| = m$, then
          \[fg\in \Surj(X) \cdot U_{m, n} \subseteq U_{m, n}.\]

    \item
          If $m= \aleph_0$ and $|(n)g^{-1}| \neq \aleph_0$, then there is some
          $i\in (n)g^{-1}$ such that $|(i)f^{-1}|=\aleph_0$ and so
          \[fg\in U_{\aleph_0, i}\cdot \{h\in \Surj(\N): (i)h = n\} \subseteq U_{
                \aleph_0, n}.\]
  \end{enumerate}
  To show that $\S_2$ is Polish,  we define
  $\mathcal{T}_{m, n}$ to be the topology generated by the pointwise topology
  together with the sets $U_{i, n}$ for $i \leq m$ and for all $m, n \in \N$.
  We show by induction on $m$, that $\mathcal{T}_{m, n}$ is Polish for all $m,
    n \in \N$.

  By the definition of $\Surj(\N)$, $U_{0, n}= \varnothing$ for all $n\in
    \mathbb{N}$. It follows that, $\mathcal{T}_{0, n}$ is the pointwise
  topology, which is Polish, for all $n\in \mathbb{N}$. This establishes the
  base case of the induction.

  For the inductive step, suppose that $m>0$ and $\mathcal{T}_{m-1, n}$ is
  Polish for all $n\in \N$.
  The topology $\mathcal{T}_{m, n}$ is generated by
  $\mathcal{T}_{m - 1, n}$ and $U_{m, n}$. To apply \cite[Lemma
    13.2]{Kechris1995aa} and conclude that $\mathcal{T}_{m, n}$ is Polish,
  it suffices to show that $U_{m, n}$ is closed in $\mathcal{T}_{m - 1,
      n}$. By the definition of $U_{m, n}$,
  \[
    \Surj(\N) \backslash U_{m, n}  = \{f\in \Surj(\N): |(n)f^{-1}| < m\}
    \cup \{f\in \Surj(\N): |(n)f^{-1}| > m\}.
  \]
  The set $V = \{f\in \Surj(\N): |(n)f^{-1}| < m\}$ is the union of the open
  sets
  $\{f\in \Surj(\N): |(n)f^{-1}| = i\} = U_{i, n}$, for all $i < m$,
  and so $V$ is open in $\T_{m - 1, n}$. The set $W =  \{f\in \Surj(\N):
    |(n)f^{-1}| > m\}$ is the union of the open sets
  \[
    \{f\in \Surj(\N): (y)f = n \text{ for all }y\in Y\}
  \]
  in the pointwise topology, for all finite subsets $Y$ of $\N$ with at least
  $m + 1$ elements, and so $W$ too is open. Therefore $\Surj(\N) \backslash
    U_{m, n}$ is open and so $U_{m, n}$ is closed.

  We have shown that $\T_{m, n}$ is Polish for every $m, n\in \N$.
  For every $m\in \N$, we define $\mathcal{T}_m$ to be the least topology
  containing $\mathcal{T}_{m, n}$ for all $n \in \mathbb{N}$. By
  \cite[Lemma 13.3]{Kechris1995aa}, $\mathcal{T}_m$ is Polish for
  all $m \in \mathbb{N}$. By considering only case (1) in the proof that
  $\S_2$ is a semigroup topology, it follows that $\mathcal{T}_m$ is also a
  semigroup topology for all $m \in \mathbb{N}$.

  If $\mathcal{T}$ is the least topology containing
  $\mathcal{T}_m$ for all $m \in \N$, then, by \cite[Lemma
    13.3]{Kechris1995aa} again, $\mathcal{T}$ is Polish.
  If $n\in \N$ is arbitrary, then
  \[
    \Surj(\N)\setminus U_{\aleph_0, n}
    =  \{f\in \Surj(\N): |(n)f^{-1}| < \aleph_0\}
  \]
  which is the union of all of the open sets $U_{i, n}$, $i\in \N$, as in the
  inductive step above, is open in $\T$.
  Hence the topology generated by $\T$ and any of the sets
  $U_{\aleph_0, n}$, $n\in \N$,
  is a Polish topology. The topology generated by the (countable) union of all
  such topologies is $\S_2$,
  and so, by applying  \cite[Lemma 13.3]{Kechris1995aa} one last time, $\S_2$
  is Polish.
  \medskip

  \noindent
  \textbf{(\ref{theorem-surjective-ii}).}
  By the proof of part \eqref{theorem-surjective-iv}, it suffices to show that
  $\mathcal{T}_i$ and $\mathcal{T}_j$ are distinct for all $i\neq j$.
  Furthermore, since
  $\mathcal{T}_i \subseteq \mathcal{T}_j$ whenever $i \leq j$, it
  suffices to show that
  $\mathcal{T}_i \neq \mathcal{T}_{i+1}$ for every $i\in \N$.
  We show that $U_{i + 1, 0}$ is not open in $\mathcal{T}_i$. Suppose that
  $f\in U_{i + 1, 0}$ is arbitrary.
  Every basic open neighbourhoods of
  $f$ in $\mathcal{T}_i$ is of the form
  \[
    \bigcap_{(m, n) \in Z} U_{m, n} \cap \{g\in \Surj(\N): h\subseteq g\}
    = \{g\in \Surj(\N): h \subseteq g \text{ and }m = |(n)g^{-1}| \text{ for
      all
    }(m, n)\in Z\},\]
  where $h\in P_{\N}$  and $Z\subseteq \{1, 2
    \ldots, i\} \times \N$ are finite. In particular, every basic open
  neighbourhood of $f$ contains $g\in \Surj(\N)$ such that
  $|(0)g^{-1}| > i+1$ and hence no basic open neighbourhood of $f$ is
  contained in $U_{i+1, 0}$.
  \medskip

  \noindent
  \textbf{(\ref{theorem-surjective-iii}).}
  We will show that the topology $\S_1$ on $\Surj(\N)$ satisfies
  \cref{theorem-luke-1}\eqref{theorem-luke-1-i}, and so it is possible to topologically embed
  $\Surj(\N)$ with the topology $\S_1$ into $\N ^ \N$ with the pointwise
  topology. Suppose that $\sigma$ is the equivalence relation of $\Surj(\N)$
  with classes $\Sym(\N)$ and $\Surj(\N)\setminus\Sym(\N)$. Then $\sigma$ is a
  right congruence on $\Surj(\N)$ (even a two-sided congruence, but this is not
  important here).  If	$\set{\rho_i}{i\in \N}$ is a sequence of right
  congruences on $\Surj(\N)$ such that $\set{m/\rho_i}{m\in \Surj(\N)}$ is a
  subbasis for the pointwise topology, then $ \set{m/\rho_i}{m\in
      \Surj(\N)}\cup \{\sigma\}$ satisfies 
  \cref{theorem-luke-1}\eqref{theorem-luke-1-i}.
\end{proof}

\subsection{Diagram Monoids}\label{subsection-diagram}

In this section we consider topologies on the infinite partition monoid; see
\cite{east2014infinite,east2017infinite, fitzgerald2003presentation} for
previous literature on partition monoids. The so-called \textit{diagram
  monoids}, including the partition monoid, arose as the multiplicative monoids
of certain algebras in representation theory~\cite{HALVERSON2005869}. In some
sense, diagram monoids are also a generalisation of transformation monoids, in
a somewhat different direction than monoids of binary relations.
The results of this section are rather negative, in the sense that, we will
show that no second countable $T_1$ topology is semitopological for the
partition monoid.

For a set \(X\) we define the partition monoid \(\mathfrak{P}_X\) as follows.
The underlying set of \(\mathfrak{P}_X\) is the set of partitions of \(X \times
\{0, 1\} \). We will use partitions and the corresponding equivalence relations
interchangeably.
We define \(\operatorname{Diag}(s,t)\) to be the least equivalence relation on
\(X \times \{0, 1, 2\}\) containing
\[\{((x, a), (y, b))\in (X\times \{0,1,2\})^2:((x, a), (y, b))\in s \text{ or }
  ((x, a-1), (y, b-1))\in t\}.\]
The product of \(s\) and \(t\) is defined to be
\[\{((x,a),(y,b))\in (X\times \{0, 1\})^2: ((x,2a),(y,2b))\in
  \operatorname{Diag}(s,t)\}\]
and is denoted $st$. It is possible to show that this multiplication is
associative, and that the identity is
\[
  \set{\{(x, 0), (x, 1)\}}{x \in X}.
\]

To prove the main theorem in this section we require the following technical
lemma.

\begin{lem}\label{bad-semilattice}
  Let \(X\) be a set,  and let \(SL_X:=X\cup \{0\}\) where $0\not\in X$ be the
  semigroup with multiplication defined by
  \[
    st =
    \begin{cases}
      s & \text{if } s = t    \\
      0 & \text{if } s \not=t
    \end{cases}
  \]
  ($SL_X$ is a meet semilattice where every pair of non-zero elements is incomparable).
  If \(SL_X\) is a \(T_1\) semitopological semigroup, then
  the minimum size of a basis for $SL_X$ is at least $|X|$.
\end{lem}
\begin{proof}
  Suppose that $\T$ is a $T_1$ topology that is semitopological for $SL_X$.
  Let \(x\in X\). Then \(\{0\}\rho_x^{-1}= SL_X\backslash \{x\}\). As \(\T\) is
  \(T_1\), it follows that \(\{x\}\) is open. Thus the subspace topology on \(X\)
  is the discrete topology and so the minimum size of a basis is $|X|$.
\end{proof}

\begin{theorem}
  \label{thm-bipartition}
  Let $X$ be an infinite set. Then the following hold:
  \begin{enumerate}[\rm (i)]
    \item\label{thm-bipartition-i} if $X$ is countable, then \(\mathfrak{P}_X\) admits a second
          countable \(T_0\) semigroup topology;
    \item\label{thm-bipartition-ii}
          no second countable $T_1$ topology is semitopological for
          $\mathfrak{P}_X$.
  \end{enumerate}
\end{theorem}
\begin{proof}
\textbf{(\ref{thm-bipartition-i}).}
  We define $\T$ to be the topology on \(\mathfrak{P}_X\) with subbasis
  consisting of sets of the form:
  \[\{f\in \mathfrak{P}_X: (a, b)\in f\}\]
  for all $a, b\in X \times \{0, 1\}$.

  This topology is clearly \(T_0\). It is also second countable as it has been
  defined by a countable subbasis. Let \(f,g\in \mathfrak{P}_X\) and let \(U:=
  \{h\in \mathfrak{P}_X: (a, b)\in h\}\) be an arbitrary subbasic open
  neighbourhood of \(fg\). It suffices to find neighbourhoods \(U_f\) and \(U_g\)
  of $f$ and $g$, respectively, such that \(U_fU_g\subseteq U\).
  From the definition of the multiplication in \(\mathfrak{P}_{X}\) there is a
  finite path in $\operatorname{Diag}(f, g)$ \(f\) and \(g\) connecting \(a\) to
  \(b\). We can now choose \(U_f\) to be those elements of \(\mathfrak{P}_{X}\)
  which share all pairs belonging to \(f\) that are involved in this path. We may
  define \(U_g\) analogously.
  \medskip

 \textbf{(\ref{thm-bipartition-ii}).}
  For any partition \(P\) of \(X\), we define the partition \(P'\) of \(X
  \times \{0, 1\}\)
  to consist of the blocks $B\times \{0, 1\}$ for every block $B$ of $P$.
  If $S$ is the subsemigroup of $\mathfrak{P}_{X}$ consisting of those
  partitions $P'$ with at most two parts, then it is easy to verify that \(S\) is
  a semilattice isomorphic to \(SL_{2^{|X|}}\). Thus, by \cref{bad-semilattice},
  $\mathfrak{P}_X$ is not second countable.
\end{proof}

It might be worth noting that \cref{thm-bipartition}\eqref{thm-bipartition-ii} holds for the dual
symmetric inverse monoid and the factorisable dual symmetric inverse monoid,
the definitions of which are even longer than their names, and hence are
omitted.

\section{Monoids of continuous functions}\label{section-continuous}

If $X$ is a topological space, then we denote by $C(X)$ the monoid of
continuous functions from $X$ to $X$. The \textit{Cantor set} is $2 ^
  \N$ and the \textit{Hilbert cube} is $[0, 1] ^ \N$ where $[0, 1]$ is the
closed
unit interval in $\R$. In this section, we consider semigroup topologies on the
monoids $C([0, 1] ^ \N)$ and $C(2 ^ \N)$.

One standard topology on $C(X)$ is the compact-open topology, which we now
define.
If $X$ and $Y$ are topological spaces, then the \textit{compact-open topology}
on $C(X, Y)$ (the space of all continuous functions from $X$ to $Y$) is the
topology generated by the subbasis consisting of the sets:
\[
  [K, U] = \set{f\in C(X, Y)}{(K)f \subseteq U},
\]
where $K\subseteq X$ is compact, and $U\subseteq Y$ is open.

If $X$ is compact and metrizable with compatible metric $d$, then $C(X, X) =
  C(X)$ is separable with respect to the compact-open topology, and the
topology
is induced by
\begin{equation}
  \label{equation-infinity-metric}
  d_{\infty}(f, g) = \sup\set{d((x)f, (x)g)}{x\in X};
\end{equation}
see~\cite[Proposition~1.3.3]{Mill2001aa}. The space $X$ is complete since it is
compact, and $d_{\infty}$ is therefore complete also, meaning that $C(X)$ with
the compact-open topology is Polish.

\begin{theorem}\label{theorem-compact-open-hilbert}
  Let $X$ and $Y$ be compact metrizable spaces.  If $C(X,Y)$ is first countable
  and
  Hausdorff with respect to some topology $\T$, and $a : X
    \times C(X, Y) \to Y$ defined by
  \[
    (p, f)a = (p)f
  \]
  is continuous, then $\T$ contains the compact-open topology.
\end{theorem}
\begin{proof}
  It suffices to show that if $K$ is compact in $X$ and $U$ is open in $Y$,
  then the
  complement $F = \set{f\in C(X,Y)}{\exists k\in K,\
      f(k) \not\in U}$ of $[K, U]$ is closed in $\T$.

  Suppose that $f_0, f_1, \ldots \in F$ converges to
  $f\in C(X,Y)$ with respect to $\T$. For every $i\in \N$, there exists $x_i\in
    K$ such that $(x_i)f_i\not\in U$. Since $K$ is compact, it follows that
  $(x_i)_{i\in\N}$ contains a convergent subsequence $(x_{(i)y})_{i\in \N}$.
  If $x$ is the limit of this subsequence, then
  \[
    (x)f = (x, f) a
    = \left(\lim_{i\to \infty} x_{(i)y}, \lim_{i\to \infty}f_{(i)y}\right)a
    = \lim_{i\to \infty} (x_{(i)y}, f_{(i)y})a
    = \lim_{i\to \infty} (x_{(i)y})f_{(i)y} \in Y\setminus U
  \]
  and so $f\in F$, and $F$ is closed in $\T$, as required.
\end{proof}

\begin{lem}
  \label{lem-constants}
  Let $X$ be topological space and suppose that $C(X)$ is a first countable
  Hausdorff topological semigroup with respect to a topology $\T$. Then the set
  $F$ of constant functions in $C(X)$ is closed in $\T$.
\end{lem}
\begin{proof}
  It is routine to verify that $F$ is the set of right zeros in $C(X)$. If
  $f\in C(X)$ and $(k_i)_{i\in \N}$ is a sequence in $F$ that converges to
  $k\in C(X)$, then $fk = \lim_{i \to \infty} fk_i = \lim_{i\to \infty} k_i =
    k$
  and so $k\in F$.
\end{proof}

\subsection{The Hilbert cube}
\label{subsection-hilbert}

In this section, we show that the compact-open topology is the unique Polish
semigroup topology on the monoid $C(Q)$ of continuous functions on the
Hilbert cube $Q = [0, 1] ^ {\N}$.
Since the compact-open topology is Polish on the group of homeomorphisms of any
compact metric space (see, for example, \cite[Proposition 1.3.3]{Mill2001aa}),
it follows that the group of homeomorphisms $H(Q)$ of the Hilbert cube is a
Polish subgroup of $C(Q)$ with the compact-open
topology. In fact, it was shown by Kallman~\cite{Kallman1986aa} that the
compact-open topology is the unique Polish topology on $H(Q)$.

We will use the following metric on $Q$:
\begin{equation}
  \label{equation-hilbert-metric}
  d((x_0, x_1, \ldots), (y_0, y_1, \ldots)) =
  \sum_{i = 0} ^ \infty \frac{|x_i - y_i|}{2 ^ {i + 1}},
\end{equation}
which is bounded above by $1$, and so the corresponding $d_{\infty}$ is also
bounded above by $1$.

We require Theorem 5.2.4 from~\cite{Mill2001aa}, which we state for the sake of
completeness.

\begin{prop}[Theorem 5.2.4 in~\cite{Mill2001aa}]
  \label{prop-mill}
  Let $K_1$ and $K_2$ be compact subsets of $(0, 1) ^ \N$ and let $k: K_1 \to
    K_2$ be a homeomorphism with $d_{\infty}(k, \id_{K_1}) < \varepsilon$. Then
  there exists a homeomorphism $h_k \in H(Q)$ such that $d_{\infty}(h_k, \id_Q)
    < \varepsilon$ and $h_k|_{K_1} = k$.
\end{prop}

The next proposition provides a key step in the proof of the result of this
section.

\begin{prop}\label{prop-luke-3}
  Suppose that $C(Q)$ is a Polish semigroup with
  respect to some
  topology. Then the function $a : Q \times C(Q) \to Q$ defined by
  \[
    (p, f)a = (p)f
  \]
  is continuous.
\end{prop}
\begin{proof}
  Let $\T$ denote the given Polish semigroup topology on $C(Q)$.
  Suppose that $F$ denotes the set of constant functions in $C(Q)$ and
  that $\phi : F\to Q$ is the function that sends $f\in F$ to the
  unique value in $(Q)f$.

  We will show that $\phi$ is a homeomorphism, which
  will allow us to conclude the proof as follows. If $\gamma : Q\times C(Q)\to
    C(Q) \times C(Q)$ is defined by $(p, f)\gamma = (p \phi ^ {-1}, f)$, then
  $\gamma$ is continuous since it is continuous in each coordinate, and if $M:
    C(Q)\times C(Q) \to C(Q)$ is the multiplication function on $C(Q)$, then
  $M$
  is continuous also. Thus
  $(p, f)\gamma M \phi = (p\phi ^ {-1}, f) M \phi = (p\phi ^ {-1}\circ f)\phi$
  and $(p)\phi ^{-1}\in F$ is constant with value $p$, and so $p\phi ^
    {-1}\circ f$ is constant with value $(p)f$, and, finally, $(p\phi ^
    {-1}\circ
    f)\phi = (p)f = (p, f)a$. In other words, $a: Q \times C(Q) \to Q$ as a
  composite of the continuous functions $\gamma$, $M$, and $\phi$, is itself
  continuous.

  We start by showing that $\phi$ is continuous. If $U$ is an
  open set in $Q$, it suffices to show that $(U)\phi ^ {-1} = \set{f\in F}{(Q)f
      = \{q\},\ q \in U}$ is open in $F$.  If $f: Q\to Q$ is defined by
  \[
    (q)f = (d(q, Q\setminus U), d(q, Q\setminus U), \ldots),
  \]
  then since $f$ is continuous in every coordinate, $f$ is
  continuous. Since $C(Q)$ is $T_1$, $F$ is $T_1$, and so the singleton set
  containing the constant function $g: Q\to Q$ with value $(0, 0, \ldots)$ is
  closed. Hence $\{g\}\rho_f ^ {-1} \cap F$ is closed in $F$ (recall that $\rho_f: C(Q) \to C(Q)$  is defined by $(h)\rho_f = h \circ f$ for all $h\in C(Q)$), and so
  $F\setminus (\{g\}\rho_f ^ {-1})$ is open.
  If $h\in F\setminus (\{g\}\rho_f ^ {-1})$, then there exists $q\in Q$ such
  that
  \[
    (q)hf\not = (0, 0, \ldots)
  \]
  and so $(q)h\in U$ by the definition of $f$. On the other hand, if $h\in F$
  is such that $(q)h\in U$, then $(q)hf\not = (0, 0, \ldots)$, and so
  $F\setminus (\{g\}\rho_f ^ {-1}) = (U)\phi ^ {-1} = \set{f\in F}{(Q)f
      = \{q\},\ q \in U}$ is open, as required.

  Next, we show that $\phi ^ {-1}|_{(0, 1) ^ \N} : (0, 1) ^ \N \to F$ is
  continuous also. Since $F$ is closed in $C(Q)$ with respect to $\T$, by
  \cref{lem-constants}, it follows that $F$ is Polish. Since $\phi$ is
  continuous, $\phi$ is Borel measurable, and so by
  \cref{prop-borel-measurable}, $\phi ^ {-1}$ is Borel measurable also.
  It suffices, by \cref{prop-borel-imples-cts}, to endow
  $(0, 1) ^ \N$ and $((0, 1) ^ \N )\phi ^ {-1}\subseteq F$ with Polish
  semitopological group structures such that $\phi ^ {-1}|_{(0, 1) ^ \N} : (0,
    1) ^ \N \to ((0, 1) ^ \N )\phi ^ {-1}$ is a homomorphism. With the
  operation
  of component-wise addition,
  $\R ^ \N$ is an abelian topological group, and $\R ^ \N$ is homeomorphic to
  $(0, 1) ^ \N$. Therefore we may endow $(0, 1) ^ \N$ with the additive group
  structure of $\R ^ \N$ corresponding to an order-isomorphism between $\R$ and
  $(0, 1)$. We define $*$ on $((0, 1) ^ \N) \phi ^ {-1}$ by
  $x * y = ((x)\phi + (y)\phi)\phi ^ {-1}$. It remains to show that $((0, 1) ^
    \N) \phi ^ {-1}$ with $*$ is right topological. That $((0, 1) ^
    \N) \phi ^ {-1}$ with $*$ is left topological will then
  follow immediately since $\R ^ \N$ is abelian. Suppose that $g = (g_0, g_1,
    \ldots) \in (0, 1) ^ \N$ is arbitrary. Then $\rho_g: (0, 1) ^ \N \to (0, 1)
    ^
    \N$ is continuous since $(0,1) ^ \N$ is a topological group.  We extend
  $\rho_g$ to $\rho_g': Q \to Q$ so that the
  $i$th coordinate of $(x_0, x_1, \ldots)\rho_g'$ is $ x_i + g_i$ if $ x_i \in
    (0, 1)$ and $x_i$ if $x_i = 0$ or $x_i = 1$. Since $\rho_g'$ is an
  order-isomorphism onto $[0, 1]$ in every component, it follows that $\rho_g'$
  is continuous in every component, and is hence continuous, i.e.\ $\rho_g' \in
    C(Q)$. Thus $x \mapsto x \circ \rho_g'$, where $x\in C(Q)$ is continuous,
  and
  so, in particular, $x\mapsto x \circ \rho_g'$ restricted to $x\in ((0, 1) ^
    \N
    )\phi ^ {-1}$ is also continuous (on $((0, 1) ^ \N)\phi ^ {-1}$ with the
  subspace topology). If $x\in ((0, 1) ^ \N)\phi ^ {-1} \subseteq F$ is
  arbitrary, then $x$ is a constant function with value $(x_0, x_1, \ldots)\in
    (0, 1) ^ \N$ and so
  \[
    x * (g)\phi ^ {-1} = ((x)\phi + g)\phi ^ {-1}
    = (x)\phi \rho_g \phi ^ {-1}\]
  by the definition of $*$ and $\rho_g$. Recall that $(x)\phi = (x_0, x_1,
    \ldots) \in (0, 1) ^ \N$ by the definition of $\phi$.
  But
  \[(x)\phi \rho_g \phi ^ {-1} = (x_0, x_1, \ldots) \rho_g\phi ^ {-1}
    = (x_0 + g_0, x_1 + g_1, \ldots)\phi ^ {-1}
  \]
  and $(x_0 + g_0, x_1 + g_1, \ldots)\phi ^ {-1}$ is the constant function from
  $Q$ to $Q$ with value $(x_0 + g_0, x_1 + g_1, \ldots)$, which is equal to
  $(x) \circ \rho_g$. Therefore
  \[
    x * (g)\phi ^ {-1} = ((x)\phi + g)\phi ^ {-1}
    = (x)\phi \rho_g \phi ^ {-1}
    = x \circ \rho_g
    = x \circ \rho_g'
  \]
  and since $x \mapsto x \circ \rho_g'$ is continuous on $((0, 1) ^ \N)\phi ^
    {-1}$, it follows that $((0, 1) ^ \N) \phi ^ {-1}$ with $*$ is a right
  topological group.

  To conclude the proof, we must show that $\phi ^ {-1}: Q\to F$ is
  continuous. We define $\theta : [0, 1] \to [1/4, 3/4]$
  by $(x)\theta = x/2 + 1/ 4$ and let $\overline{\theta}\in C(Q)$ be the
  function
  that applies $\theta$ in every coordinate of $Q$. If $\theta': [0, 1]\to [0,
      1]$ is any continuous extension of $\theta ^ {-1}$, and
  $\overline{\theta'}
    \in C(Q)$ applies $\theta'$ in every coordinate, then $\theta\theta' =
    \id_Q$ and $\phi ^{-1} = \overline{\theta} \phi ^ {-1} |_{(0, 1) ^ \N}
    \rho_{\overline{\theta'}}$. Hence $\phi ^ {-1}$ being the composite of
  the continuous functions $\overline{\theta}$,  $\phi ^ {-1} |_{(0, 1) ^
    \N}$, and $\rho_{\overline{\theta'}}$ is itself continuous.
\end{proof}

\begin{theorem}\label{propx-hilbert}
  The monoid $C(Q)$ of continuous functions on the Hilbert cube $Q=[0,1]^{\N}$
  has the following properties:
  \begin{enumerate}[\rm (i)]
    \item\label{propx-hilbert_i} $C(Q)$ has property \textbf{X} with respect to
          its group of units $H(Q)$;
    \item\label{propx-hilbert_ii} the compact-open topology is the unique
          Polish semigroup topology on $C(X)$;
    \item\label{propx-hilbert_iii} if the group of homeomorphisms $H(Q)$ has
          automatic continuity with respect to a class of topological semigroups, then
          $C(Q)$ has automatic continuity with respect to the same class.
  \end{enumerate}
\end{theorem}
\begin{proof}
  \noindent\textbf{(\ref{propx-hilbert_i}).}
  If $S$ is $C(Q)$ with the compact-open topology, $\B$ is the basis
  for the compact-open topology on $C(Q)$ consisting of open balls with respect
  to the $d_{\infty}$ metric on $C(Q)$, and $T$ is $H(Q)$, then, we show that
  there
  exist $f, g\in C(Q)$ and $t_s\in H(Q)$ such that $s = ft_sg$ and for every
  ball $B =
    B_{d_{\infty}}(t_s, \varepsilon) \in \B$ there exists $\delta > 0$ such
  that
  $B_{d_{\infty}}(s, \delta) \subseteq f(B\cap H(Q))g$.

  Suppose that $s\in C(Q)$ is arbitrary. We define the required $f, g\in
    C(Q)$ via two continuous functions $f', g': [0, 1] \to [0, 1]$ that are
  defined by:
  \begin{equation*}
    (x) f' = \frac{2x + 1}{4}, \qquad (x)g' =
    \begin{cases}
      0                & \text{if } x \leq \frac{1}{4}                  \\
      \frac{4x - 1}{2} & \text{if } \frac{1}{4} \leq x \leq \frac{3}{4} \\
      1                & \text{if } x \geq \frac{3}{4}.
    \end{cases}
  \end{equation*}
  Then both $f'$ and $g'$ are continuous, and, moreover, $f'g'$ is the identity
  function $\id_{[0, 1]}$. We define $f, g, h \in C(Q)$ by applying $f'$ and
  $g'$ in
  certain coordinates, as follows:
  \begin{eqnarray*}
    (x_0, x_1, \ldots)f & = & ((x_0)f', (x_1)f', \ldots)\\
    (x_0, x_1, \ldots)g & = & ((x_1)g', (x_3)g', \ldots)\\
    (x_0, x_1, \ldots)h & = & ((x_0)g', (x_1)g', \ldots).
  \end{eqnarray*}
  Since $f'g' = \id_{[0, 1]}$, it follows that $fh = \id_Q$.
  Let $A = [1/4, 3/4] ^ \N$.  Of course, $A$ is homeomorphic to $Q$, and $A$ is
  a compact subset of $(0,
    1) ^ \N$.
  To find $t_s\in H(Q)$ such that $ft_sg = s$, we show that there exists a
  homeomorphism $t_s'$ from $A$ into a subspace of $A$, that can be extended
  to a homeomorphism of $Q$ by \cref{prop-mill}.  We denote by
  $\pi_i: Q\to [0, 1]$ the $i$th projection of $Q$, that is,
  \[
    (x_0, x_1, \ldots)\pi_i = x_i
  \]
  for any $(x_0, x_1, \ldots)\in Q$.  If $\vec{x} = (x_0, x_1, \ldots) \in A$
  and $k\in C(Q)$ is arbitrary,
  then we define $t_k': A \to A$ by
  \begin{equation}
    \label{equation-extendible}
    (\vec{x})t_k' = (x_0, x_1, \ldots)t_k'
    = (x_0, (\vec{x})hk\pi_0f', x_1, (\vec{x})hk\pi_1f', \ldots).
  \end{equation}
  From the definition of $t_k'$, if $\vec{x} \in Q$ is arbitrary, then
  \begin{eqnarray*}
    (\vec{x})ft_k'g & = & (x_0, x_1, \ldots)ft_s'g \\
    & = & ((x_0)f', (x_1)f', \ldots) t_s'g \\
    & = & ((x_0)f', (\vec{x})fhk\pi_0f', (x_1)f', (\vec{x})fhk\pi_1f', \ldots)g
    \\
    & = & ((x_0)f', (\vec{x})k\pi_0f', (x_1)f', (\vec{x})k\pi_1f', \ldots)g\\
    & = & ((\vec{x})k\pi_0f'g', (\vec{x})k\pi_1f'g', \ldots)\\
    & = & ((\vec{x})k\pi_0, (\vec{x})k\pi_1, \ldots)\\
    & = & (\vec{x})k
  \end{eqnarray*}
  and so $k = ft_k'g$.

  Since $t_k'$ is continuous in every component, $t_k'$ is continuous, and it
  is clearly a bijection onto its image. Since $A$ is compact, and every
  bijective continuous function between compact Hausdorff spaces is a
  homeomorphism, it follows that $t_k'$ is a homeomorphism for every $k\in
    C(Q)$. By our choice of metric for $Q$ in \eqref{equation-hilbert-metric},
  $d_{\infty}(t_s', \id_A) < 2$.  Hence by \cref{prop-mill}, we may
  extend $t_s'$ to $t_s\in H(Q)$.

  If $\varepsilon\in \R$, $\varepsilon > 0$, is arbitrary, then we will show
  that
  $
    B_{d_{\infty}}(s, \varepsilon) \subseteq f(B_{d_{\infty}}(t_s,
    \varepsilon)\cap H(Q))g$.

  If $a\in B_{d_{\infty}}(s, \varepsilon)$ and $t_a':A \to A$ is defined as in
  \eqref{equation-extendible}, then $ft_a'g = a$.
  Since $f': [0, 1]
    \to [0, 1]$ is a contraction and $d_{\infty}(a, s) < \varepsilon$, it
  follows that
  \begin{eqnarray*}
    d_{\infty}(t_a', t_s')
    & = & \sup\set{d(\vec{x}t_s', \vec{x}t_a')}{\vec{x}\in A} \\
    & = & \sup\left\{\sum_{i \in \N}
    \frac{|\vec{x}hs\pi_{i}f' - \vec{x}ha\pi_{i}f'|}{2 ^ {2i+2}} : \vec{x}\in
    A\right \} \\
    & \leq &\sup\left\{\sum_{i \in \N}
    \frac{|\vec{x}hs\pi_{i} - \vec{x}ha\pi_{i}|}{2 ^ {2i+2}} : \vec{x}\in
    A\right\} \\
    & = & \sup\left\{\sum_{i \in \N}
    \frac{|\vec{x}s\pi_{i} - \vec{x}a\pi_{i}|}{2 ^ {2i+2}} :
    \vec{x}\in (A)h = Q\right\}\\
    & \leq & \sup\left\{\sum_{i \in \N}
    \frac{|\vec{x}s\pi_{i} - \vec{x}a\pi_{i}|}{2 ^ {i+1}} :
    \vec{x}\in Q\right\} \\
    & = & d_{\infty}(a, s)  < \varepsilon.
  \end{eqnarray*}
  Then, since $(t_s') ^ {-1}: (A)t_s' \to A$ is a surjective function,
  \[
    d_{\infty}((t_s') ^ {-1} t_a', \id_{(A)t_a'}) =
    d_{\infty}((t_s') ^ {-1} t_a', (t_s') ^ {-1} t_s') = d_{\infty}(t_a', t_s')
    < \varepsilon,
  \]
  and so, by \cref{prop-mill},  there exists $\phi\in H(Q)$
  extending $(t_s') ^ {-1} t_a'$ such that $d_{\infty}(\phi, \id_{Q}) <
    \varepsilon$. It follows that $t_s\phi\in B_{d_{\infty}}(t_s, \varepsilon)$
  and $ft_s\phi g = ft_s'\phi g = ft_a'g = a$ and so
  $a \in f(B_{d_{\infty}}(t_s, \varepsilon)\cap H(Q))g$.
  \medskip

  \noindent\textbf{(\ref{propx-hilbert_ii}).}
  It follows immediately from
  \cref{theorem-compact-open-hilbert} and \cref{prop-luke-3}, that every
  Polish semigroup topology on $C(Q)$ contains the compact-open topology.
  It therefore suffices to show that no Polish semigroup
  topology for $C(Q)$ contains the compact-open topology properly. This follows
  from
  part \eqref{propx-hilbert_i} together with
  \cref{lem-luke-0}\eqref{lem-luke-0-iii}.
  \medskip

  \noindent\textbf{(\ref{propx-hilbert_iii}).}
  This follows immediately from part \eqref{propx-hilbert_i} and
  \cref{lem-luke-0}\eqref{lem-luke-0-iv}.
\end{proof}

\subsection{The Cantor space}
\label{subsection-cantor}

In this section, we show that the monoid of continuous functions $C(2 ^ \N)$ on
the Cantor space $2 ^ \N$ has a unique Polish semigroup topology -- the
compact-open topology. The proof is analogous to, but not the same as, the
proof of the uniqueness of the Polish semigroup topology on the Hilbert cube
given in the last section.

To prove that the compact-open topology is also a maximal Polish semigroup
topology on $2 ^ \N$ we require an analogue of \cref{prop-mill}.
We were not able to locate such an analogue in the literature,
and so we provide our own, the proof of which is similar to the proof of
\cref{prop-mill} given in~\cite{Mill2001aa}. In the following
proposition, we require the $d_{\infty}$ metric on $C(2 ^ \N)$ for which we
ought to fix a metric on $2 ^ \N$. Since $2 ^ \N$ is a subset of the Hilbert
cube, we can define the metric $d$ on $2 ^ \N$ to be the restriction of the
metric on the Hilbert cube defined in~\eqref{equation-hilbert-metric}. Then
\[
  d((x_0, x_1, \ldots), (y_0, y_1, \ldots)) =
  \sum_{i = 0} ^ \infty \frac{|x_i - y_i|}{2 ^ {i+1}},
\]
where $(x_0, x_1, \ldots), (y_0, y_1, \ldots) \in 2 ^ \N$.
We also require a metric on $2 ^ \N \times 2 ^ \N$ which we define by
\[
  \rho((a_0, a_1), (b_0, b_1)) =
  \max\{d(a_0, b_0), d(a_1, b_1)\}.
\]
Let the associated metric $\rho_{\infty}$ be defined as
in~\eqref{equation-infinity-metric}.
If $n > 0$, we denote the finite sequence of length $n$ consisting
solely of $0$
by $0 ^ n$, and we denote the infinite sequence
consisting solely of the value $0$ by $0 ^ {\infty}$.

\begin{prop}\label{prop-elliott-mill}
  Suppose that $A$ and $B$ are closed subsets of $2 ^ \N$ such that there
  exists
  a homeomorphism $\phi : A \to B$ such that $d_{\infty}(\phi, \id_A) <
    \varepsilon$.  Then there exists a homeomorphism $\phi': 2 ^ \N \times 2 ^
    \N
    \to 2 ^ \N \times 2 ^ \N$ such that $(a, 0 ^ {\infty}) \phi' = (a\phi, 0 ^
      {\infty})$ for all $a\in A$, and $\rho_{\infty}(\phi', \id_{2 ^ \N \times
      2 ^
      \N}) < \varepsilon$.
\end{prop}
\begin{proof}
  It is well-known that if $F$ is a non-empty closed subset of the
  Cantor space, then there exists a continuous function $g_F: 2 ^ \N \to F$
  such
  that $(x)g_F = x$ for all $x\in F$ (see, for example, the proof
  of~\cite[Lemma
    3.59]{Aliprantis2006aa}).

  Let $\varepsilon\in \R$, $\varepsilon > 0$, be fixed, and let $n\in \N$ be
  such that $\sum_{i = n} ^ {\infty} \frac{1}{2 ^ {i+1}} < \varepsilon$.
  We will view $2 ^ \N \times 2 ^ \N$ as $2 ^ \N \times (2 ^ n \times 2 ^ \N)$
  in the obvious way, but continue using the metric $\rho$.
  Moreover, let $+$ and $-$ be componentwise addition and subtraction modulo
  $2$ on $2^\N$. Define $\phi_0, \phi_1, \phi_2 : 2 ^ \N \times 2 ^ \N \to 2 ^
    \N \times 2 ^ \N$ by
  \begin{eqnarray*}
    (a, (b, c)) \phi_0 & = & (a, (b, c + a)) \\
    (a, (b, c)) \phi_1 & = & (a - (c)g_A + (c)g_A\phi, (b, c)) \\
    (a, (b, c)) \phi_2 & = & (a, (b, c - (a)g_B\phi ^ {-1})).
  \end{eqnarray*}
  Then $\phi_0, \phi_1$, and $\phi_2$ are continuous,
  since they are continuous in each component, and $\phi_0 ^ {-1}, \phi_1
    ^ {-1}$, and $\phi_2 ^ {-1}$  are well-defined and continuous. Hence
  $\phi_0,
    \phi_1$, and $\phi_2$ are homeomorphisms.
  The required homeomorphism $\phi': 2 ^ \N \times 2 ^ \N
    \to 2 ^ \N \times 2 ^ \N$ is $\phi' = \phi_0\phi_1\phi_2$ since for $a \in
    A$
  \begin{eqnarray*}
    (a, 0 ^ {\infty}) \phi_0\phi_1\phi_2 & = & (a, (0 ^n,  a)) \phi_1\phi_2 \\
    & = & (a - (a)g_A + (a)g_A\phi, (0 ^n,  a))
    \phi_2 \\
    & = & ((a)\phi, (0 ^ n,  a))\phi_2 \\
    & = & ((a)\phi, (0 ^ n, a - ((a)\phi g_B
    \phi ^ {-1}))) \\
    & = & ((a)\phi, 0 ^ {\infty}).
  \end{eqnarray*}

  It remains to show that $\rho_{\infty}(\phi', \id_{2 ^ \N \times 2 ^ \N}) =
    \sup \set{\rho((x, y)\phi', (x, y))}{x, y \in 2 ^ \N}  <
    \varepsilon$. By definition, \[\rho((x, y)\phi', (x, y)) = \max\{d((x,
    y)\phi'\pi_0, x), d(((x, y)\phi'\pi_1, y)\},\] where $\pi_0, \pi_1: 2 ^ \N
    \times 2 ^ \N \to 2 ^ \N$ are the projection functions.
  By the assumption on $n$, and the fact that in the definitions of $\phi_0$,
  $\phi_1$, and $\phi_2$ the first $n$ terms in the second component are not
  altered, it follows that \[
    d((x, y)\phi'\pi_1, y) < \sum_{i = n} ^ {\infty} \frac{1}{2 ^ {i+1}}
    < \varepsilon.\]  Since $\phi_0$ and $\phi_2$ do not change the first
  coordinate of their arguments, it suffices to check that
  $d((a, (b, c))\phi_1\pi_0, a) < \varepsilon$
  for all $(a, (b, c)) \in 2 ^ \N\times 2 ^ \N$.
  If $(a, (b, c)) \in 2 ^ \N\times 2 ^ \N$ is arbitrary,
  then the map $x \mapsto (c)g_A - a + x$  is an isometry on $2^\N$, with
  respect to the metric $d$, and so
  \begin{eqnarray*}
    d(a, (a, (b, c))\phi_1\pi_0) & = & d(a, a - (c)g_A + (c)g_A\phi) \\
    & = & d(((c)g_A - a) + a,
    ((c)g_A - a) +  a - (c)g_A +
    (c)g_A\phi)\\
    & = & d((c)g_A, (c)g_A\phi) \\
    & < & \varepsilon,
  \end{eqnarray*}
  since $d_{\infty}(\phi, \id_A) < \varepsilon$.
\end{proof}

\begin{theorem}\label{propx-cantor}
  The monoid $C(2^\N)$ of continuous functions on the Cantor set together with
  the compact-open topology has property \textbf{X} with respect to its group of
  units $H(2^\N)$.
\end{theorem}
\begin{proof}
  Similarly to the proof of \cref{propx-hilbert}\eqref{propx-hilbert_i}, we
  will show
  that for every $s\in C(2 ^ \N)$, there exist $f, g\in C(2 ^ \N)$ and
  $t_s\in H(2 ^ \N)$ such that $s = ft_sg$ and for every ball $B =
    B_{d_{\infty}}(t_s, \varepsilon)$ there exists $\delta > 0$ such that
  $B_{d_{\infty}}(s, \delta) \subseteq f(B\cap H(2 ^ \N))g$.

  Let $\Phi : 2 ^ \N \to 2 ^ \N \times 2 ^ \N$ be a homeomorphism. If
  $\varepsilon > 0$, then since $\Phi$ is homeomorphism between compact
  metric spaces, $\Phi$  and $\Phi ^ {-1}$ are uniformly continuous, and so
  there exists $r(\varepsilon) \in \R$
  such that $r(\varepsilon) > 0$ and the following conditions hold:
  \begin{enumerate}[(i)]
    \item\label{propx-cantor-proof-i}
          $d(a, b) < r(\varepsilon)$ implies that $\rho((a)\Phi, (b)\Phi) <
            \varepsilon$ for any $a, b\in 2 ^ \N$;
    \item\label{propx-cantor-proof-ii}
          $\rho(a, b) < r(\varepsilon)$ implies that $d((a)\Phi ^ {-1}, (b)
            \Phi ^
            {-1}) < \varepsilon$ for all $a, b \in 2 ^ \N \times 2 ^ \N$.
  \end{enumerate}
  Let $s\in C(2 ^ \N)$ be arbitrary, and let $f, g\in C(2 ^ \N)$ be defined by:
  \begin{eqnarray}
    \label{eq-f}
    (x)f & = & \left((x, x)\Phi ^ {-1}, 0 ^ {\infty}\right)\Phi ^{-1} \\
    \label{eq-g}
    (x)g & = & (x)\Phi \pi_0 \Phi \pi_0,
  \end{eqnarray}
  where, as defined in the proof of \cref{prop-elliott-mill}, $\pi_0: 2 ^ \N
    \times 2 ^ \N \to 2 ^ \N$ is the projection function.
  Note that $fg$ is the identity on $2 ^ \N$. Next, we will show that there
  exists $t_s\in H(2 ^ \N)$ such that $s = ft_sg$.

  Let $s', \Delta: 2 ^ \N \to 2 ^ \N \times 2 ^ \N$ be defined by
  \begin{equation}\label{eq-s-prime}
    (x)s' = ((x)s, x)
  \end{equation}
  and $(x)\Delta = (x, x)$. Clearly $\Delta$ is a homeomorphism from
  $2 ^ \N$ to $(2 ^ \N) \Delta$, and, $s'$ is injective and continuous, since
  it is continuous in each component, and the inverse of $s'$ is also
  continuous because $((x)s, x) \mapsto x$ is a projection. Hence $s'$ is a
  homeomorphism between $2 ^ \N$ and $(2 ^ \N) s'$. If $\phi = \Phi \Delta
    ^{-1} s' \Phi ^ {-1}$ and $x\in 2 ^ \N$ is arbitrary, then
  \begin{equation}
    \label{eq-phi}
    (x, x)\Phi ^ {-1} \phi = (x, x)\Phi ^ {-1} \Phi \Delta ^{-1} s' \Phi ^ {-1}
    = (x)s' \Phi ^ {-1}
    = ((x)s, x) \Phi ^ {-1}
  \end{equation}
  and so $\phi$ is a homeomorphism from $A = (2 ^ \N)\Delta \Phi ^ {-1} = \set{
      (x, x)\Phi ^ {-1} }{x \in 2 ^ \N}$ to $B = (A)\phi$.
  By \cref{prop-elliott-mill}, there exists a homeomorphism
  $\zeta: 2 ^ \N\times 2 ^ \N \to  2 ^ \N\times 2 ^ \N$ such that
  \begin{eqnarray}\label{eq-zeta}
    ((x, x)\Phi ^ {-1}, 0 ^ \infty) \zeta = ((x, x)\Phi ^ {-1}\phi, 0 ^
      {\infty})
  \end{eqnarray}
  for all $x\in 2 ^ \N$.

  We will show that the required homeomorphism of $2 ^ \N$ is
  \begin{equation}\label{eq-t-s}
    t_s = \Phi \zeta\Phi ^ {-1}.
  \end{equation}
  If $x\in 2 ^ \N$ is arbitrary, then
  \begin{equation*}
    \begin{array}{rclr}
      (x)ft_sg & =                         & ((x, x)\Phi^{-1}, 0 ^ {\infty})
      \Phi ^ {-1} t_s g
               & \text{by }\eqref{eq-f}
      \\
               & =                         & ((x, x)\Phi^{-1}, 0 ^ {\infty})
      \zeta \Phi^{-1} g
               & \text{by }\eqref{eq-t-s}
      \\
               & =                         & ((x, x)\Phi ^ {-1}\phi, 0 ^
        {\infty})\Phi ^{-1} g
               & \text{by }\eqref{eq-zeta}
      \\
               & =                         & ((x, x)\Phi ^ {-1}\phi, 0 ^
        {\infty})\Phi ^{-1} \Phi \pi_0
      \Phi \pi_0
               & \text{by }\eqref{eq-g}
      \\
               & =                         & (x, x)\Phi ^ {-1}\phi \Phi \pi_0
      \\
               & =                         & ((x)s, x) \Phi ^ {-1} \Phi \pi_0
               & \text{by }\eqref{eq-phi}
      \\
               & =                         & (x)s.
    \end{array}
  \end{equation*}

  Suppose that $\varepsilon > 0$ is arbitrary. It remains to show that
  there exists $\delta > 0$ such that $B_{d_{\infty}}(s, \delta) \subseteq
    f(B_{d_\infty}(t_s, \varepsilon) \cap H(2 ^ \N))g$. We set $\delta =
    r(r(\varepsilon))$. Let $u \in B_{d_{\infty}}(s, \delta)$. We define
  $u': 2 ^ \N \to 2 ^ \N \times 2 ^ \N$  by $(x)u' = ((x)u, x)$ for all $x\in 2
    ^ \N$, and, similar to the proof above for $s'$,  $u'$ is a homeomorphism
  from $2 ^ \N$ to $(2 ^ \N)u'$. Then $\rho_{\infty}(u', s') = d_{\infty}(u, s)
    < \delta$. Hence
  \[
    \rho_{\infty}(\id_{(2 ^ {\N})s'}, (s') ^ {-1} u')
    = \rho_{\infty}((s') ^ {-1} s' , (s') ^ {-1} u')
    = \rho_{\infty}(u', s') = d_{\infty}(u, s) < \delta
  \]
  since left multiplication by $(s') ^{-1}$ is an isometry. It follows that
  \[
    d_{\infty}(\Phi \id_{(2 ^ {\N})s'}\Phi ^ {-1},
    \Phi (s') ^ {-1} u'\Phi ^ {-1})
    < r(\varepsilon)
  \]
  by \eqref{propx-cantor-proof-i} and \eqref{propx-cantor-proof-ii}, from the second paragraph of this proof.

  By \cref{prop-elliott-mill} applied to the homeomorphism $\Phi
    (s') ^ {-1} u'\Phi ^ {-1} : (2 ^ \N)s'\Phi ^ {-1} \to (2 ^ \N) u'\Phi ^
    {-1}$, there exists $\gamma\in H(2 ^ \N \times 2 ^ \N)$ such that
  \begin{equation}
    \label{eq-gamma}
    (a, 0 ^ \infty) \gamma = (a\Phi (s') ^ {-1} u'\Phi ^ {-1}, 0 ^ \infty)
  \end{equation}
  and $\rho_{\infty}(\gamma, \id_{2 ^ \N \times 2 ^ \N}) < r(\varepsilon)$.
  Since left multiplication by $\Phi ^ {-1} t_s \Phi$ is an isometry of $C(2 ^
    \N \times 2 ^ \N)$,
  \[\rho_{\infty}(\Phi ^ {-1} t_s \Phi \gamma, \Phi ^ {-1} t_s \Phi)
    = \rho_{\infty}(\gamma, \id_{2 ^ \N \times 2 ^ \N})
    < r(\varepsilon)
  \]
  and so
  \[
    d_{\infty}(t_s \Phi  \gamma \Phi ^ {-1}, t_s)
    =
    d_{\infty}(\Phi \Phi ^ {-1} t_s \Phi \gamma\Phi ^ {-1},
    \Phi \Phi ^ {-1}t_s \Phi \Phi ^ {-1})
    < \varepsilon.
  \]
  Hence $f (t_s \Phi  \gamma \Phi ^ {-1}) g \in f(B_{d_{\infty}}(t_s,
    \varepsilon) \cap H(2 ^ \N))g$. We conclude the proof by showing that
  $f (t_s \Phi	\gamma \Phi ^ {-1}) g = u$. Suppose that $x\in 2 ^ \N$ is
  arbitrary. Then
  \begin{equation*}
    \begin{array}{rclr}
      (x) f (t_s \Phi  \gamma \Phi ^ {-1}) g
                    & =                            & ((x, x) \Phi ^ {-1}, 0 ^
        {\infty})\Phi ^ {-1}t_s \Phi  \gamma \Phi ^
      {-1} g        & \text{by }\eqref{eq-f}
      \\
                    & =                            & ((x, x) \Phi ^ {-1}, 0 ^
        {\infty}) \zeta\gamma\Phi ^ {-1} g
                    & \text{by }\eqref{eq-t-s}
      \\
                    & =                            & ((x, x) \Phi ^ {-1}\phi, 0
        ^ {\infty}) \gamma\Phi ^ {-1} g
                    & \text{by }\eqref{eq-zeta}
      \\
                    & =                            & ((x)s, x)\Phi ^ {-1}, 0 ^
        {\infty}) \gamma\Phi ^ {-1} g
                    & \text{by }\eqref{eq-phi}
      \\
                    & =                            & (((x)s, x)\Phi ^ {-1}\Phi
      (s') ^ {-1} u'\Phi ^ {-1}, 0 ^ {\infty})
      \Phi ^ {-1} g & \text{by }\eqref{eq-gamma}
      \\
                    & =                            & (((x)s, x) (s') ^ {-1}
      u'\Phi ^ {-1}, 0 ^ {\infty})
      \Phi ^ {-1} g
      \\
                    & =                            & ((x)u'\Phi ^ {-1}, 0 ^
        {\infty})\Phi ^ {-1} g
                    & \text{by }\eqref{eq-s-prime}
      \\
                    & =                            & (((x)u, x)\Phi ^ {-1}, 0 ^
      {\infty})\Phi ^ {-1} g                                                    \\
                    & =                            & (((x)u, x)\Phi ^ {-1}, 0 ^
        {\infty})\Phi ^ {-1}\Phi \pi_0\Phi\pi_0
                    & \text{by }\eqref{eq-g}
      \\
                    & =                            & ((x)u, x)\Phi ^ {-1}
      \Phi\pi_0                                                                 \\
                    & =                            & (x)u,
    \end{array}
  \end{equation*}
  which concludes the proof.
\end{proof}

\begin{theorem}\label{AC-cantor}
  The monoid of continuous functions on the Cantor set $2 ^ \N$ together with
  the compact-open topology has automatic continuity with respect to the class
  of second countable topological semigroups.
\end{theorem}
\begin{proof}
  This follows from \cref{propx-cantor}, \cref{lem-luke-0}, and the
  automatic continuity of the group of homeomorphisms of the Cantor set.
\end{proof}

\begin{prop}\label{prop-phi-is-homeo}
Suppose that $\T$ is a second countable Hausdorff semigroup topology on $C(2^{\N})$. Let $F$ be the subspace of constant functions in $C(2^{\N})$ and let $\phi:F \to 2 ^ \N$ (with the usual topology) be the function   mapping every $f\in F$ to the unique value in $(2^{\N})f$. Then $\phi$ is a homeomorphism.
\end{prop}
\begin{proof}
  Suppose that
  $\B$ is any basis for $2 ^ \N$ consisting of clopen sets and such that $2
    ^ \N\not\in \B$.

  To show that $\phi$ is continuous, suppose that $U\in\B$.
  It suffices to show that $(U)\phi ^ {-1} = \set{f\in
      F}{(2 ^ \N)f
      = \{q\},\ q \in U}$ is open in $F$.  If $y\not\in U$, then we define
  $f: 2 ^ \N\to 2 ^ \N$ such that
  \begin{equation*}
    (x)f =
    \begin{cases}
      x & \text{if } x \in U      \\
      y & \text{if } x \not\in U.
    \end{cases}
  \end{equation*}
  Since $U$ and $2 ^ \N \setminus U$ are open, and $f$ is continuous when
  restricted to either of the sets, it follows that $f$ is continuous.	Since
  $C(2 ^ \N)$
  is $T_1$, $F$ is $T_1$, and so the singleton set containing the constant
  function $g: 2 ^ \N\to 2 ^ \N$ with value $y$ is closed. Hence $\{g\}\rho_f ^
    {-1} \cap F$ is closed in $F$, and so $F\setminus (\{g\}\rho_f ^ {-1})$ is
  open in $F$.
  If $h\in F\setminus (\{g\}\rho_f ^ {-1})$, then there exists $x\in 2 ^ \N$
  such that
  \[
    (x)hf\not = y
  \]
  and so $(x)h\in U$ by the definition of $f$. On the other hand, if $h\in F$
  is such that $(x)h\in U$, then $(x)hf\not = y$, and so
  $F\setminus (\{g\}\rho_f ^ {-1}) = (U)\phi ^ {-1} = \set{f\in F}{(2 ^ \N)f
      = \{q\},\ q \in U}$ is open, as required.

  Next, we show that $\phi ^ {-1} :2 ^ \N \to F$ is continuous also. Note that
  the topology induced on $F$ by the compact-open topology is precisely the
  standard topology on $2^\N$ (when viewing an element of $F$ as its image
  under $\phi$). It therefore suffices to show that the topology induced on
  $F$ by $\mathcal{T}$ is contained in the topology induced on $F$ by the
  compact-open topology. This follows immediately from	 \cref{AC-cantor}.
\end{proof}

\begin{prop}\label{prop-cantor-action}
  Suppose that $C(2 ^ \N)$ is a second countable Hausdorff semigroup with
  respect to some topology. Then the function $a : 2 ^ \N \times C(2 ^ \N) \to
    2 ^ \N$ defined by
  \[
    (p, f)a = (p)f
  \]
  is continuous.
\end{prop}
\begin{proof}
  Let $\T$ denote the given second countable Hausdorff semigroup topology on $C(2 ^
    \N)$; $F$ denote the set of constant
  functions in $C(2 ^ \N)$; and $\phi : F\to 2 ^ \N$ be the function that sends
  $f\in F$ to the unique value in $(2 ^ \N)f$. By the same argument as given in the second paragraph of the proof of  \cref{prop-luke-3}, it suffices to show that $\phi$ is a homeomorphism, which follows by \cref{prop-phi-is-homeo}. 
\end{proof}

As immediate corollaries of the last proposition together with
\cref{theorem-compact-open-hilbert}, we obtain the following.

\begin{cor}\label{minimum-Polish-cantor}
  Every second countable Hausdorff semigroup topology on the monoid of
  continuous functions on the  Cantor space $2 ^ \N$ contains the compact-open
  topology.
\end{cor}

\begin{cor}\label{cor-cantor-unique}
  The compact-open topology is the unique second countable Hausdorff semigroup
  topology on the monoid of continuous functions on the Cantor set.
\end{cor}
\begin{proof}
  By \cref{minimum-Polish-cantor} every such topology contains the
  compact-open topology and by \cref{AC-cantor} every such topology is
  contained in the compact-open topology.
\end{proof}

If $S$ is a semigroup, then we denote by $S ^{\dagger}$ the \textit{dual
  semigroup} of $S$, that is, $S ^ {\dagger}$ is the set $S$ together with the
operation $\ast$ where $s \ast t$ is defined to be the product of $t$ and $s$
in $S$ for all $s, t\in S$. If $\T$ is topological for the semigroup $S$, then,
by \cref{prop-anti-automorph}, $\T$ is also topological for $S ^ {\dagger}$.

\begin{prop}\label{lem-anti-automatic-continuity}
  Let $\mathcal{C}$ be a class of topological semigroups, and let $\mathcal{C}
    ^ {\dagger}$ denote the class of dual semigroups of semigroups in
  $\mathcal{C}$.
  If $S$ is a topological semigroup, then the following are equivalent:
  \begin{enumerate}[\rm (i)]
    \item
          $S$ has automatic continuity with respect to $\mathcal{C}$;
    \item
          the dual semigroup $S ^ {\dagger}$ has automatic continuity with
          respect to $\mathcal{C} ^ {\dagger}$;
    \item every anti-homomorphism from $S$ to a semigroup in $\mathcal{C} ^
            {\dagger}$ is continuous.
  \end{enumerate}
\end{prop}
\begin{proof}
  The identity function $\iota : T \to T ^ {\dagger}$ is an anti-isomorphism
  from $T$ to its dual semigroup $T ^ {\dagger}$. The proof can be completed
  using \cref{prop-anti-automorph} and the fact that the composite of continuous
  functions is continuous.
\end{proof}

A \textit{Stone space} (also called a \textit{Boolean space}) is a compact
Hausdorff topological space with a basis of clopen sets. A classical result of
Stone~\cite{Stone1937} gives a correspondence between Stone spaces and Boolean
algebras. In modern terminology, this is an equivalence, called \textit{Stone
  duality}, between the category of Boolean algebras and homomorphisms, and the
dual category of Stone spaces and continuous functions; see, for
example,~\cite[Theorem 34]{givant2008introduction} for more details. Up to
isomorphism, there is only one countably infinite atomless Boolean algebra
which we denote by \(B_\infty\). This object is the Fra\"iss\'e limit of the
class of non-trivial finite Boolean algebras. The countably infinite atomless
Boolean algebra $B_{\infty}$ is also the Stone dual object of the Cantor space
$2 ^ \N$ and thus \(\End(B_{\infty})\) is anti-isomorphic to the monoid
\(C(2^\N)\).

\begin{theorem}\label{AC-Boolean}
  If $B_{\infty}$ is the countably infinite atomless Boolean algebra, then
  the following hold:
  \begin{enumerate}[\rm (i)]
    \item\label{AC-Boolean-i} the pointwise topology is the unique second countable Hausdorff
          semigroup topology on \(\End(B_{\infty})\);
    \item\label{AC-Boolean-ii} \(\End(B_{\infty})\)	has automatic continuity with respect to the
          pointwise topology and the class of second countable topological semigroups.
  \end{enumerate}
\end{theorem}
\begin{proof}
  \noindent \textbf{(\ref{AC-Boolean-i}).}
  Since \(\End(B_{\infty})\) is anti-isomorphic to \(C(2^\N)\), from
  \cref{prop-anti-automorph} and \cref{cor-cantor-unique}, it follows that
  \(\End(B_{\infty})\) has a unique second countable Hausdorff semigroup
  topology. The pointwise topology is a second countable and Hausdorff semigroup
  topology on $\End(B_{\infty})$, and hence the only one.
  \medskip

  \noindent \textbf{(\ref{AC-Boolean-ii}).}
  The automatic continuity then follows from
  \cref{lem-anti-automatic-continuity} and \cref{AC-cantor} as any homomorphism
  from \(\End(B_{\infty})\) can be factored into an anti-isomorphism to
  \(C(2^\N)\) followed by an anti-homomorphism from \(C(2^\N)\).
\end{proof}

\section{A little foray into the land of clones}\label{section-clones}

Further examples of function clones which have automatic homeomorphicity with respect to close subclones of $\mathscr{O}_\N$ can be found:
\begin{itemize}
  \addtocounter{enumi}{11}
  \item the countable generic posets $(P, \leq)$ and $(P, <)$;

  \item the rational Urysohn space;

  \item the rational Urysohn sphere;

  \item the countable random graph with all the loops added.
\end{itemize}

In this section we will extend some of the results in the preceding sections from monoids to clones, which are introduced in the following subsections.  

\subsection{Function clones}
If $C$ is any set, $g: C ^ m \to C$ for some $m\in \N\setminus \{0\}$ and $f_1,\ldots, f_m: C ^ n \to C$ for some $n\in \N\setminus\{0\}$, then we define $(f_1,\ldots, f_m) \circ_{m, n} g : C ^ n \to C$ by
\[
  (x_1,\ldots, x_n) \mapsto \left( (x_1,\ldots, x_n)f_1, \ldots, (x_1, \ldots, x_n)f_m \right)g.      \]
If the values of $m$ and $n$ are clear from the context, then we write $\circ$ instead of $\circ_{m, n}$.
We denote by $\pi ^ n_i : C ^ n \to C$ the \textit{$i$-th projection of $C ^ n$} which is defined by
\[
  (x_1, \ldots, x_n) \mapsto x_i.
\]

A \emph{(function) clone $\mathscr{C}$} with domain $C$ is a set of functions of finite
arity from $C$ to $C$ which is closed under composition  and also contains all
projections. More precisely, the following hold:
\begin{enumerate}[(i)]
  \item if $g \in \mathscr{C}$ is $m$-ary and $f_1,\ldots, f_m \in
          \mathscr{C}$ are $n$-ary, then
        $(f_1,\ldots, f_m) \circ g\in \mathscr{C}$;

  \item $\pi ^ n_i\in \mathscr{C}$ for every $n\in \N$ and every $i \in \{1, \ldots, n\}$.
\end{enumerate}
If $\mathscr{C}$ is a function clone with domain $C$ and $\T$ is a topology on $C$, then we will say that
$\T$ is \textit{topological} for $\mathscr{C}$, or $\T$ is a \textit{clone topology}, on $\mathscr{C}$ if the composition function $\circ_{m, n}$ is continuous for every $m, n\in\N$;
see~\cite{Bodirsky2017aa} for more information.

Let $\mathscr{O}_X$ denote \emph{the full function clone on a set $X$}, that is,
the set of all finite arity functions from $X$ to
itself. In the definition of $\mathscr{O}_X$ given here, following \cite{Bodirsky2017aa,kerkhoff2014short,wattanatripop2020clones},
nullary functions are not permitted.
Some authors, see, for example, \cite{behrisch2014clones,mckenzie2018algebras,trnkova2009all}, allow nullary operations in the definition of $\mathscr{O}_X$.

We define the \emph{topology of pointwise convergence} on
$\mathscr{O}_X$ in a similar way to its namesake for the full transformation
monoid $X ^ X$, with the subbasic opens sets of the form
\[
  U_{(a_1, \dots, a_n),b}=\{ f \in \mathscr{O}_X \colon (a_1, \ldots, a_n)f = b\}
\]
for some fixed $a_1, \ldots, a_n, b \in X$. Similarly to the monoid case, if $X$ is countably infinite, this topology is Polish and topological for the full function clone $\mathscr{O}_X$.

A \textit{polymorphism} of a structure $\mathbb{X}$ is a homomorphism from some finite and positive power of $\mathbb{X}$ into $\mathbb{X}$. The set of all polymorphisms of $\mathbb{X}$ forms a function clone on the set $X$ which is closed in $\mathscr{O}_X$. This clone is called the \textit{polymorphism clone of $\mathbb{X}$} and denoted by $\Poly(\mathbb{X})$.

By restricting to the morphisms in a fixed category, we can extend this definition to obtain a clone of polymorphisms, for any object, in any category of sets and functions which allows finite products. For example we consider the clone \(\Poly(2^\N)\) of continuous maps from finite, positive powers of the Cantor space to itself.

A \textit{clone homomorphism} is a map between function clones which preserves arities, maps projections to corresponding projections, and preserves the composition maps \(\circ_{m, n}\) for every \(n\) and \(m\).
A topological clone $\mathscr{C}$ is said to have \emph{automatic continuity} with
respect to a class of topological clones, if every homomorphism from $\mathscr{C}$ to a
member of that class is continuous. 
The following observation and lemma associate a semigroup to each clone. They will be used to show that several clones have automatic continuity by using the automatic continuity of the associated semigroups. 

Let $\mathscr{C}$ be a function clone which is topological with respect to
the topology $\mathcal{T}$. Define a semigroup $S_{\mathscr{C}}$ in the
following way: the elements of $S_{\mathscr{C}}$ are the functions in
$\mathscr{C}$; and the binary operation $*$ on $S_{\mathscr{C}}$ is given
by
\[
  (x_1,\ldots, x_n) f * g = ((x_1, \ldots, x_n)f, \ldots,
  (x_1, \ldots, x_n)f)g
\]
for all $x_1, \ldots, x_n \in C$, whenever $f$ is an $n$-ary function in
$\mathscr{C}$. It is routine to verify that $*$, as defined above, is
associative, and so $S_{\mathscr{C}}$ is indeed a semigroup. Moreover, all clone homomorphisms from a function clone to a function clone are also homomorphisms of the corresponding semigroups. Observe also, that if both $f$ and $g$ are unary functions, then $*$ is simply
the usual composition of functions. Finally, it follows from the fact that
composition of functions in $\mathscr{C}$ is continuous under
$\mathcal{T}$, that $\mathcal{T}$ makes $S_{\mathscr{C}}$ into a topological
semigroup.

\begin{lem}\label{lem:lift-ac}
  Let \(\mathscr{C}\) be a topological clone, and suppose that \(S_\mathscr{C}\) has automatic continuity with respect to the topology on \(\mathscr{C}\) and the class of second countable topological semigroups. Then \(\mathscr{C}\) has automatic continuity with respect to the class of second countable topological clones.
\end{lem}
\begin{proof}
  Suppose that \(\mathscr{A}\) is a second countable topological clone. Let \(f:\mathscr{C}\to \mathscr{A}\) be a clone homomorphism. It follows from the definitions of \(S_\mathscr{C}\) and \(S_\mathscr{A}\) that \(f\) is also a homomorphism between these semigroups.  The semigroups \(S_\mathscr{C}\) and \(S_\mathscr{A}\)  are topological with respect to the topologies on \(\mathscr{C}\) and \(\mathscr{A}\), respectively, as discussed before the lemma. Thus by the assumption that \(S_{\mathscr{C}}\) has automatic continuity, the map \(f\) is continuous.
\end{proof}

\subsection{The full function clone on a countably infinite set}
\label{subsection-full-clone}
Let $X$ be a countably infinite set. In Section \ref{subsection-full-transf} it was shown that the pointwise topology is the unique Polish semigroup topology on the full transformation monoid $X^X$ and that $X^X$ under the pointwise topology has automatic continuity with respect to the class of second countable topological semigroups. We will now show that the full function clone $\mathscr{O}_X$ has the analogous properties for clones. 
\begin{theorem}
  Let $X$ be a countably infinite set. The topological clone $\mathscr{O}_X$ with the pointwise topology has automatic continuity with respect to the class of second countable topological clones.
\end{theorem}

\begin{proof}
  As discussed above, the topology of pointwise convergence $\mathcal{P}$
  makes the semigroup $S_{\mathscr{O}_X}$ into a topological semigroup.
  Moreover, $\Sym(X) \subseteq \mathscr{O}_X$ and it has automatic continuity under
  the subspace topology of $\mathcal{P}$. Therefore, by~\cref{lem-luke-0}\eqref{lem-luke-0-iv},
  it
  suffices to show that $S_{\mathscr{O}_X}$ satisfies property \textbf{X} with
  respect to $\Sym(X)$. In order to do so fix $s \in \mathscr{O}_X$ of arity $n
    \geq 1$.

  Let $f$ be any injective $n$-ary function such that  the set $X \setminus
    \im(f)$ is infinite, and let $g$ be a unary surjective function such
  that $(x)g^{-1}$ is infinite for every $x \in X$. It is now routine
  to define a unary $t_s$ such that $s = f t_s g$. In particular, choose
  $t'$, an injective partial transformation on $X$, such that $\left( (\mathbf{x})f
    \right)t' \in \left( (\mathbf{x})s \right)g^{-1}$ for ever $\mathbf{x} \in X^n$. Since $t'$
  could have been chosen so that the sets $X \setminus \dom(t')$ and $X
    \setminus \im(t')$ are both countably infinite, $t'$ can be extended to
  $t_s \in \Sym(X)$, so that $s = f t_s g$.

  Let $B$ be a basic open ball containing $t_s$, that is,
  \[
    B = \{h \in X^X \colon (x)h = (x)t_s \text{ for all } x\in F\}
  \]
  for some fixed finite set $F \subseteq X$. Next, we will show that the
  open neighbourhood $U = \{u \in X^{X^n} \colon (\mathbf{x})u = (\mathbf{x})s \text{ for all } \mathbf{x}
    \in (F)f^{-1} \}$ of $s$ is contained in $f (B \cap \Sym(X)) g$. Let $u \in
    U$. Similarly to before, let $b'$ be an injective partial transformation on
  $X$ such that $\left((\mathbf{x})f \right)b' \in \left( (\mathbf{x})u \right)g^{-1}$ for all $\mathbf{x} \in X^n$.
  Moreover, this can be done in such a way that $(x)b' = (x)t_s$ for all $x
    \in F$. Since $F$ is finite, we may assume that $b'$ was chosen so that the
  sets $X \setminus \dom(b')$ and $X \setminus \im(b')$ are countably
  infinite. Hence $b'$ can be extended to $b \in \Sym(X)$, and therefore $u =
    f b g \in f (B \cap \Sym(X)) g$, as required.
\end{proof}

The proof of the next lemma is very similar to the proof of
\cref{thm-elliott-1}\eqref{thm-elliott-1-ii} $\Rightarrow$ \eqref{thm-elliott-1-iii} and omitted.
\begin{lem}\label{lem-from-semigroup-to-clone}
  Let $\mathscr{C}$ be a function clone on an infinite set $X$ containing all
  of the constant functions and such that for every $x\in X$ there exists a
  unary $f_{x}\in \mathscr{C}$ such that $(x) f_{x} ^ {-1} = \{x\}$ and
  $(X)f_{x}$ is finite.  If $\mathcal{T}$ is a $T_1$ clone topology of
  $\mathscr{C}$, then $\{f\in \mathscr{C} \colon (a_1, \ldots, a_n)f = b \}$
  is open in $\mathcal{T}$ for all $n \geq 1$ and all $a, b \in X$.
\end{lem}

Combining the lemmas above we obtain the following result.
\begin{cor}
  Let $X$ be a countably infinite set. The topology of pointwise convergence
  is the unique Polish topology which makes $\mathscr{O}_X$ into a
  topological clone.
\end{cor}

\subsection{The Horn clone}\label{subsection-horn-clone}
Having considered $\mathscr{O}_X$, the analogue of the semigroup of all transformations on $X$, we now consider the analogue of the monoid $\Inj(X)$ of injective functions from $X$ to $X$.
The function clone \(\mathscr{H}_X\) is defined to be the subclone of \(\mathscr{O}_{X}\) generated by the injective $n$-ary function for every arity $n\in \N \setminus \{0\}$.
In the literature, this object is referred to as the \textit{Horn clone}; see \cite{bodirsky2010reducts,Bodirsky2017aa}.

In this section, we show that the pointwise topology can be extended to infinitely many Polish topologies for \(\mathscr{H}_X\) when $X$ is countably infinite.
Among these Polish topologies on $\mathscr{H}_{X}$ there is a maximum one, and $\mathscr{H}_{X}$ has automatic continuity with respect to this topology and the class of second countable topological clones. 
The main theorem, and its proof, in this section are similar to~\cref{theorem-injective}.

For \(n\)-ary \(f\in \mathscr{H}_X\), there exists a non-empty subset \(I_f\) of \(\{1, \dots, n\}\) such that
\[(\mathbf{x})f=(\mathbf{y})f \quad\text{if and only if}\quad  (\mathbf{x})\pi^n_i=(\mathbf{y})\pi^n_i \text{ for all }i\in I_f\]
for all \(\mathbf{x}, \mathbf{y}\in X^{n}\).
The elements of \(\mathscr{H}_X\) are exactly characterised by the existence of such a non-empty set.

Since nullary operations are vacuously injective, if they were permitted in the definition of $\mathscr{O}_X$, then the sets $I_f$ in the definition of $\mathscr{H}_X$ could be empty. Although we have opted to disallow nullary functions in the definition of $\mathscr{O}_X$, and consequently require the sets $I_f$ to be non-empty in the definition of $\mathscr{H}_X$,  the results in this paper hold regardless for the other versions of \(\mathscr{H}_X\).

The following is the main result of this section.

\begin{theorem} \label{horn-clone}
  Let $X$ be a countably infinite set and let 
  $\mathcal{H}$ be the topology on the Horn clone $\mathscr{H}_{X}$ generated by the pointwise topology together with the sets 
\[ V_{y}:=\{f\in \mathscr{H}_{X}: y\notin \im(f)\} \quad\text{and}\quad
F_n := \{f\in \mathscr{H}_{X}: |X\backslash\im(f)|= n\}
\] 
for \( y\in X\) and $n\in \N \cup \{\aleph_0\}$. Then the following hold:
  \begin{enumerate}[\rm (i)]
    \item \label{horn-clone-i}
    $\mathscr{H}_{X}$ admits infinitely many distinct Polish clone topologies, 
    including  the pointwise topology and \(\mathcal{H}\);

    \item \label{horn-clone-ii} 
    $S_{\mathscr{H}_{X}}$ has property \textbf{X} with respect to $\mathcal{H}$ and $\Sym(X)$;

    \item \label{horn-clone-iii} 
    $\mathscr{H}_{X}$ has automatic continuity with respect to the topology $\mathcal{H}$
    and the class of second countable topological clones;
    
    \item \label{horn-clone-iv} 
    $\mathcal{H}$ is the maximum Polish clone topology on $\mathscr{H}_{X}$;

    \item \label{horn-clone-v} 
          if $\T$ is any clone topology on $\mathscr{H}_{X}$ and $\T$ induces
          the pointwise topology on $\Sym(X)$, then $\T$ is contained in
          $\mathcal{H}$.
  \end{enumerate}
\end{theorem}

We require the following three lemmas in the proof of~\cref{horn-clone}. 

\begin{lem}\label{lem-preimages-horn-clone}
  If $X$ is countably infinite, $f\in \mathscr{H}_{X} ^ {(n)}$,  and $x\in X$ is such that $(x)f ^ {-1}$ is non-empty, then 
  $(x)f ^ {-1} = T_1\times \cdots\times T_n$ where 
  $|T_i| = 1$ if $i\in I_f$ and $T_i = X$ otherwise.
\end{lem}
\begin{proof}
  Let \(T = (x)f^{-1}\) be non-empty and let \(\x \in T\). If $\y\in X ^ n$, then, by the definition of \(I_f\), \((\y)f=(\x)f\) if and only if  \((\x)\pi_i=(\y)\pi_i\) for all \(i\in I_f\). We define $T_i$
 to be $X$ if $i\not\in I_f$ and $\{(\x)\pi_i\}$ otherwise. Hence
  \[(x)f ^ {-1} = \set{\y\in X ^ n}{(\y)f = (\x)f}= \{\y\in X^n:(\y)\pi_i= (\x)\pi_i \text{ for }i\in I_f\} = T_1\times \cdots\times T_n,\]
  as required.
\end{proof}

\begin{lem}
\label{lem-hard-to-comprehend}
  If $X$ is countably infinite, then the sets 
  \[
   \mathscr{H}_{X}^{(n)} := \{f\in \mathscr{H}_{X} \colon \ar(f)=n\}\quad \text{and}\quad
  W_{S,n}  := \{f\in \mathscr{H}_{X}^{(n)} \colon I_f= S\}
  \]
  are open in the pointwise topology on $\mathscr{H}_{X}$ for all \( n\in \N\), and all finite \(S\subseteq \N\).
\end{lem}
\begin{proof}
  For \(n\in \N\), let \(\mathbf{x} \in X^n\) be fixed. Then the set $\mathscr{H}_{X}^{(n)}$ is the union of the subbasic open sets $U_{\mathbf{x}, y} = \set{f\in \mathscr{H}_{X}}{(\mathbf{x})f = y}$ (for the pointwise topology) for all $y\in X$. Hence $\mathscr{H}_{X}^{(n)}$ is open in the pointwise topology.
  
 Let $n\in \N$ be arbitrary and \(S\subseteq \{1, \ldots n\}\) be finite. For every \(i\in \{1, \ldots n\}\), we let \(\mathbf{x}_{i},\mathbf{y}_{i}\in X^{n}\) be such that \((\mathbf{x}_{i})\pi_j^ n \neq (\mathbf{y}_{i})\pi_j^ n\) if and only if \(i=j\) ($\mathbf{x}_{i}$ and $\mathbf{y}_{i}$ disagree only in their $i$th component). We will now show that
  \[W_{S,n}= \left(\bigcap_{i\in S}\{g\in \mathscr{H}_{X}^{(n)}: (\mathbf{x}_{i})g \neq (\mathbf{y}_{i})g\}\right)\cap\left(\bigcap_{i\in \{1, \ldots, n\}\backslash S}\{g\in \mathscr{H}_{X}^{(n)}: (\mathbf{x}_{i})g = (\mathbf{y}_{i})g\}\right).\]
  We denote the intersection in the previous displayed equations by $J$.
  Note that if $f\in \mathscr{H}_{X} ^ {(n)}$, then 
  \begin{equation}\label{thelabel}
  (\mathbf{x}_i)f \neq (\mathbf{y}_i)f\text{ if and only if }i \in I_f.
  \end{equation}
  If $f\in W_{S,n}$, then $f\in \mathscr{H}_{X} ^ {(n)}$ and $I_f = S$. Hence 
  $(\mathbf{x}_i)f \neq (\mathbf{y}_i)f$ if and only if $i \in S$. Hence $f\in J$. Conversely, if $f\in J$, then $f\in \mathscr{H}_{X} ^ {(n)}$ and so $I_f$ is well-defined. If $i \in S$, then, since $f\in J$, $(\mathbf{x}_i)f\neq (\mathbf{y}_i)f$ and so $i\in I_f$ by \eqref{thelabel}. Similarly, if $i\not \in S$, then $(\x_i)f = (\y_i)f$ and so $i \not \in I_f$. It follows that $I_f=S$ and we have shown that $W_{S, n} = J$.
\end{proof}

\begin{lem}
\label{lem-not-so-hard-to-comprehend}
  Let $X$ be countably infinite, and let \(\mathcal{H}_1\) be the topology on \(\mathscr{H}_{X} \) generated by the pointwise topology and the sets $V_y= \{f \in \mathscr{H}_X \colon y \notin \im(f)\}$ for all $y \in X$. Then the set
  \[P_{T, y, n}  := \{f\in \mathscr{H}_{X}^{(n)}\colon (y)f^{-1}= T\}\] is open in \(\mathcal{H}_1\) for every $y \in X$, $n\in \N$, and $T\subseteq X^ n$.
\end{lem}
\begin{proof}
  If \(T = \varnothing\), then \(P_{T, y,n}= V_y\) which is open in $\mathcal{H}_1$.
  If $T\neq \varnothing$ and $f\in P_{T, y, n}$, then, by \cref{lem-preimages-horn-clone}, 
  \(T = T_1 \times \cdots \times T_n\) 
  where \(T_i\) is a singleton for $i\in I_f$ and $T_i=X$ for $i\not \in I_f$. Note that $T$ does not depend on the choice of $f\in P_{T, y, n}$, and so, in particular, $I_f = I_g$ for all $f, g \in P_{T, y,n}$. We denote this set by $Y$ and fix any \(\x \in T\).
  Therefore \(P_{T, y, n} = U_{\x, y} \cap W_{Y, n}\),
  where $U_{\x, y}$ is one of the sub-basic open sets in the definition of the pointwise topology, and $W_{Y, n}$
  is open in the pointwise topology by \cref{lem-hard-to-comprehend}.
\end{proof}

\begin{proof}[Proof of \cref{horn-clone}]
  \noindent
  \textbf{(\ref{horn-clone-i}).}
  First, we demonstrate that \(\mathscr{H}_X\) is a closed subset of \(\mathscr{O}_X\) in the pointwise topology.
  Let \(f \in \mathscr{O}_X \setminus \mathscr{H}_X\) be \(n\)-ary. Then for every non-empty \(S \subseteq \{1, \ldots, n\}\) there are \(\mathbf{x}_S, \mathbf{y}_S \in X^n\) such that 
  \( (\mathbf{x}_S)f = (\mathbf{y}_S)f\) if and only if there exists $i\in S$ such that  \( (\mathbf{x}_S)\pi_i \not= (\mathbf{y}_S)\pi_i\). 
  Hence 
  \[
  \bigcap_{\varnothing \neq S\subseteq \{1, \ldots, n\}} U_{\x_S, (\x_S)f}\cap U_{\y_S, (\y_S)f}
  \]
  is an open neighbourhood of $f$ contained in $\mathscr{O}_X \setminus \mathscr{H}_X$. 
  Thus \(\mathscr{O}_X \setminus \mathscr{H}_X\) is open in the pointwise topology, and so \(\mathscr{H}_X\) is closed in \(\mathscr{O}_X\). Hence \(\mathscr{H}_X\) with the pointwise topology is a Polish clone.
  
  Next, we show that \(\mathcal{H}\) is a clone topology by showing that $\circ_{n,m}: (\mathscr{H}_{X}^{(n)})^{m} \times \mathscr{H}_{X}^{(m)} \rightarrow \mathscr{H}_{X}$ is continuous for every \(n, m \in \N\). 
  Since \(\mathscr{H}_X\) is a Polish clone with respect to the pointwise topology, the set \((U_{\mathbf{x}, y})\circ_{n, m}^{-1}\) is open in the pointwise topology, and hence in \(\mathcal{H}\).
  It suffices to consider 
  \(f_1, \dots, f_m \in \mathscr{H}_X^{(n)}\) and \(g\in \mathscr{H}_X^{(m)}\) such that 
  \((f_1, f_2, \ldots, f_m) \circ g\) belongs to a subbasic open set $B$ and either $B = V_{y}$ for some $y\in X$ or $B = F_n$ for some $n\in \N\cup \{\aleph_0\}$. We will find a neighbourhood of \(((f_1, \ldots, f_m), g)\) which is mapped into \(B\) by \(\circ_{n, m}\).

  Suppose that \(y \in X\) is such that \((f_1, \ldots, f_m) \circ_{n,m} g \in V_y\). 
  Then either:
  \(y \notin \im(g)\) or \[(( \mathbf{z})f_1, \ldots, (\z)f_m) \not \in (y)g^{-1}\] for all \(\mathbf{z}\in X^n\). In the former case, \((\mathscr{H}_{X}^{(n)})^m \times (\mathscr{H}_{X}^{(m)}\cap V_y)\) is the required neighbourhood. 
  
  In the latter case, it follows from \cref{lem-preimages-horn-clone}, that $(y)g ^ {-1} = T_1 \times \cdots \times T_m$ where $|T_i| = 1$ if $i \in I_g$ and $T_i = X$ if $i\not\in I_g$. Hence if \(\z \in X^n\) is arbitrary, then $(( \mathbf{z})f_1, \ldots, (\z)f_m) \not \in (y)g^{-1}$ and so
   there exists $i\in I_g$ such that  \(\{(\mathbf{z})f_i\} \neq T_i\). 
   In other words, \(\bigcap_{i\in I_g} (T_i)f_i^{-1} = \varnothing\).
  If \(i \in I_g\) and \(h_i \in P_{(T_i) f_i^{-1},T_i,n}\), which is an open set by \cref{lem-not-so-hard-to-comprehend},
  then $(T_i)h_i^{-1} =(T_i)f_i^{-1}$, and so 
  \(\bigcap_{i\in I_g} (T_i)h_i^{-1} = \bigcap_{i\in I_g} (T_i)f_i^{-1} = \varnothing\). Hence
  \[
   \left( \prod_{i = 1}^m A_i\right) \times P_{(y)g^{-1}, y, m},
  \]
  where $A_i =  P_{(T_i) f_i^{-1},T_i,n}$ if $i\in I_g$ and $A_i = \mathscr{H}_{X} ^{(n)}$ if $i\not\in I_g$, 
  is the required open neighbourhood of \(((f_1, \ldots, f_m), g)\).

  Suppose that \(\kappa \leq \aleph_0\) and \((f_1, \ldots, f_m) \circ_{n, m} g \in F_\kappa\). Note that 
  \begin{align*}
    \kappa &= |X \setminus \im((f_1, \dots, f_m) \circ_{n,m}g)| \\ 
    &=|X\setminus \im(g)|+|\{x \in \im(g) \colon 
     ((\y)f_1, \ldots, (\y)f_m) \notin (x)g^{-1} \text{ for all } \y \in X^n\} |.
  \end{align*}
  Denote the second cardinal in the sum above by \(\lambda\). Then, by  \cref{lem-preimages-horn-clone},
  \begin{align*}
    \lambda &= |\{x \in \im(g) \colon 
     ((\y)f_1, \ldots, (\y)f_m) \notin (x)g^{-1} \text{ for all } \y \in X^n\} | \\ 
     &= |\{x \in \im(g) \colon 
     \text{ for each } \y \in X^n \text{ there is } i \in I_g \text{ such that } (\y)f_i \notin (x)g^{-1}\pi_i \} | \\
    &= |\{x \in \im(g) \colon 
     \text{ there is no } \y \in X^n \text{ with } (\y)f_i \in (x)g^{-1}\pi_i  \text{ for all }i \in I_g\}|.
  \end{align*}
  For every \(x \in \im(g)\), define \(\alpha_x \in X^{I_g}\) to be such that \((i)\alpha_x\) is the unique element in \((x)g^{-1}\pi_i\) for every \(i \in I_g\).
  Then, by definition of \(I_g\), the map \(\Phi \colon  
  \im(g) \to X^{I_g}\), given by \(x \mapsto \alpha_x\), is an injection. Therefore
  \begin{align*}
    \lambda 
    &=  |\{\alpha_x \in X^{I_g} \colon x \in \im(g) 
    \text{ and there is no } \y \in X^n \text{ with } (\y)f_i \in (x)g^{-1}\pi_i  \text{ for all }i \in I_g\}|\\
    &=  |\{\alpha_x \in X^{I_g} \colon x \in \im(g) 
    \text{ and there is no } \y \in X^n \text{ with } (\y)f_i = (i)\alpha_x  \text{ for all }i \in I_g\}|
  \end{align*}
  Moreover, suppose \(\alpha \in X^{I_g}\) is such that there is no \(\y \in X^n\) with \((\y)f_i = (i)\alpha\) for all \(i \in I_g\). Then we can choose \(\z \in X^n\) such that \((\z)\pi_i = (i)\alpha\) for all \(i \in I_g\). Hence \(\alpha = \alpha_x\) where \(x = (\z)g\) and there is no \(\y \in X^n\) with \((\y)f_i = (i)\alpha_x\) for all \(i \in I_g\), and so the sets
  \(\{\alpha_x \in X^{I_g} \colon x \in \im(g) 
    \text{ and there is no } \y \in X^n \text{ with } (\y)f_i = (i)\alpha_x  \text{ for all }i \in I_g\}\) 
    and 
  \(\{\alpha \in X^{I_g} \colon \text{there is no } \mathbf{y}\in X^n \text{ with } (\mathbf{y})f_i = (i)\alpha \text{ for all } i\in I_g\}\) are equal.
  Therefore,
  \begin{equation}\label{equation-coimage-size}
    \kappa = |X\backslash \im(g)| + |\{\alpha \in X^{I_g} \colon \text{there is no } \mathbf{y}\in X^n \text{ with } (\mathbf{y})f_i = (i)\alpha \text{ for all } i\in I_g\}|.
  \end{equation}
  There are 4 cases to consider. Since $g\in \mathscr{H}_{X}$, $I_g \neq \varnothing$.
  \medskip

  \noindent \textbf{Case $\boldsymbol\alpha$.} Suppose that \(I_g = \{i\}\)
  for some \(i \in \{1, \ldots, m\}\).
  Then \(\kappa = |X\backslash \im(g)|+ |X\backslash \im(f_i)|\) by~\eqref{equation-coimage-size}. 
  Hence
  \[\left(\prod_{k = 1}^n A_k\right)
  \times \left(W_{I_g, m} \cap F_{|X \setminus \im(g)|}\right),\]
  where \(A_i = F_{|X\backslash \im(f_i)|} \cap \mathscr{H}_{X}^{(n)}\) and \(A_k = \mathscr{H}_{X}^{(n)}\) if \(k \neq i\),
   is the required neighbourhood.\medskip
  
  \noindent\textbf{Case $\boldsymbol\beta$.} Suppose there are distinct \(i,  j\in I_g\) with \(I_{f_i}\cap I_{f_j}\neq \varnothing\). We will show that \(\kappa = \aleph_0\).
  
  Suppose that \(k\in I_{f_i}\cap I_{f_j}\). Let \((
  \x_{r})_{r\in \N}\) be an infinite sequence of pairwise distinct  tuples in \(X^n\) such that \((\x_a)\pi_l =(\x_b)\pi_l\) for all \(a, b \in \N\) and all \(l\neq k\). There is at most one \(a \in \N\) with \((\x_1)f_i=(\x_a)f_j\) as \((\x_a)f_j \neq (\x_b)f_j\) whenever \(a \neq b\). For each \(a\) with \((\x_1)f_i\neq (\x_a)f_j\) we choose an \(m\)-tuple \(\y_a\) whose $i$th entry is \((\x_1)f_i\) and whose $j$th entry is \((\x_a)f_j\). Suppose that there is \(\z \in X^n\) such that
  \((\y_a)g = ((\z)f_1, \ldots, (\z)f_m) \circ g\). Then
  \((\x_1)f_i = (\y_a)\pi_i = (\z)f_i\) and \((\x_a)f_j = (\y_a)\pi_j = (\z)f_j\) since \(i, j \in I_g\).
  Moreover, it follows that \( (\x_1)\pi_k = (\z)\pi_k = (\x_a)\pi_k\) since \(k \in I_{f_i} \cap I_{f_j}\), which is a contradiction if \(a > 1\). Therefore, the points \((\y_a)g\) are distinct and not in the image of \((f_{1}, \ldots, f_m)\circ g\) for \(a>1\), and thus
  \(\kappa = \aleph_0\).
  
  Then
  \[\left(\prod_{k = 1}^m W_{I_{f_k}, n}\right)\times W_{I_g,m}\] is the required neighbourhood.\medskip

  \noindent \textbf{Case $\boldsymbol\gamma$.} Suppose that \(|I_g| \geq 2\) and  \(I_{f_i}\cap I_{f_j}= \varnothing\) for all distinct \(i, j\in I_g\), and suppose further that \(f_i\) is surjective for all \(i\in I_g\). Then for every \(\alpha \in X^{I_g}\) we may find \(\mathbf{y} \in X^n\) such that \( (\mathbf{y})f_i = (i)\alpha\) for all \(i \in I_g\), by simply choosing the appropriate values on each set of coordinates \(I_{f_i}\) separately. Hence \(\kappa = |X \backslash \im(g)|\) by~\eqref{equation-coimage-size}. Let \(A_k=W_{I_{f_k}, n}\cap F_{0}\) for \(k\in I_g\) and let \(A_k=\mathscr{H}_X^{(n)}\) for \(k\notin I_g\). Then
  \[\left(\prod_{k = 1}^m   A_k   \right) \times (W_{I_g, m}\cap F_{|X \backslash \im(g)|})\]
  is the required neighbourhood.\medskip

  \noindent\textbf{Case $\boldsymbol\delta$.} Suppose that \(|I_g| \geq 2\) and  \(I_{f_i}\cap I_{f_j}= \varnothing\) for all distinct \(i, j\in I_g\), and suppose further that \(f_i\) is not surjective for some \(i\in I_g\). Then \(\aleph_0\geq \kappa \geq  |X \backslash \im(g)| + |(X\backslash \im(f_i)) \times X^{|I_g|-1}|= \aleph_0\) by~\eqref{equation-coimage-size}. Hence \(\kappa = \aleph_0\), and so
  \[\left(\prod_{k = 1}^m W_{I_{f_k},n}\cap F_{|    X\backslash \im(f_k)|}\right) \times (W_{I_g,m}\cap F_{|X \backslash \im(g)|})\]
  is the required neighbourhood. Therefore, concluding the argument that \(\mathcal{H}\) is a clone topology.

  For each \(y\in X\), the set \(V_y\) is closed in the pointwise topology. It follows by~\cite[Lemma 13.2]{Kechris1995aa} that for each \(y\in X\) the topology generated by the pointwise topology and \(V_y\) is Polish, hence by~\cite[Lemma 13.3]{Kechris1995aa} the topology \(\mathcal{H}_1\) generated by the pointwise topology together with all of them is Polish.

  As in the proof of~\cref{theorem-injective}, we can now define a sequence of topologies \(\S_0 = \mathcal{H}_1, \S_1, \ldots\) 
  where \(\S_{n + 1}\) is generated by \(\S_{n}\) and \(F_{n}\). By an argument similar to the one in the proof of~\cref{theorem-injective}, it can
  be shown that \(F_n\) is closed in \(\S_n\) and \(F_{\aleph_0}\) is closed in the topology generated by \(\bigcup_{n=0}^\infty \mathcal{S}_n\). Hence
  all of these topologies \(\S_1, \S_2, \dots\)  as well as \(\mathcal{H}\) are Polish. Moreover, it follows from the proof that \(\mathcal{H}\) is a clone topology that \(\S_0, \S_1, \dots\) are also clone topologies. Finally, when restricted to $\Inj(X)$, the topologies in this proof agree with the corresponding topologies in~\cref{theorem-injective}, which are all distinct. It follows that the clone topologies in this proof are distinct also. 
  \medskip
  
  \noindent\textbf{(\ref{horn-clone-ii}).}
  Let $s\in S_{\mathscr{H}_X}$ be \(n\)-ary for some \(n \in \N\). We show that there 
  exist $f_s, g_s\in S_{\mathscr{H}_X}$ and $a_s\in \Sym(X)$ such that $s = f_s a_s g_s$ and
  for every basic open neighbourhood $B$ of $a_s$ there exists an open
  neighbourhood $U$ of $s$ such that $U \subseteq f_s (B\cap \Sym(X))g_s$.

  Let $f_s= s$,  let $a_s$ and $g_s$ both be the identity map on \(X\), and let $B$ be a
  basic open neighbourhood of $a_s$. Then there is some
  finite $Y\subseteq X$ such that
  \[B\cap \Sym(X) = \{h\in \Sym(X): (y)h = y \text{ for all }y\in Y\}.\]
  It can be shown that $h \in \mathscr{H}_X^{(n)}$ is an element of $f_s (B\cap \Sym(X))g_s = s (B\cap \Sym(X)$ if and only if the following hold
  \begin{enumerate}
    \item  $y\notin \im(h)$ for all $y\in Y \backslash \im(s)$;

    \item $(y)s^{-1}= (y)h^{-1}$  for all $y\in Y \cap \im(s)$;

    \item $|X \backslash \im(s)|= |X \backslash \im(h)|$;

    \item $I_h= I_s$.
  \end{enumerate}
  In the reverse implication, the required element of \(B\cap \Sym(X)\) can be chosen to be any bijection extending the partial bijection mapping \(x\in \im(s)\) to the single point in the set \((x)s^{-1}h\).
  
  It follows that
  \[
    f_s (B\cap \Sym(X))g_s = \bigcap_{y \in Y \setminus \im(f_s)} V_y \cap \bigcap_{y \in Y\cap \im(f_s)} P_{(y)s^{-1}, y, n} \cap F_{|X \setminus \im(s)|} \cap W_{I_s, n},
  \]
  and so $f_s (B\cap \Sym(X))g_s$ is an open neighbourhood of $s$ in \(\mathcal{H}\), as required.\medskip
  
  \noindent\textbf{(\ref{horn-clone-iii}).}
 From part \eqref{horn-clone-ii} the monoid $S_{\mathscr{H}_{X}}$ has property \textbf{X} with respect to $\Sym(X)$ and the topology $\mathcal{H}$. Property \textbf{X} implies automatic continuity by~\cref{lem-luke-0}, and so  $S_{\mathscr{H}_{X}}$ has automatic continuity with respect to the class of second countable topological semigroups. It follows from \cref{lem-from-semigroup-to-clone} that $\mathscr{H}_{X}$ has automatic continuity with respect to the class of second countable topological clones.\medskip
  
  \noindent\textbf{(\ref{horn-clone-iv}).}
  This follows immediately from part \eqref{horn-clone-iii}.\medskip
  
 \noindent\textbf{(\ref{horn-clone-v}).}
Let $\mathcal{T}$ be a clone topology for \(\mathscr{H}_X\) that induces the pointwise topology on $\Sym(X)$. Every clone topology for \(\mathscr{H}_X\) is also a semigroup topology for \(S_{\mathscr{H}_X}\). By part~\eqref{horn-clone-ii}
and~\cref{lem-luke-0}, every semigroup topology on \(S_{\mathscr{H}_{X}}\) which induces the pointwise topology on \(\Sym(X)\), such as $\mathcal{T}$, is contained in \(\mathcal{H}\). 
\end{proof}


\subsection{The clones of the Cantor space and the countably infinite atomless Boolean algebra}
\label{subsection-poly-clone}
In this section, we consider the polymorphism clone of the Cantor space $2 ^ \N$ which is defined as
\[\Poly(2^\N):=\bigcup_{n\in \N\backslash \{0\}} C((2^\N)^n, 2^\N),\]
where $C(X,Y)$ denotes the set of continuous functions from a topological space $X$ to a space $Y$. The polymorphism clone of the countably infinite atomless Boolean algebra, which is \[\Poly(B_{\infty}):=\bigcup_{n\in \N\backslash \{0\}} \Hom((B_\infty)^n, B_\infty)\]
where \(\Hom((B_\infty)^n, B_\infty)\) denotes the set of homomorphisms from \((B_\infty)^n\) to \(B_\infty\). For more information on this Boolean algebra see~\cite{givant2008introduction}.

The compact-open topology on \(\Poly(2^\N)\) is defined  
so that for each \(n\in \N\backslash\{0\}\) we consider the function space \(C((2^\N)^n, 2^\N)\) with the compact-open topology (as defined previously) and view \(\Poly(2^\N)\) as the disjoint union of these spaces with the union topology. One can verify that this topology makes \(\Poly(2^\N)\) a topological clone.

The clone \(\Poly(2^\N)\) is an analogue of the monoid $C(2^{\N})$ of continuous functions on the Cantor space $2 ^ \N$. In~\cref{AC-cantor}, it is shown that \(C(2^\N)\) with the compact-open topology has automatic continuity with respect to the class of second countable topological semigroups and that the compact-open topology is the unique second countable Hausdorff semigroup topology on $C(2^{\N})$. The Stone duality theorem gives an anti-isomorphism (see, for example,~\cite[Theorem 34]{givant2008introduction}) between the (topological) semigroups \(C(2^\N)\) and \(\End(B_\infty)\), which was used in \cref{AC-Boolean} to prove analogues of the results about $C(2 ^ \N)$ for \(\End(B_\infty)\).

Stone duality associates products of Boolean algebras with coproducts (disjoint unions) of Stone spaces. Although the Stone dual of \(\End(B_\infty)\) is \(C(2^\N)\), the Stone dual of \(\Poly(B_\infty)\) is not \(\Poly(2^\N)\). Instead, the Stone dual $\Poly(B_\infty) ^ {\dagger}$ 
of $\Poly(B_\infty)$ is the disjoint union
  \begin{equation}\label{eqn-stone-dual-poly}
  \bigcup_{n\in \N\backslash\{0\}} C(2^\N, 2^\N \times \{1, \ldots, n\}),
  \end{equation}
  and operation \(\pentagram_{n, m}\) defined by 
    \[(x) \big(f\pentagram_{n, m} (g_1, \dots, g_n)\big)=((x)f\pi_1)g_{(x)f\pi_2}\]
  for every \(f\in C(2^\N, 2^\N \times \{1, \ldots, n\})\), \(g_1, \dots, g_n \in C(2^\N, 2^\N \times \{1, \ldots, m\})\) and \(x\in 2^\N\) such that  \((x)f= ((x)f\pi_1,(x)f\pi_2)\in 2 ^ \N \times \{1, \ldots, n\}\).
  In this context, for every \(n\in \N\), the Stone dual of \(\Hom(B_{\infty}^n, B_\infty)\) is \(C(2^\N, 2^\N \times \{1, \ldots, n\})\). Moreover, the Stone duality translates the pointwise topology to the compact-open topology and vice-versa.
While the results in this section are stated in terms of \(\Poly(2^\N)\) and \(\Poly(B_\infty)\), one of the proofs related to \(\Poly(B_\infty)\) is given in terms of $\Poly(B_\infty) ^ {\dagger}$, which is the reason for introducing  $\Poly(B_\infty) ^ {\dagger}$.
 It seems to the authors of this paper that the two clones \(\Poly(2^\N)\) and \(\Poly(B_\infty)\) are in fact meaningfully distinct.

To prove the main theorem of this section, we require the following proposition, which is similar to~\cref{prop-cantor-action}.
\begin{prop}\label{prop-cantor-action-clone}
  Suppose that $\Poly(2 ^ \N)$ is a second countable Hausdorff topological clone with
  respect to some topology $\T$. Then for all \(n\in \N\) the function $a_n : (2 ^ \N)^n \times C((2^\N)^n, 2 ^ \N) \to
    2 ^ \N$ defined by
  \[
    (\x, f)a_n = (\x)f
  \]
  is continuous where $(2 ^ \N)^n \times C((2^\N)^n, 2 ^ \N)$ has the product topology given by the subspace topology on $C((2^\N)^n, 2 ^ \N)$ induced by $\T$ and the usual topology on $2 ^ \N$.
\end{prop}
\begin{proof}
  If we restrict \(\Poly(2^\N)\) to its elements of arity \(1\), then we obtain \(C(2^\N)\). By \cref{prop-phi-is-homeo},
   if $F$ denotes the set of constant maps of arity \(1\), and $\phi : F\to 2 ^ \N$ is the function that sends
  $g\in F$ to the unique value in $(2 ^ \N)g$, then \(\phi\) is a homeomorphism.

  Let \(\alpha:(2^\N)^n\to F^n\) be defined by $(x_1, x_2, \ldots, x_n)\alpha = ((x_1)\phi ^ {-1}, (x_2)\phi ^ {-1}, \ldots, (x_n)\phi ^ {-1})$. Then $\alpha$ is continuous since it is continuous in every component. If \(\beta:(2^\N)^n \times C((2^\N)^n,2^\N) \to F^n \times C((2^\N)^n,2^\N)\) is defined by $(\x, f)\beta = ((\x)\alpha, f)$ for all $\x\in (2^\N)^n$, $f\in C((2^\N)^n,2^\N)$, then $\beta$ is also continuous. 
  Then \(a_n\) is the composite of the continuous functions \(\beta\), the function clone \(\circ_{1, n}\), and \(\phi\), and hence, \(a_n\) is continuous.
\end{proof}

\begin{theorem}\label{AC-cantor-clone}
  Let \(\mathcal{K}\) be the compact-open topology on \(\Poly(2^\N)\). Then the following hold:
  \begin{enumerate}[\rm (i)]
      \item\label{AC-cantor-clone-i} \(S_{\Poly(2^\N)}\) has property \textbf{X} with respect to \(\mathcal{K}\) and \(C(2^\N)\);
      
      \item\label{AC-cantor-clone-ii} the clone \(\Poly(2^\N)\) has automatic continuity with respect to the class of second countable topological clones;
      
      \item\label{AC-cantor-clone-iii} \(\mathcal{K}\) is the only second countable Hausdorff clone topology on \(\Poly(2^\N)\).
  \end{enumerate}
\end{theorem}
\begin{proof}  
   \textbf{(\ref{AC-cantor-clone-i}).}  Let \(s\in \Poly(2^\N)\) be arbitrary and of arity \(n\). Let \(f_s:(2^\N)^{n} \to 2^\N\) be a homeomorphism, \(t_s:= f_s^{-1}s\) and \(g_s:= 1_{2^\N}\) (the identity function on $2 ^ \N$). Then \(f_st_sg_s=f_sf_s^{-1}sg_s=s\), \(f_s,g_s\in \Poly(2^\N)\) and \(t_s\in C(2^\N)\). Let \(B\) be a neighbourhood of \(t_s\) in \(C(2^\N)\). Then 
  \[t_s\in \bigcap_{i= 1} ^ {k} [K_i, U_i]\subseteq B,\]
   for some compact \(K_i\), open \(U_i\), and \(k\in \N\). Hence
  \[s=f_st_s\in  f_s \bigcap_{i = 1 } ^ k [K_i, U_i]\quad\text{and so}\quad
  \bigcap_{i = 1} ^ {k} [(K_i)f_s^{-1}, U_i] = f_s \bigcap_{i = 1} ^ {k} [K_i, U_i] \subseteq f_sB=f_sBg_s.\]
  Since each of $[(K_i)f_s^{-1}, U_i]$ is a subbasic open set, 
  \(f_sBg_s\) is a neighbourhood of $s$, as required.\medskip
  
 \noindent  \textbf{(\ref{AC-cantor-clone-ii}).} By part \textbf{(\ref{AC-cantor-clone-i})}, the semigroup \(S_{\Poly(2^\N)}\) has property \textbf{X} with respect to \(C(2^\N)\) and the compact-open topology. It is shown in~\cref{AC-cantor} that \(C(2^\N)\) has automatic continuity with respect to the class of second countable topological semigroups and the compact-open topology. Hence, by \cref{lem-luke-0}\eqref{lem-luke-0-iv}, and the fact \(S_{\Poly(2^\N)}\) has automatic continuity with respect to class of second countable topological semigroups. Therefore \cref{lem:lift-ac} implies that \(\Poly(2^\N)\) has automatic continuity, as required.\medskip
 
 \noindent  \textbf{(\ref{AC-cantor-clone-iii}).}
 Suppose that \(\Poly(2^\N)\) is equipped with a second countable Hausdorff clone topology $\T$. By part \textbf{(\ref{AC-cantor-clone-ii})}, this topology must be contained in the compact-open topology.
  It is shown in~\cref{theorem-compact-open-hilbert} that if 
 $C(X,Y)$ is first countable and
  Hausdorff with respect to some topology $\S$, for arbitrary compact metrizable spaces $X$ and $Y$, and $a : X \times C(X, Y) \to Y$ defined by
  \[
    (p, f)a = (p)f
  \]
  is continuous, then $\S$ contains the compact-open topology. 
Hence, by \cref{prop-cantor-action-clone}, $\T$ contains the compact-open topology.
\end{proof}

\begin{theorem}\label{AC-Boolean-clone}
The following hold:
\begin{enumerate}[\rm (i)]
    \item\label{AC-Boolean-clone-i} \(S_{\Poly(B_{\infty})}\) has property \textbf{X} with respect to \(\End(B_{\infty})\) and the pointwise topology;
    \item\label{AC-Boolean-clone-ii} \(\Poly(B_{\infty})\) has automatic continuity with respect to the pointwise topology and the class of second countable topological clones;
    \item\label{AC-Boolean-clone-iii} the pointwise topology is the only second countable Hausdorff clone topology on \(\Poly(B_{\infty})\).
\end{enumerate}
\end{theorem}
\begin{proof}
The proofs of parts \eqref{AC-Boolean-clone-i} and \eqref{AC-Boolean-clone-ii} of this theorem are similar to the corresponding parts of the proof of \cref{AC-cantor-clone}. \medskip

\noindent\textbf{(\ref{AC-Boolean-clone-i}).}
Let \(s\in \Poly(B_{\infty})\) be arbitrary and of arity \(n\).
Note that all finite non-zero powers of \(B_\infty\) are isomorphic to \(B_\infty\). This can be seen by noting that they are all atomless, or by noting that every finite non-empty disjoint union of \(2^\N\) with itself is homeomorphic to \(2^\N\).

  Let \(f_s:(B_{\infty})^{n} \to B_{\infty}\) be an isomorphism, \(t_s:= f_s^{-1}s\), and \(g_s:= 1_{B_{\infty}}\) (the identity function on $B_{\infty}$). Then \(f_st_sg_s=f_sf_s^{-1}sg_s=s\), \(f_s,g_s\in \Poly(B_{\infty})\) and \(t_s\in \End(B_{\infty})\).
  Let \(B\) be a neighbourhood of \(t_s\) in \(\End(B_{\infty})\). Then 
  \[t_s\in \bigcap_{i = 1} ^ {k} \{h\in \Hom(B_{\infty}^n, B_{\infty}): (a_i)h=b_i\}\subseteq B,\]
  for some \(a_i, b_i\in B_{\infty}\) and finite \(k\).  Since
  \[s\in \bigcap_{i = 1} ^ k \{h\in \Hom(B_{\infty}^n, B_{\infty}): ((a_i)f_s^{-1})h=b_i\}\subseteq f_sBg_s,\]
  it follows that $f_sBg_s$ is a neighbourhood of $s$. 
  \medskip

\noindent \textbf{(\ref{AC-Boolean-clone-ii}).}
The monoid $\End(B_{\infty})$ has automatic continuity with respect to the pointwise topology and the class of second countable topological semigroups; see~\cref{AC-Boolean}.
The proof of this part of the theorem then follows by a similar argument to \cref{AC-cantor-clone}\eqref{AC-cantor-clone-ii}. 
  \medskip

\noindent \textbf{(\ref{AC-Boolean-clone-iii}).}
  Suppose that \(\Poly(B_\infty)\) is equipped with a second countable Hausdorff clone topology $\T$. By part \eqref{AC-Boolean-clone-ii}, $\T$ is contained in the usual pointwise topology.

  It remains to show that $\T$ contains the pointwise topology. 
  Stone duality translates: the pointwise topology on $\Poly(B_\infty)$ to the compact-open topology on $\Poly(B_\infty) ^ {\dagger}$ (defined in \eqref{eqn-stone-dual-poly}); and the topology $\T$ to a second countable Hausdorff topology $\T ^ {\dagger}$ on $\Poly(B_\infty) ^ {\dagger}$ such that $\pentagram_{m, n}$ is continuous with respect to $\T ^ {\dagger}$ for all $m,n\in\N$.
  Hence it suffices to show that the topology $\T ^ {\dagger}$ on $\Poly(B_\infty) ^ {\dagger}$ obtained by translating $\T$ via Stone duality contains the compact-open topology. 
  In other words, it is enough to fix $n\in \N$ and show that the subspace topology on \(C(2^\N, 2^\N \times \{1, \ldots, n\})\) induced by $\T ^ {\dagger}$ contains the compact-open topology. As in the proof of \cref{AC-cantor-clone}\eqref{AC-cantor-clone-iii}, by~\cref{theorem-compact-open-hilbert}, it suffices to show that 
   \(a_n:2^\N \times C(2^\N,  2^\N \times \{1, \ldots, n\}) \to 2^\N \times \{1, \ldots, n\}\) defined by
  \[(p, f)a_n= (p)f\]
  is continuous.
 Recall that \(\pi_i: 2 ^ \N \times \{1, \ldots, n\} \to 2^\N\) denotes the $i$th projection and note that $(p,f)a_n=((p)a_n\pi_1,(p)a_n\pi_2)$ for all $p\in 2^{\N}$ and $f\in C(2^\N, 2^\N \times \{1, \ldots, n\})$. We conclude the proof by showing that $a_n\pi_1$ and $a_n\pi_2$ are continuous.
  
  By \cref{prop-phi-is-homeo},
  if $F$ denotes the set of constant maps of arity \(1\) in $C(2 ^ \N)$, and $\phi : F\to 2 ^ \N$ is the function that sends $g\in F$ to the unique value in $(2 ^ \N)g$, then \(\phi\) is a homeomorphism.
  Let \(1_{2^\N}\) be the identity function on \(2^\N\). If $p\in 2 ^ \N$ and $f \in C(2^\N, 2^\N \times \{1, \ldots, n\})$,   then \( (p)\phi^{-1}f\pi_1 \) is the constant map on \(2^\N\) with  value \( (p)f \pi_1\). Hence
  \[(p,f)a_n\pi_1 = (p)f\pi_1 = (p)\phi ^ {-1}f\pi_1 \phi =  (p)\phi^{-1} f \pi_11_{2^\N} \phi 
  = \left(((p)\phi^{-1}f) \pentagram_{n, 1} (1_{2^\N}, \ldots, 1_{2^\N})\right)\phi
  \]
  for all $p\in 2 ^ \N$ and $f\in C(2 ^ \N, 2 ^ \N \times \{1, \ldots, n\})$. 
  
    Suppose that $n\in \N$  and that 
  \[M_n :C(2^\N)\times C(2^\N, 2^\N \times \{1,\ldots, n\})\to C(2^\N, 2^\N \times \{1,\ldots, n\})\]
  is defined by  $(f, g)M_n = f \circ g$. Then $ (f, g)M_n = f \circ g = \overline{f}\pentagram_{1, n} (g)$ where $\overline{f}: 2 ^ \N \to 2 ^ \N \times \{1\}$ is defined by $(p)\overline{f}= ((p)f, 1)$, and so the functions $M_n$ are continuous with respect to \(\mathcal{T}^{\dagger}\).
  Since \(\phi^{-1}\) and \(M_n\) are continuous, the maps \( (p, f) \mapsto (p)\phi^{-1}f\) and
  \( (p, f) \mapsto (1_{2^\N}, \ldots, 1_{2^\N})\) are continuous. Moreover, \(\pentagram_{n, 1}\) and \(\phi\) are continuous, and so \(a_n \pi_1\) is continuous also.
  
  Let \(q:2^\N \to \{1,\ldots, n\}\) be continuous and surjective, and  let \(c_{m}: 2 ^ \N \to 2 ^ \N\) be any constant map with value \(x_m \in (m)q^{-1}\) for all \(m\in \{1,\ldots, n\}\). Then
  \[ (p,f)a_n \pi_2 = (p)f\pi_2 = x_{(p)f \pi_2}q
  = c_{(p)f\pi_2} \phi q 
  = \left((p)\phi^{-1}f \pentagram_{n,1} (c_{1}, \ldots, c_{n})\right)\phi q\]
  for all $p\in 2^\N $ and \(f \in C(2^\N, 2^\N \times \{1, \ldots, n\})\).
   The proof that $a_n\pi_2$ is continuous is then similar to the proof that $a_n\pi_1$ is continuous.
\end{proof}

\bibliography{sim}{}
\bibliographystyle{hplain}

\end{document}